\documentclass{article}

\usepackage{amsmath, amssymb}
\usepackage{fullpage}
\usepackage{amsthm, mathrsfs, color}
\usepackage{dsfont}
\usepackage{hyperref}
\usepackage{slashed}
\usepackage{graphicx}
\usepackage[affil-it]{authblk}

\newtheorem{theorem}{Theorem}[section]

\newtheorem{proposition}[theorem]{Proposition}

\newtheorem{remark}[theorem]{Remark}

\numberwithin{equation}{section}

\newcommand{\supp}{\mathrm{supp}}

\newcommand{\gs}{\mathbb{X}}

\begin{document}

\title{Inverse modified scattering and polyhomogeneous expansions for the Vlasov--Poisson system}

\author[$*$ $\ddag$]{Volker Schlue}

\author[$\dag$]{Martin Taylor}

\affil[$\dag$]{\small Imperial College London,
Department of Mathematics,
South~Kensington~Campus,~London~SW7~2AZ,~United~Kingdom\vskip.2pc martin.taylor@imperial.ac.uk\vskip.2pc \ }

\affil[$*$]{\small University of Melbourne, School of Mathematics and Statistics, Parkville, VIC, 3010, Australia\vskip.2pc  volker.schlue@unimelb.edu.au\vskip.2pc \  }

\affil[$\ddag$]{\small Universit\"at M\"unster,
Mathematisches~Institut, Einsteinstrasse~62,~48149~M\"unster,~Germany\vskip.2pc  volker.schlue@uni-muenster.de\vskip.2pc \  }

\date{April 24, 2024}

\maketitle

\begin{abstract}
	We give a new proof of well posedness of the inverse modified scattering problem for the Vlasov--Poisson system: for every suitable scattering profile there exists a solution of Vlasov--Poisson which disperses and scatters, in a modified sense, to this profile.  Further, as a consequence of the proof, the solutions are shown to admit a polyhomogeneous expansion, to any finite but arbitrarily high order, with coefficients given explicitly in terms of the scattering profile.  The proof does not exploit the full ellipticity of the Poisson equation.
\end{abstract}

\tableofcontents

\section{Introduction}

The Vlasov--Poisson system describes the evolution of an ensemble of collisionless particles, interacting via a collectively generated gravitational or electrostatic potential force.  The system on $\mathbb{R}^3$, in the electrostatic case, takes the form
\begin{equation} \label{eq:VP1}
	\gs_{\phi} f = 0,
\end{equation}
\begin{equation} \label{eq:VP2}
	\Delta_{\mathbb{R}^3} \phi (t,x) = \varrho(t,x),
	\qquad
	\varrho(t,x) = \int_{\mathbb{R}^3} f(t,x,p) dp,
\end{equation}
where $f:I \times \mathbb{R}^{3} \times \mathbb{R}^3 \to [0,\infty)$, and $\phi, \varrho \colon I \times \mathbb{R}^{3} \to \mathbb{R}$, for a suitable interval $I \subset \mathbb{R}$, and the operator $\gs_{\phi}$ is defined by,
\[
	\gs_{\phi} = \partial_t + p^i \partial_{x^i} + \partial_{x^i} \phi (t,x) \partial_{p^i}.
\]

This article concerns the inverse scattering problem for the system \eqref{eq:VP1}--\eqref{eq:VP2}.
  The first main result is a new proof of the existence of future-global solutions which scatter in the modified sense to a given asymptotic profile.

\begin{theorem}[Inverse modified scattering map] \label{thm:intromain}
	For any smooth compactly supported function $f_{\infty} \colon \mathbb{R}^3 \times \mathbb{R}^3 \to [0,\infty)$, there exists $T_0 > 0$ and a smooth solution $f \colon [T_0,\infty) \times \mathbb{R}^3 \times \mathbb{R}^3 \to [0,\infty)$ of the Vlasov--Poisson system such that, for all $x,p\in \mathbb{R}^3$,
	\begin{equation} \label{eq:intromainconvergence}
		\lim_{t\to \infty} f\big(t,x+tp - \log t \, \nabla \phi_{\infty} ( p), p \big)
		=
		f_{\infty}(x,p),
	\end{equation}
	where $\phi_{\infty}$ is the unique solution of
	\begin{equation} \label{eq:intromainphiinfty}
		\Delta \phi_{\infty} (p) = \varrho_{\infty}(p),
		\qquad
		\phi_{\infty} (p) \to 0, \quad \text{as } \vert p \vert \to \infty,
		\qquad
		\varrho_{\infty}(p) = - \int_{\mathbb{R}^3} f_{\infty} (x,p) dx.
	\end{equation}
\end{theorem}

Theorem \ref{thm:intromain} was first obtained by Flynn--Ouyang--Pausader--Widmayer \cite{FOPW}, using a proof based on a certain \emph{pseudo-conformal inversion} of $\mathbb{R}_t \times \mathbb{R}_x^3 \times \mathbb{R}_p^3$.  Theorem~\ref{thm:intromain} has also recently been generalised to the Vlasov--Maxwell system independently, with a different method, by Bigorne \cite{Big23}.

\begin{remark}[Modified scattering]
  For free transport (namely if $f$ solves \eqref{eq:VP1} with $\phi \equiv0$), 
  the quantity $	f(t,x+t p,p)$ is independent of $t$.  We would say that a solution $f$ of Vlasov--Poisson \emph{scatters to free transport} if this quantity  was to converge to a function independent of $t$ as $t \to \infty$.
  We say that the solutions of Theorem \ref{thm:intromain} scatter \emph{in a modified sense} in view of the presence of the logarithmic correction involving $\nabla \phi_{\infty}$ in \eqref{eq:intromainconvergence}.  See Figure \ref{fig:log:traject} for a comparison of the trajectories $t\mapsto x+tp - \log t \, \nabla \phi_{\infty} ( p)$ with the free trajectories $t \mapsto x+t p$.
\end{remark}

\begin{figure}

  \centering

  \includegraphics[scale=0.8]{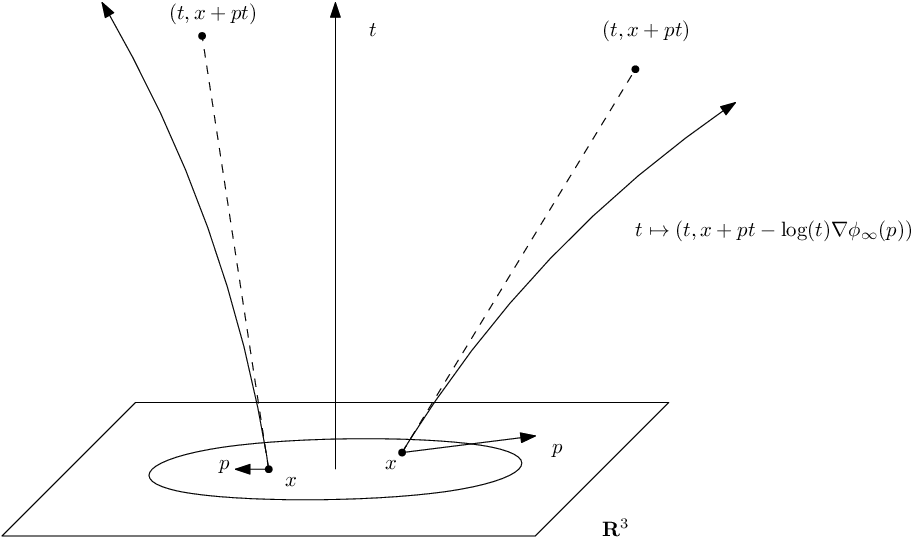}

  \caption{Trajectories of particles with initial position $x$ and momentum $p$ in the potential $\phi_\infty$.}

  \label{fig:log:traject}

\end{figure}

  Our proof of Theorem~\ref{thm:intromain} proceeds by constructing suitable \emph{approximate solutions} to \eqref{eq:VP1}--\eqref{eq:VP2}.
These are polyhomogeneous expansions of the form
\begin{equation} \label{eq:introansatz1}
	f_{[K]}(t,x,p)
	=
	\sum_{k=0}^K \sum_{l=0}^{k} \frac{(\log t)^l}{t^k} f_{k,l}\big(x-tp + \log t \, \nabla \phi_{\infty} ( p), p \big),
\end{equation}
\begin{equation} \label{eq:introansatz2}
	\phi_{[K]}(t,x)
	=
	\frac{1}{t}
	\sum_{k=0}^K \sum_{l=0}^{k} \frac{(\log t)^l}{t^k} \phi_{k,l} \left( \frac{x}{t} \right),
	\quad
	\varrho_{[K]}(t,x)
	=
	\frac{1}{t^3}
	\sum_{k=0}^K \sum_{l=0}^{k} \frac{(\log t)^l}{t^k} \varrho_{k,l} \left( \frac{x}{t} \right),
\end{equation}
  which are determined explicitly from the scattering profile $f_{\infty}$ alone.
  In fact, we show that these approximate solutions provide a detailed description the asymptotic behaviour of the solutions  in Theorem \ref{thm:intromain}:

\begin{theorem}[Polyhomogeneous expansion of modified scattering solutions] \label{thm:intromain2}
	For any smooth compactly supported function $f_{\infty} \colon \mathbb{R}^3 \times \mathbb{R}^3 \to [0,\infty)$, there are sequences of smooth compactly supported functions
	\[
		f_{k,l} \colon \mathbb{R}^3 \times \mathbb{R}^3 \to \mathbb{R},
		\qquad
		\phi_{k,l}, \varrho_{k,l} \colon \mathbb{R}^3 \to \mathbb{R},
		\qquad
		\text{for }
		k = 0,1,2,\ldots, \text{ and } l=0,\ldots k,
	\]
	defined explicitly in terms of $f_{\infty}$, with
	\[
		f_{0,0} = f_{\infty}, \qquad \varrho_{0,0} = \varrho_{\infty}, \qquad \phi_{0,0} = \phi_{\infty},
	\]
	such that, for any $K \in \mathbb{N}$, there exists $T_0 > 0$ such that the solution $(f,\varrho,\phi)$ of Theorem \ref{thm:intromain} satisfies, for all $(t,x,p) \in [T_0,\infty) \times \mathbb{R}^3 \times \mathbb{R}^3$,
	\begin{equation} \label{eq:fintromain}
		\Big\vert
		f(t,x,p)
		- f_{[K]}(t,x,p)
		\Big\vert
		\leq
		C_K \mathcal{F}_K \left( \frac{\log t}{t} \right)^{K+1},
	\end{equation}
	where $f_{[K]}$ is given by \eqref{eq:introansatz1}, and with $\phi_{[K]}$ and $\rho_{[K]}$ given by \eqref{eq:introansatz2},
	\begin{align}
		\Big\vert
		\nabla \phi(t,x)
		-
		\nabla \phi_{[K]}(t,x)
		\Big\vert
		&
		\leq
		C_K \mathcal{F}_K \frac{(\log t)^{K+1}}{t^{K+3}},
		\label{eq:phiintromain}
		\\
		\Big\vert
		\varrho(t,x)
		-
		\varrho_{[K]}(t,x)
		\Big\vert
		&
		\leq
		C_K \mathcal{F}_K \frac{(\log t)^{K+1}}{t^{K+4}},
		\label{eq:rhointromain}
	\end{align}
      	where $C_K>1$ is constant depending on $K$, and $\mathcal{F}_K$ is defined in terms of the $L^2$ norm of $f_{\infty}$, along with the $L^2$ norm of a large number (depending on $K$) of derivatives of $f_{\infty}$.  Similar statements to \eqref{eq:fintromain}--\eqref{eq:rhointromain} hold for, appropriately weighted, higher order derivatives.
\end{theorem}

Polyhomogeneous expansions of the form \eqref{eq:fintromain}--\eqref{eq:rhointromain} have also very recently been obtained independently, in a closely related setting, by Bigorne--Velozo Ruiz \cite{BiRu}.  See Theorem \ref{thm:BiRu} below.

  While the explicit expressions for $\varrho_{k,l}$, $\phi_{k,l}$ and $f_{k,l}$ are given in the proof of Theorem \ref{thm:approx} in Section \ref{subsec:thmapprox} below,
 we will discuss the constructions for $K=0$ and $K=1$  in detail in Section \ref{subsec:proofoverview}.
 The proof of Theorem \ref{thm:intromain} is based on the existence of expansions of the form \eqref{eq:fintromain}--\eqref{eq:rhointromain}, and thus the proof of both Theorem \ref{thm:intromain} and Theorem \ref{thm:intromain2} are established together.

  A similar approach, of using approximate solutions to finite order, has been used to obtain inverse scattering results for nonlinear wave equations in \cite{LS23}. \emph{Homogeneous} asymptotics as in \eqref{eq:introansatz2}, and their implications for the scattering problem for wave equations have recently been considered in \cite{LS24}.
  For another recent example in general relativity, of the broad strategy of using sequences of increasingly better approximate solutions to prove existence of solutions of an inverse scattering-type problem, see \cite{FoLu}.

  \begin{remark}[Motivation: elliptic and hyperbolic field equations] \label{rmk:ellipticity}
    A motivation to give a new proof of Theorem~\ref{thm:intromain} in this paper is to illustrate a method which can be applied to other collisionless kinetic equations, such as the Vlasov--Maxwell and Einstein--Vlasov systems.
In contrast to the  \emph{ellipticity} of the Poisson equation,
the field equations in each of these two examples are, appropriately interpreted, \emph{hyperbolic}.  As such, we have elected not to utilise the full ellipticity of the Poisson equation in the proof of Theorem \ref{thm:intromain}.  In concrete terms, this means that the estimate \eqref{eq:gradelliptic} below --- which has appropriate analogues
for hyperbolic operators --- is used, but not the estimate \eqref{eq:fullelliptic}.  We have also chosen, for this reason, to base the proof of Theorem \ref{thm:intromain} entirely on $L^2$ estimates.
\end{remark}

\begin{remark}[Simpler proof for small scattering data]
  The proof given in this paper simplifies considerably if the scattering profile $f_{\infty}$ is further assumed to be suitably small.
In this case, an approximation to order $K=0$ suffices; see discussion in Section~\ref{subsec:introexistence} below.
\end{remark}

The final main result concerns the uniqueness of the solutions of Theorem \ref{thm:intromain}, which is shown to hold in the class of solutions agreeing with the polyhomogeneous expansions of Theorem \ref{thm:intromain2} to sufficiently high order.

\begin{theorem}[Uniqueness of modified scattering solutions] \label{thm:intromain3}
	For any smooth compactly supported function $f_{\infty} \colon \mathbb{R}^3 \times \mathbb{R}^3 \to [0,\infty)$, the solution $(f,\varrho,\phi)$ of Theorem \ref{thm:intromain} is unique in the class of solutions which satisfy an expansion of the form \eqref{eq:fintromain} to sufficiently high order.
\end{theorem}

The work \cite{FOPW} also provides a uniqueness statement as in Theorem \ref{thm:intromain3}, but without requiring the restriction to the class of solutions which a priori admit expansions of the form \eqref{eq:fintromain}.

For a more precise statement of Theorem \ref{thm:intromain}, Theorem \ref{thm:intromain2}, and Theorem \ref{thm:intromain3}, see Section \ref{section:theorem} below.

\subsection{Scattering results and the small data regime}
\label{subsec:scattering}

  It is well known that solutions of Vlasov--Poisson \eqref{eq:VP1}--\eqref{eq:VP2}, for which $f\vert_{t=0}$ is appropriately small, scatter in modified sense, meaning that there exists a function $f_{\infty} \colon \mathbb{R}^3 \times \mathbb{R}^3 \to [0,\infty)$ such that \eqref{eq:intromainconvergence} holds.
  This insight is due to Choi--Kwon \cite{ChKw},   and an alternative proof was later given by Ionescu--Pausader--Wang--Widmayer~\cite{IPWW}.

\begin{theorem}[Modified scattering for solutions with small Cauchy data \cite{ChKw, IPWW}] \label{thm:forward}
	For any function $f_C \colon \mathbb{R}_x \times \mathbb{R}_p \to [0,\infty)$ which is suitably regular and small, in an appropriate norm, there exists a solution $(f,\varrho,\phi)$ which exists globally in time, attains the Cauchy data, $f(0,\cdot,\cdot) \equiv f_C$, disperses as $t \to \infty$ and scatters in a modified sense.  More precisely, there exists a function $f_{\infty} \colon \mathbb{R}_x^3 \times \mathbb{R}_p^3 \to [0,\infty)$ such that \eqref{eq:intromainconvergence} holds for all $x,p\in \mathbb{R}^3$.
\end{theorem}

The part of Theorem \ref{thm:forward} concerning global existence and dispersion is a classical result of Bardos--Degond \cite{BaDe}, for which there now exists a number of alternative proofs and improvements \cite{Dua,HRV,Smu,Wa}, and has been generalised to certain large dispersive solutions \cite{Sch}.  Theorem \ref{thm:forward} has also been generalised to higher dimensions by Pankavich \cite{Pan} (see also Remark \ref{rmk:higherdimensions} below), and to the Vlasov--Maxwell system by Bigorne \cite{Big22} and Pankavich--Ben-Artzi \cite{PaBe23}.  Furthermore the very recent work of Bigorne--Velozo Ruiz \cite{BiRu}, completed independently, provides a detailed description of the asymptotic behaviour of the solutions of Theorem \ref{thm:forward}.

\begin{theorem}[Polyhomogeneous expansion for solutions with small Cauchy data \cite{BiRu}] \label{thm:BiRu}
	For any function $f_C \colon \mathbb{R}_x^3 \times \mathbb{R}_p^3 \to [0,\infty)$ which is suitably regular and small, in an appropriate norm, there exist functions $f_{k,l} \colon \mathbb{R}^3 \times \mathbb{R}^3 \to \mathbb{R}$, $\phi_{k,l}, \varrho_{k,l} \colon \mathbb{R}^3 \to \mathbb{R}$ for $k = 0,1,2,\ldots$, and $l=0,\ldots k$, such that the solutions of Theorem \ref{thm:forward} admit polyhomogeneous expansions of the form \eqref{eq:fintromain}--\eqref{eq:rhointromain}.
\end{theorem}

A scattering theory for the system \eqref{eq:VP1}--\eqref{eq:VP2} may be formulated in terms of the following questions:

\begin{enumerate}
	\item[(i)]
		\emph{Existence of the forward modified scattering operator:} For a given time $T_0 \in \mathbb{R}$, for which classes of Cauchy data $f_C = f\vert_{t=T_0}$ does there exist a corresponding scattering state $f_{\infty}$?  What can be said about the forward scattering map $\mathcal{S}_F \colon f_C \mapsto f_{\infty}$?
	\item[(ii)]
		\emph{Uniqueness of forward modified scattering states:} Do solutions giving rise to the same scattering state coincide, i.\@e.\@, when acting on appropriate spaces, is the forward scattering map $\mathcal{S}_F \colon f_C \mapsto f_{\infty}$ injective?
	\item[(iii)]
		\emph{Asymptotic completeness of the forward modified scattering map:}  For every $f_{\infty}$, does there exist corresponding Cauchy data $f_C$, i.\@e.\@, when acting on appropriate spaces, is the map $\mathcal{S}_F \colon f_C \mapsto f_{\infty}$ surjective?
\end{enumerate}

 Theorem~\ref{thm:intromain}  addresses question (iii),  and Theorem \ref{thm:intromain3} concerns question~(ii).
    In the \emph{small data regime}, Theorem~\ref{thm:forward}    gives  a resolution of question~(i).
    Despite global existence results for general classes of Cauchy data \cite{LiPe, Pfa, Sch91}, question (i) remains open in general.

Due to the time reversibility of the system \eqref{eq:VP1}--\eqref{eq:VP2}, a resolution of questions (i)--(iii) also provides a resolution of the corresponding questions for the \emph{backwards modified scattering operator} $\mathcal{S}_B \colon f_C \mapsto f_{-\infty}$.  Such an $f_{-\infty} \colon (-\infty,-T_0] \times \mathbb{R}^3 \times \mathbb{R}^3 \to [0,\infty)$ satisfies, for all $x,p\in \mathbb{R}^3$,
\begin{equation} \label{eq:intromainconvergencepast}
	\lim_{t\to -\infty} f\big(t,x+tp - \log t \, \nabla \phi_{-\infty} ( p), p \big)
	=
	f_{-\infty}(x,p),
\end{equation}
where $f \colon (-\infty,-T_0] \times \mathbb{R}^3 \times \mathbb{R}^3 \to [0,\infty)$ is the unique solution of the Vlasov--Poisson system attaining the Cauchy data $f_C$, and $\phi_{-\infty}$ is the unique solution of \eqref{eq:intromainphiinfty}, after replacing $f_{\infty}$ with $f_{-\infty}$.

\begin{remark}[Time reversed analogue of Theorem~\ref{thm:intromain}]\label{rmk:reversed}
    The time reversed analogue of Theorem~\ref{thm:intromain} states that for any smooth compactly supported function $f_{-\infty} \colon \mathbb{R}^3 \times \mathbb{R}^3 \to [0,\infty)$, there exists $T_0 > 0$ and a solution $f \colon (-\infty,-T_0] \times \mathbb{R}^3 \times \mathbb{R}^3 \to [0,\infty)$ of the Vlasov--Poisson system such that, for all $x,p\in \mathbb{R}^3$, \eqref{eq:intromainconvergencepast} holds.
    Similarly, expansions of the form \eqref{eq:fintromain}--\eqref{eq:rhointromain} hold for $(t,x,p) \in (-\infty,-T_0] \times \mathbb{R}^3 \times \mathbb{R}^3$.
	\end{remark}

One can also ask analogous questions of the \emph{full modified scattering operator} $\mathcal{S} \colon f_{-\infty} \mapsto f_{\infty}$:
\begin{enumerate}
	\item[(iv)]
		\emph{Existence of the full modified scattering operator:} For which classes of past scattering states $f_{-\infty}$ does there exist a corresponding future scattering state $f_{\infty}$?  What can be said about the full scattering map $\mathcal{S} \colon f_{-\infty} \mapsto f_{\infty}$?
	\item[(v)]
		\emph{Uniqueness of full modified scattering states:} Is the full scattering map, when acting on appropriate spaces, $\mathcal{S} \colon f_{-\infty} \mapsto f_{\infty}$ injective?
	\item[(vi)]
		\emph{Asymptotic completeness of the full modified scattering map:}  Is the full scattering map, when acting on appropriate spaces, $\mathcal{S} \colon f_{-\infty} \mapsto f_{\infty}$ surjective?
\end{enumerate}

In the small data regime, affirmative answers of questions (iv)--(vi)  are due to Flynn--Ouyang--Pausader--Widmayer~\cite{FOPW}.

Alternatively,
in view of the time reversibility, Theorem \ref{thm:intromain}  can also be combined directly with Theorem~\ref{thm:forward} to answer questions (iv)--(vi) in the affirmative, under the additional assumption that $f_{-\infty}$ is small:
For any smooth compactly supported function $f_{-\infty} \colon \mathbb{R}^3 \times \mathbb{R}^3 \to [0,\infty)$, which is small in an appropriate sense, there exists a smooth compactly supported function $f_{\infty} \colon \mathbb{R}^3 \times \mathbb{R}^3 \to [0,\infty)$ and a smooth solution $f \colon \mathbb{R} \times \mathbb{R}^3 \times \mathbb{R}^3 \to [0,\infty)$ of the Vlasov--Poisson system \eqref{eq:VP1}--\eqref{eq:VP2} such that both \eqref{eq:intromainconvergence} and \eqref{eq:intromainconvergencepast} hold.  Moreover, the function $f_{\infty}$ is quantitatively controlled by $f_{-\infty}$.

Finally, we remark that there have been previous ``scattering constructions'' for Vlasov--Poisson where the spatial domain is the torus \cite{CaMa, HwVe}.

\subsection{Further remarks on the main results}

  Before we turn to the proof of Theorem \ref{thm:intromain}, Theorem \ref{thm:intromain2}, and Theorem \ref{thm:intromain3}, several further remarks are given:

\begin{remark}[Repulsive sign not relevant]
	The proof of Theorem \ref{thm:intromain} makes no essential use of the \emph{repulsive} sign of the nonlinearity in \eqref{eq:VP1}--\eqref{eq:VP2} and thus equally applies, with the obvious modifications to the statements, to the \emph{gravitational Vlasov--Poisson system}, i.\@e.\@ the equation
	\[
		\partial_t f + p^i \partial_{x^i} f - \partial_{x^i} \phi \, \partial_{p^i} f = 0,
	\]
	coupled to equation \eqref{eq:VP2}.  Attention is restricted here to \eqref{eq:VP1}--\eqref{eq:VP2} in order to ease notation.
\end{remark}

\begin{remark}[Convergence to scattering data]
	The convergence \eqref{eq:intromainconvergence} is shown to in fact hold in a much stronger sense, with a quantitative rate. In particular, see already the estimate \eqref{eq:fintromain} with $K=0$.  See also equation \eqref{eq:dataattained} in Theorem \ref{thm:backwardsproblem} for a stronger statement. 
\end{remark}

\begin{remark}[Finite regularity]
	The function $f_{\infty}$ in Theorem \ref{thm:intromain} is assumed, for simplicity, to be smooth.  The proof also holds, however, for far rougher $f_{\infty}$ (for example for $f_{\infty}$ lying in a suitable Sobolev space).  No attempt has been made to optimise the regularity requirements here.
\end{remark}

\begin{remark}[Higher dimensions] \label{rmk:higherdimensions}
	The proof of Theorem \ref{thm:intromain} can be adapted to the case of $d+1$ dimensions, for $d \geq 4$, to obtain solutions of Vlasov--Poisson which scatter to free transport, i.\@e.\@ which satisfy
	\begin{equation} \label{eq:fscatter}
		\lim_{t\to \infty} f(t,x+tp,p) = f_{\infty}(x,p),
	\end{equation}
	for all $x,p\in \mathbb{R}^3$.  The proof, in fact, simplifies considerably, and the $3+1$ dimensional case should be viewed as critical, in a certain sense.  The borderline $t^{-1}$ behaviour discussed in Section \ref{subsec:introexistence} below is replaced by integrable $t^{2-d}$ behaviour (meaning that, for $\check{f}(t,x,p) = f(t,x,p) - f_{\infty}(x-tp,p)$, inequality \eqref{eq:introEtinequ} becomes $\check{\mathcal{E}}(t) \leq \mathcal{F} \int_t^{T_f} s^{-d+2} \check{\mathcal{E}}(s) ds + \mathcal{F}t^{-d+3}$), obviating the need for the approximate solutions \eqref{eq:introansatz1}--\eqref{eq:introansatz2} in the proof.  Recall that the analogue of Theorem \ref{thm:forward} in higher dimensions, where the solutions satisfy \eqref{eq:fscatter}, is due to Pankavich \cite{Pan}.
	
	Theorem \ref{thm:intromain2} also generalises to higher dimensions to give expansions of the form
	\[
		f(t,x,p) = \sum_{k=0}^K \frac{1}{t^k} f_{k}(x-tp, p ) + \mathcal{O} \bigg(\frac{1}{t^{1+K}} \bigg),
	\]
	\[
		\varrho(t,x)
		=
		\frac{1}{t^d}
		\sum_{k=0}^K \frac{1}{t^k} \varrho_{k} \left( \frac{x}{t} \right)
		+
		\mathcal{O} \bigg(\frac{1}{t^{d+K+1}} \bigg),
		\qquad
		\nabla \phi(t,x)
		=
		\frac{1}{t^{d-1}}
		\sum_{k=0}^K \frac{1}{t^k} (\nabla \phi_{k}) \left( \frac{x}{t} \right)
		+
		\mathcal{O} \bigg(\frac{1}{t^{d+K}} \bigg).
	\]
\end{remark}

\begin{remark}[$\nabla \phi_{\infty} ( p)$ vs $\nabla \phi_{\infty} (\frac{x}{t})$]
	As will be seen in the proof, in the context of Theorem \ref{thm:intromain}, $p$ and $\frac{x}{t}$ are comparable in the support of $f$ and thus, in the argument of the elements of the expansion \eqref{eq:fintromain}, $x-tp + \log t \, \nabla \phi_{\infty} ( p)$ can be replaced with $x-tp + \log t \, \nabla \phi_{\infty} (\frac{x}{t})$, at the expense of further terms (involving further $\log t$ corrections) in the summation \eqref{eq:fintromain}.
	
	Indeed, one sees that further terms are necessary in such an expansion in view of the following fact.  Since $f_{\infty}$, and hence $\phi_{\infty}$, is smooth and, as will be seen in the proof of Theorem \ref{thm:intromain}, $x-tp + \log t \, \nabla \phi_{\infty} ( p)$ is uniformly bounded in the support of the solution, it follows that
	\begin{multline*}
		f_{\infty}\Big(x-tp + \log t \, \nabla \phi_{\infty} \Big( \frac{x}{t} \Big), p \Big)
		=
		f_{\infty}\big(x-tp + \log t \, \nabla \phi_{\infty} ( p), p \big)
		\\
		+
		\frac{(\log t)^2}{t}
		\partial_i \phi_{\infty}(p) \partial_i \partial_j \phi_{\infty}(p) 
		(\partial_{x^j} f_{\infty})\big(x-tp + \log t \, \nabla \phi_{\infty} ( p), p \big)
		+
		\mathcal{O} \left( \frac{\log t}{t} \right).
	\end{multline*}
	Thus, if $x-tp + \log t \, \nabla \phi_{\infty} ( p)$ is replaced with $x-tp + \log t \, \nabla \phi_{\infty} (\frac{x}{t})$ in the argument of the elements of the expansion, it is necessary to include a $\frac{(\log t)^2}{t}$ term in the expansion (which is not present in \eqref{eq:fintromain}).
\end{remark}

\begin{remark}[Convergence of the series]
	Under alternative assumptions on $f_{\infty}$, such as analyticity, one may hope to show that the infinite series
	\[
		\sum_{k=0}^\infty \sum_{l=0}^{k} \frac{(\log t)^l}{t^k} f_{k,l}\big(x-tp + \log t \, \nabla \phi_{\infty} ( p), p \big),
	\]
	converges to the solution of Vlasov--Poisson of Theorem \ref{thm:intromain}.  Such convergence is not considered here.
\end{remark}

\subsection{Overview of the proof}
\label{subsec:proofoverview}

In this section an overview of the proof of Theorem \ref{thm:intromain} is given.  As noted above, the proof of Theorem \ref{thm:intromain2} is established at the same time.  The approximate solutions of Theorem \ref{thm:intromain2} are discussed in Section \ref{subsec:introexpansion}, and in Section \ref{subsec:introexistence} it is outlined how the approximate solutions are used to prove the existence of a solution of attaining the scattering data $f_{\infty}$, which moreover agrees with the approximate solutions to increasing order.

\subsubsection{The proof of the main results --- the approximate solutions} \label{subsec:introexpansion}
In the logic of the proof of Theorem \ref{thm:intromain}, the first step is to give explicit definitions of the functions $f_{k,l}$, $\phi_{k,l}$, $\varrho_{k,l}$ in terms of $f_{\infty}$.  Each $\phi_{k,l}$ is defined in terms of the corresponding $\varrho_{k,l}$ as the unique solution of the Poisson equation
\begin{equation} \label{eq:introPoissonapprox}
	\Delta_{\mathbb{R}^3} \phi_{k,l} = \varrho_{k,l},
	\qquad
	\phi_{k,l}(w) \to 0, \quad \text{as } \vert w \vert \to \infty,
\end{equation}
for which there is a well known representation formula for solutions, and so it remains to define $f_{k,l}$ and $\varrho_{k,l}$.

Given such functions define $f_{[K]}$, $\varrho_{[K]}$, and $\phi_{[K]}$, for each $K \geq 0$, by \eqref{eq:introansatz1}--\eqref{eq:introansatz2}.
The functions $(f_{[K]}, \varrho_{[K]}, \phi_{[K]})$ are said to be an \emph{approximate solution of order $K$} if, for each $x,p\in\mathbb{R}^3$,
\begin{equation} \label{eq:introapproxsolutions}
	\gs_{\phi_{[K]}} f_{[K]} (t,x,p)
	=
	\mathcal{O} \bigg(\frac{(\log t)^{1+K}}{t^{2+K}} \bigg),
	\qquad
	\int_{\mathbb{R}^3} f_{[K]} (t,x,p) dp - \varrho_{[K]} (t,x) 
	=
	\mathcal{O} \bigg( \frac{(\log t)^K}{t^{4+K}} \bigg).
\end{equation}

The functions $f_{k,l}$, and $\varrho_{k,l}$ are defined simply by inserting the expressions \eqref{eq:introansatz1}--\eqref{eq:introansatz2} into the Vlasov--Poisson system \eqref{eq:VP1}--\eqref{eq:VP2} as an ansatz and solving order by order so that $(f_{[K]}, \varrho_{[K]}, \phi_{[K]})$ is an approximate solution of order $K$.  Remarkably, after inserting these expressions, though $\varrho_{k,l}$ appears in the expression for $f_{k,l}$, the quantity $f_{k,l}$ does not appear in the expression for $\varrho_{k,l}$ and so each quantity can be explicitly defined.  In order to illustrate the procedure, the cases $K=0$ and $K=1$ are presented here explicitly.  The full details are given in Section \ref{section:expansions}.  See, in particular, Theorem \ref{thm:approx}.

The fact that each approximate solution $(f_{[K]}, \phi_{[K]}, \varrho_{[K]})$ agrees with the true solution to order $K$ is only later shown, as part of the proof of the existence of the true solution (see the discussion in Section \ref{subsec:introexistence} below).

In order to simplify expressions, the notation
\begin{equation} \label{eq:introdefofy}
	y = y(t,x,p) = x-tp + \log t \, \nabla \phi_{\infty} ( p),
\end{equation}
is used.  Since $f_{\infty}$ is compactly supported, the function $y$ is uniformly bounded by a constant $M>0$ in the support of $f$ and in the support of $f_{[K]}$ for all $K \geq 0$.  In particular, it follows that, if $T_0$ is suitably large,
\begin{equation} \label{eq:introybound}
	\vert y(t,x,p) \vert \leq M,
	\qquad
	\vert x \vert \leq 3M t,
	\qquad
	\vert p \vert \leq 2M,
	\qquad
	\text{in } \supp(f) \text{ and } \supp(f_{[K]}).
\end{equation}
For fixed $p$ and $y$, the trajectory $t\mapsto y + tp - \log t \nabla \phi_{\infty}$ is as depicted in Figure \ref{fig:log:traject}.  Fixing $y$ and $p$ can be viewed as identifying a particle trajectory in phase space.

It follows from \eqref{eq:introybound} that, for each $t,x$, the $p$-supports, $\supp(f(t,x,\cdot))$ and $\supp(f_{[K]}(t,x,\cdot))$, are contained inside a ball centred at $\frac{x}{t} + \frac{\log t}{t} \phi_{\infty}(\frac{x}{t})$ with a radius of order $t^{-1}$ (cf.\@ \eqref{eq:intronablaphicancel} below) and so
\begin{equation} \label{eq:introindicator}
	\int_{\mathbb{R}^3} \mathds{1}_{\supp(f)} (t,x,p)
	+
	\int_{\mathbb{R}^3} \mathds{1}_{\supp(f_{[K]})} (t,x,p) dp \lesssim \frac{1}{t^3}.
\end{equation}
It also follows that the right hand sides in \eqref{eq:introapproxsolutions} also satisfy the support properties \eqref{eq:introybound}.

For functions $h\colon \mathbb{R}^3 \times \mathbb{R}^3 \to \mathbb{R}$, such as the functions $f_{k,l}$, for $i=1,2,3$, the notation $\partial_{x^i} h$ is always used to denote the derivative with respect to the $i$-th component, and the notation $\partial_{p^i} h$ is used to denote the derivative with respect to the $i+3$-th component.

\subsubsection*{The case $K=0$}
It is convenient to first consider the case $K=0$.  Defining
\[
	f_{0,0}(y,p)
	=
	f_{\infty}\big(y, p \big),
	\qquad
	\varrho_{0,0}(w)
	=
	\varrho_{\infty} \left( w \right),
\]
so that
\[
	f_{[0]}(t,x,p)
	=
	f_{\infty}\big(y(t,x,p), p \big),
	\qquad
	\phi_{[0]}(t,x)
	=
	\frac{1}{t}
	\phi_{\infty} \left( \frac{x}{t} \right),
	\qquad
	\varrho_{[0]}(t,x)
	=
	\frac{1}{t^3}
	\varrho_{\infty} \left( \frac{x}{t} \right),
\]
with $y(t,x,p)$ defined by \eqref{eq:introdefofy}, we check that \eqref{eq:introapproxsolutions} holds for $K=0$.  One computes
\begin{align} \label{eq:introphi0f0}
	\gs_{\phi_{[0]}} f_{[0]}(t,x,p)
	=
	\
	&
	\frac{1}{t} \Big( \partial_i \phi_{\infty}(p) - \partial_i \phi_{\infty} \Big(\frac{x}{t}\Big) \Big) (\partial_{x^i} f_{\infty})(y,p)
	\\
	&
	+
	\frac{\log t}{t^2} \partial_j \phi_{\infty} \Big(\frac{x}{t}\Big) \partial_i \partial_j \phi_{\infty}(p) (\partial_{x^i} f_{\infty})(y,p)
	+
	\frac{1}{t^2} \partial_i \phi_{\infty} \Big(\frac{x}{t}\Big) (\partial_{p^i} f_{\infty})(y,p).
	\nonumber
\end{align}
For $y$ in the support of $f_{\infty}$, so that \eqref{eq:introybound} holds,
\begin{equation} \label{eq:intronablaphicancel}
	\sup_{\vert y \vert \leq M} \Big\vert \partial_i \phi_{\infty} \Big(\frac{x}{t}\Big) - \partial_i \phi_{\infty}(p) \Big\vert
	\leq
	\Vert \nabla^2 \phi_{\infty} \Vert_{L^{\infty}} \Big( \frac{M}{t} + \frac{\log t \Vert \nabla \phi_{\infty} \Vert_{L^{\infty}}}{t} \Big),
\end{equation}
which ensures that there is a cancellation in the first line of \eqref{eq:introphi0f0} and one indeed has
\[
	\gs_{\phi_{[0]}} f_{[0]}(t,x,p)
	=
	\mathcal{O} \bigg(\frac{\log t}{t^{2}} \bigg).
\]

For $t$ suitably large, the expression \eqref{eq:introdefofy} can be inverted to give $p$ as an implicit function of $t,x,y$, and so one moreover checks
\begin{equation} \label{eq:introf0rho0}
	\int f_{[0]}(t,x,p) dp
	-
	\frac{1}{t^3}
	\varrho_{\infty} \left( \frac{x}{t} \right)
	=
	\int f_{\infty}(y,p) \det \frac{\partial p}{\partial y}dy + \frac{1}{t^3} \int f_{\infty}\left( y, \frac{x}{t} \right) dy
	=
	\mathcal{O} \left( \frac{\log t}{t^4} \right),
\end{equation}
where the uniform boundedness of $y$, and fact that $\det \frac{\partial p}{\partial y}$ is equal to $-t^{-3}$ to leading order (see Proposition \ref{prop:Jexpansion} below), has been used.

\subsubsection*{The case $K=1$}
Consider now the case $K=1$.  Recall $f_{[1]}$ and $\phi_{[1]}$ defined by \eqref{eq:introansatz1}--\eqref{eq:introansatz2}, so that
\[
	f_{[1]}(t,x,p) = f_{\infty}(y,p) + \frac{\log t}{t} f_{1,1}(y,p) + \frac{1}{t} f_{1,0}(y,p),
\quad
	\phi_{[1]}(t,x)
	=
	\frac{1}{t} \phi_{\infty} \left( \frac{x}{t} \right) +\frac{\log t}{t^2} \phi_{1,1} \left( \frac{x}{t} \right) + \frac{1}{t^2} \phi_{1,0} \left( \frac{x}{t} \right),
\]
for some smooth functions $f_{1,1}$, $f_{1,0}$, $\phi_{1,1}$ and $\phi_{1,0}$.  Revisiting \eqref{eq:introphi0f0} one sees that, in order to ensure that $\gs_{\phi_{[1]}} f_{[1]}(t,x,p) = \mathcal{O} \big(\frac{(\log t)^2}{t^{3}} \big)$, it is necessary to not only exploit the cancellation \eqref{eq:intronablaphicancel}, but to moreover take into account the higher order terms in an expansion (whose coefficients are functions of $y$ and $p$) for $\nabla \phi_{\infty} (x/t)$, in the region $\vert y \vert \leq M$.  Replacing $x/t$ using the expression \eqref{eq:introdefofy}, and Taylor expanding around $p$, one computes, for $\vert y \vert \leq M$,
\begin{equation} \label{eq:intronablaphiexpansion}
	\nabla \phi_{[1]}(t,x)
	=
	\frac{1}{t^2} \nabla \phi_{\infty} (p)
	+
	\frac{\log t}{t^3} \Big( \nabla \phi_{1,1}(p) - \nabla \phi_{\infty}(p) \cdot \nabla^{2} \phi_{\infty} (p) \Big)
	+
	\frac{1}{t^3} \Big( \nabla \phi_{1,0}(p) + y \cdot \nabla^{2} \phi_{\infty} (p) \Big)
	+
	\mathcal{O}\left(\frac{(\log t)^2}{t^4}\right).
\end{equation}
It thus follows that
\begin{align*}
	&
	\gs_{\phi_{[1]}} f_{[1]}(t,x,p)
	=
	\mathcal{O} \bigg( \frac{(\log t)^2}{t^3} \bigg)
	+
	\frac{\log t}{t^2}
	\Big[
	- f_{1,1}(y,p)
	-
	\big( \partial_i \phi_{1,1}(p)
	-
	2\partial_{j} \phi_{\infty}(p) \partial_i \partial_j \phi_{\infty}(p) 
	\big) (\partial_{x^i} f_{\infty})(y,p)
	\Big]
	\\
	&
	\quad
	+
	\frac{1}{t^2}
	\Big[
	f_{1,1}(y,p)
	-
	f_{1,0}(y,p)
	-
	\big(
	\partial_i \phi_{1,0}(p) + y^j \partial_i \partial_j \phi_{\infty}(p)
	\big)
	(\partial_{x^i} f_{\infty})(y,p)
	+
	\partial_{i} \phi_{\infty}(p) (\partial_{p^i} f_{\infty})(y,p)
	\Big]
	,
\end{align*}
and one sees that the former of \eqref{eq:introapproxsolutions} holds for $K=1$, provided the functions $f_{1,1}$, $f_{1,0}$, $\phi_{1,1}$ and $\phi_{1,0}$ satisfy
\begin{align}
	f_{1,1}(y,p)
	&
	=
	\big( 2 \partial_j \phi_{\infty}(p) \partial_i \partial_j \phi_{\infty}(p) - \partial_i \phi_{1,1}(p) \big) (\partial_{x^i} f_{\infty})(y,p),
	\label{eq:introiteratedef1}
	\\
	f_{1,0}(y,p)
	&
	=
	f_{1,1}(y,p)
	-
	\big( \partial_i \phi_{1,0}(p) + y^j \partial_i \partial_j \phi_{\infty}(p) \big) (\partial_{x^i} f_{\infty})(y,p)
	+
	\partial_i \phi_{\infty}(p) (\partial_{p^i} f_{\infty})(y,p).
	\label{eq:introiteratedef2}
\end{align}
Once $\varrho_{1,1}$ and $\varrho_{1,0}$ --- and hence $\phi_{1,1}$ and $\phi_{1,0}$ --- have been defined, $f_{1,1}$ and $f_{1,0}$ will indeed be defined by \eqref{eq:introiteratedef1} and \eqref{eq:introiteratedef2} respectively.  Observe that, regardless of how $\phi_{1,1}$ and $\phi_{1,0}$ are defined, the terms in \eqref{eq:introiteratedef1}--\eqref{eq:introiteratedef2} involving the functions $\phi_{1,1}$ and $\phi_{1,0}$ vanish upon integration in $y$:
\begin{equation} \label{eq:introiteratedefintegrated}
	\int_{\mathbb{R}^3} f_{1,1}(y,p) dy = 0,
	\qquad
	\int_{\mathbb{R}^3} f_{1,0}(y,p) dy
	=
	\Delta \phi_{\infty}(p)
	\int_{\mathbb{R}^3} f_{\infty}(y,p) dy
	+
	\partial_i \phi_{\infty}(p) \int_{\mathbb{R}^3} (\partial_{p^i} f_{\infty})(y,p) dy
	.
\end{equation}
This observation is returned to below.

Turning now to the latter of \eqref{eq:introapproxsolutions}, one similarly sees, upon revisiting \eqref{eq:introf0rho0}, that it is necessary to not only exploit the fact that $\det \frac{\partial p}{\partial y}$ is equal to $-t^{-3}$ to leading order, but to moreover take into account higher order terms in an expansion for $\det \frac{\partial p}{\partial y}$.  One computes
\begin{equation} \label{eq:introdetexpansion}
	\det \frac{\partial p}{\partial y}(t,x,y)
	=
	- \frac{1}{t^3} - \frac{\log t}{t^4} \Delta \phi_{\infty}\Big( \frac{x}{t} \Big) + \mathcal{O} \bigg( \frac{(\log t)^2}{t^5} \bigg),
\end{equation}
for $\vert y \vert + \vert x/t \vert \leq M$.  Similarly one has (recalling that \eqref{eq:introdefofy} can be inverted to give $p$ as an implicit function of $t,x,y$), for $\vert y \vert + \vert x/t \vert \leq M$,
\[
	f_{\infty}(y,p(t,x,y))
	=
	f_{\infty}\left( y, \frac{x}{t} \right)
	+
	\frac{\log t}{t}
	\partial_i \phi_{\infty}\Big( \frac{x}{t} \Big) (\partial_{p^i} f_{\infty}) \left( y, \frac{x}{t} \right)
	-
	\frac{1}{t}
	y^i (\partial_{p^i} f_{\infty}) \left( y, \frac{x}{t} \right)
	+
	\mathcal{O} \bigg( \frac{(\log t)^2}{t^2} \bigg).
\]
Recall now $\varrho_{[1]}$ defined by \eqref{eq:introansatz2}, so that
\[
	\varrho_{[1]}(t,x)
	=
	\frac{1}{t^3} \varrho_{\infty} \left( \frac{x}{t} \right)
	+
	\frac{\log t}{t^4} \varrho_{1,1} \left( \frac{x}{t} \right)
	+
	\frac{1}{t^4} \varrho_{1,0} \left( \frac{x}{t} \right),
\]
for some smooth functions $\varrho_{1,1}$ and $\varrho_{1,0}$.  It follows that
\begin{align*}
	&
	\int_{\mathbb{R}^3} f_{[1]} (t,x,p) dp - \varrho_{[1]} (t,x)
	\\
	&
	=
	- \frac{\log t}{t^4} \left[
	\varrho_{1,1} \left( \frac{x}{t} \right)
	+
	\int_{\mathbb{R}^3} f_{1,1}\left( y, \frac{x}{t} \right) dy
	+
	\Delta \phi_{\infty}\Big( \frac{x}{t} \Big) \int_{\mathbb{R}^3} f_{\infty}\left( y, \frac{x}{t} \right) dy
	-
	\partial_i \phi_{\infty}\Big( \frac{x}{t} \Big) \int_{\mathbb{R}^3} (\partial_{p^i} f_{\infty}) \left( y, \frac{x}{t} \right) dy
	\right]
	\\
	&
	\quad
	- \frac{1}{t^4} \left[
	\varrho_{1,0} \left( \frac{x}{t} \right)
	+
	\int_{\mathbb{R}^3} f_{1,0}\left( y, \frac{x}{t} \right) dy
	-
	\int_{\mathbb{R}^3} y^i (\partial_{x^i} f_{\infty}) \left( y, \frac{x}{t} \right) dy
	\right]
	+
	\mathcal{O} \bigg( \frac{(\log t)^2}{t^5} \bigg).
\end{align*}
The latter of \eqref{eq:introapproxsolutions} then holds if 
\begin{align}
	\varrho_{1,1} (w)
	&
	=
	-
	\int_{\mathbb{R}^3} f_{1,1}( y, w) dy
	-
	\Delta \phi_{\infty}(w) \int_{\mathbb{R}^3} f_{\infty} ( y, w) dy
	+
	\partial_i \phi_{\infty}(w) \int_{\mathbb{R}^3} (\partial_{p^i} f_{\infty}) ( y, w) dy,
	\label{eq:introiteratedef3}
	\\
	\varrho_{1,0} (w)
	&
	=
	-
	\int_{\mathbb{R}^3} f_{1,0} ( y, w) dy
	+
	\int_{\mathbb{R}^3} y^i (\partial_{x^i} f_{\infty}) ( y,w) dy.
	\label{eq:introiteratedef4}
\end{align}

Thus, $(f_{[1]}, \varrho_{[1]}, \phi_{[1]})$ is an approximate solution of order $1$ if $f_{1,1}$, $f_{1,0}$, $\varrho_{1,1}$, and $\varrho_{1,0}$, satisfy both \eqref{eq:introiteratedef1}--\eqref{eq:introiteratedef2}, and \eqref{eq:introiteratedef3}--\eqref{eq:introiteratedef4}.  Remarkably, as noted above, the terms involving $\phi_{1,1}$ and $\phi_{1,0}$ in \eqref{eq:introiteratedef1}--\eqref{eq:introiteratedef2} vanish upon integration in $y$ and thus, when the integrated expressions \eqref{eq:introiteratedefintegrated} are inserted into \eqref{eq:introiteratedef3}, one obtains explicit expressions for $\varrho_{1,1}$, and $\varrho_{1,0}$:
\begin{align}
	\varrho_{1,1} (w)
	&
	=
	-
	\Delta \phi_{\infty}(w) \int_{\mathbb{R}^3} f_{\infty} ( y, w) dy
	-
	\partial_i \phi_{\infty}(w) \int_{\mathbb{R}^3} (\partial_{p^i} f_{\infty}) ( y, w) dy.
	\label{eq:introiteratedef5}
\\
	\varrho_{1,0} (w)
	&
	=
	-
	\Delta \phi_{\infty}(w)
	\int_{\mathbb{R}^3} f_{\infty}(y,w) dy
	-
	\partial_i \phi_{\infty}(w) \int_{\mathbb{R}^3} (\partial_{p^i} f_{\infty}) ( y, w) dy
	+
	\int_{\mathbb{R}^3} y^i (\partial_{x^i} f_{\infty}) ( y,w) dy.
	\label{eq:introiteratedef6}
\end{align}
One therefore defines $\varrho_{1,1}$ and $\varrho_{1,0}$ by \eqref{eq:introiteratedef5} and \eqref{eq:introiteratedef6} respectively, and then defines $f_{1,1}$ and $f_{1,0}$ by \eqref{eq:introiteratedef1} and \eqref{eq:introiteratedef2} respectively (recalling that $\phi_{1,1}$ and $\phi_{1,0}$ are defined explicitly in terms of $\varrho_{1,1}$ and $\varrho_{1,0}$ as the unique solutions of the Poisson equation \eqref{eq:introPoissonapprox}).

\subsubsection*{The case $K\geq 2$}
For $K\geq 2$, one proceeds inductively by similarly inserting the expressions \eqref{eq:introansatz1}--\eqref{eq:introansatz2} into the Vlasov--Poisson system \eqref{eq:VP1}--\eqref{eq:VP2} as an ansatz and solving for $f_{k,l}$, and $\varrho_{k,l}$ order by order so that $(f_{[K]}, \varrho_{[K]}, \phi_{[K]})$ is an approximate solution of order $K$.  Higher order analogues of the steps discussed in the $K=1$ case above are required.  These steps form the content of Section \ref{section:expansions}.

In Proposition \ref{prop:phiderivexpansion}, a higher order analogue of the first order expansion \eqref{eq:intronablaphiexpansion} for $\nabla \phi_{[K]}$ is given.  The coefficients of the expansion for $\nabla \phi_{[K]}$ are denoted $\Psi_{k,l}$.  

In Section \ref{subsec:fexpansion} (see \eqref{eq:fklpsi2}--\eqref{eq:fklpsi3}) higher order analogues of the explicit expressions \eqref{eq:introiteratedef1}--\eqref{eq:introiteratedef2} for $f_{k,l}$ are given in terms of $\Psi_{k,l}$ (which are in turn given explicitly in term of $\phi_{k,l}$ in \eqref{eq:Psikl1}--\eqref{eq:Psikl2}) and in Proposition \ref{prop:fKdef} it is shown that, if these expressions are satisfied for all $0\leq k \leq K$, $0 \leq l \leq k$, then the former of \eqref{eq:introapproxsolutions} holds.  Moreover, an analogue of the observation \eqref{eq:introiteratedefintegrated} --- namely that, upon integrating the expression for $f_{k,l}$ with respect to its first argument, the terms involving $\phi_{k,l}$ all vanish and only terms involving $\phi_{k',l'}$, for $k'\leq k-1$, remain --- persists to higher orders.  See Remark \ref{rmk:totalyderiv}.

Turing again to the latter of \eqref{eq:introapproxsolutions}, one sees that a higher order analogue of the first order expansion \eqref{eq:introdetexpansion} for $\det (\partial p / \partial y)$ is required.  Such an expansion is given in Proposition \ref{prop:Jexpansion}, where the coefficients of this expansion are denoted $J_{k,l}$.  Higher order analogues of the expressions \eqref{eq:introiteratedef3}--\eqref{eq:introiteratedef4} for $\varrho_{k,l}$ are then given in terms of $f_{k,l}$ and $J_{k,l}$.  See \eqref{eq:rhoklf2}--\eqref{eq:rhoklf3}.  The analogue of the observation \eqref{eq:introiteratedefintegrated}, discussed in Remark \ref{rmk:totalyderiv}, means that these expressions, together with the expressions \eqref{eq:fklpsi2}--\eqref{eq:fklpsi3} for $f_{k,l}$, define each $\varrho_{k,l}$ and $f_{k,l}$ explicitly.  In Proposition \ref{prop:rhoKdef} it is then shown that, with these definitions, the latter of \eqref{eq:introapproxsolutions} holds.

The reader is referred to Section \ref{section:expansions} for more details.

\subsubsection{The proof of the main results --- existence of the solution}
\label{subsec:introexistence}
The existence part of the proof of Theorem \ref{thm:intromain} proceeds by considering some arbitrarily large time $T_f$, and considering an associated ``finite problem'' for each such time.  One defines ``final data'' at time $t=T_f$ to coincide with the approximate solution $f_{[K]}(T_f,\cdot,\cdot)$ (discussed in Section \ref{subsec:introexpansion}):
\begin{equation} \label{eq:introfinaltime}
	f \vert_{t=T_f} = f_{[K]}(T_f,\cdot,\cdot),
\end{equation}
for some sufficiently large $K$, to be determined later.  One then shows that the corresponding solution of Vlasov--Poisson exists on the interval $[T_0,T_f]$, for some $T_0$ which is independent of $T_f$.  The main step in the proof is in establishing uniform estimates for the solution on this interval, with constants independent of $T_f$.  Once such estimates have been established, the proof of the existence of a solution on $[T_0,\infty)$ follows from considering, on their common domain, the differences $f^{(T_f)} - f^{(T_f')}$ of such solutions corresponding to different final times $T_f<T_f'$, and showing that such differences vanish in the limit $T_f\to \infty$ (see Section \ref{section:logicofproof}).  The prescription \eqref{eq:introfinaltime}, for each $T_f$, ensures that the limiting solution on $[T_0,\infty)$ attains the scattering data $f_{\infty}$ in the sense \eqref{eq:intromainconvergence}.

The solution is not estimated directly but, for some fixed appropriate $K \geq 0$, the quantity
\[
	\check{f}_{[K]}(t,x,p)
	:=
	f(t,x,p)
	-
	f_{[K]}(t,x,p),
\]
is considered, where $f_{[K]}$ is the approximate solution discussed in Section \ref{subsec:introexpansion} (see equation \eqref{eq:introansatz1}).  This quantity satisfies an inhomogeneous equation of the form
\begin{equation} \label{eq:introfcheckk1}
	\gs_{\phi} \check{f}_{[K]} = F_{[K]},
\end{equation}
where
\begin{equation} \label{eq:introfcheckk2}
	F_{[K]} := - \gs_{\phi} f_{[K]}
	=
	- \partial_{x^i}(\phi - \phi_{[K]}) \partial_{p^i} f_{[K]} + \mathcal{O} \bigg(\frac{(\log t)^{1+K}}{t^{2+K}} \bigg),
\end{equation}
by virtue of the fact that $(f_{[K]}, \varrho_{[K]},\phi_{[K]})$ is an approximate solution of order $K$ (see \eqref{eq:introapproxsolutions}).  The definition \eqref{eq:introfinaltime} of the ``final data'' ensures that
\begin{equation} \label{eq:introfcheckk3}
	\check{f}_{[K]}(T_f,x,p) = 0,
\end{equation}
for all $x,p \in \mathbb{R}^3$.  For fixed $N\geq 6$, $L^2$ based energies of the form
\begin{equation} \label{eq:introEtdef}
	\check{\mathcal{E}}_{[K]}(t)
	: =
	\sum_{\vert I \vert +\vert J \vert = 0}^N \Vert L^I (t^{-1} \partial_p)^J \check{f}_{[K]}(t,\cdot,\cdot) \Vert_{L^2_x L^2_p},
\end{equation}
are considered, where $I$, $J$ are multi-indices and $L$ is an appropriate collection of vector fields (see Section \ref{subsec:vectorfields} below).

In the remainder of this section, the estimates for $\check{f}_{[K]}$ via the energy \eqref{eq:introEtdef}, on the interval $[T_0,T_f]$ are outlined.  After describing the procedure for obtaining $L^2$ based estimates, the case that $f_{\infty}$ is assumed to be small, in an appropriate sense, is discussed.  When $f_{\infty}$ is assumed to be appropriately small, it is only necessary to consider the energy \eqref{eq:introEtdef} for $K=0$ (and the thus the approximate solutions \eqref{eq:introansatz1}--\eqref{eq:introansatz2} are only considered up to order zero).  It is helpful to present first this simpler case.  The case for general $f_{\infty}$ is then presented, which involves considering the energy \eqref{eq:introEtdef} for $K$ suitably large, depending on an appropriate norm of $f_{\infty}$.

For simplicity, only the estimate for the energy \eqref{eq:introEtdef} with $N=0$ is discussed in the present overview.  In the proof of Theorem \ref{thm:intromain} it is important, however, to estimate derivatives of $f$ and, in order to control the nonlinear terms arising from commuting the equations, derivatives up to order $N \geq 6$ are estimated.  As noted above, a collection of vector fields $\{L_1, L_2, L_3\}$, are introduced.  These vector fields are defined in terms of the scattering profile $\phi_{\infty}$ so as to have good commutation properties with the Vlasov equation.  See Section \ref{subsec:vectorfields}.

\subsubsection*{$L^2$ estimates}
We have chosen, in this work, to make all estimates $L^2$ based (see Remark \ref{rmk:ellipticity}, and note further the simplicity with which \eqref{eq:introL2f} and \eqref{eq:introL2phi} are derived).  For any smooth, suitably decaying function $f$, and any smooth $\phi$, one has
\[	
	\partial_t f^2 + p^i \partial_{x^i} f^2 + \partial_{x^i} \phi \, \partial_{p^i} f^2 = 2 f \gs_{\phi} f,
\]
and, since the second two terms on the left vanish after integrating in $x$ and $p$,
\begin{equation} \label{eq:introL2f}
	\Vert f (t,\cdot, \cdot) \Vert_{L^2_xL^2_p}
	\lesssim
	\Vert f (T_f,\cdot, \cdot) \Vert_{L^2_xL^2_p}
	+
	\int_t^{T_f} \Vert \gs_{\phi} f (s,\cdot, \cdot) \Vert_{L^2_xL^2_p} ds.
\end{equation}
Moreover, for any smooth function $\phi$, multiplying $\Delta \phi$ by $\phi$ and integrating by parts gives,
\[
	\int_{\mathbb{R}^3} \vert \nabla \phi \vert^2 dx
	=
	-
	\int_{\mathbb{R}^3} \phi \Delta \phi dx
	\leq
	\Vert r^{-1} \phi \Vert_{L^2} \Vert r \Delta \phi \Vert_{L^2}
	\lesssim
	\Vert \nabla \phi \Vert_{L^2} \Vert r \Delta \phi \Vert_{L^2}
	,
\]
where $r(x) = \vert x \vert$, and a Hardy inequality (see Proposition \ref{prop:Hardy} below) is used in the last step.
Thus, if $\Delta \phi$ is supported in the region $\vert x \vert \lesssim t$, one has
\begin{equation} \label{eq:introL2phi}
	\Vert \nabla \phi (t,\cdot) \Vert_{L^2_x} \lesssim t \Vert \Delta \phi (t,\cdot) \Vert_{L^2_x}.
\end{equation}

\subsubsection*{Small $f_{\infty}$ and the case $K=0$}

One already sees the utility of the quantity $\check{f}_{[K]}$ for $K=0$, in that $\check{f}_{[0]}$ leads to an estimate for $f$, independent of $T_f$, provided $f_{\infty}$ is small in an appropriate norm.  Indeed, define $\check{f} = \check{f}_{[0]}$, and consider the equation \eqref{eq:introfcheckk1}--\eqref{eq:introfcheckk2}.  The approximate solution $f_{[0]}$ has the property that, for all $t \in [T_0,T_f]$,
\begin{equation} \label{eq:intropartialpf0}
	\sup_{x \in \mathbb{R}^3} \Big( \int_{\mathbb{R}^3} \vert \partial_p f_{[0]}(t,x,p) \vert^2 dp \Big)^{\frac{1}{2}}
	\leq
	t^{-\frac{1}{2}} \mathcal{F},
\end{equation}
where $\mathcal{F}$ is a constant depending on an appropriate norm of the scattering data $f_{\infty}$.  The final term on the right hand side of \eqref{eq:introfcheckk2} is supported in \eqref{eq:introybound} (and so in particular satisfies \eqref{eq:introindicator}).  Hence
\begin{equation} \label{eq:introorderestimate}
	\bigg\Vert \mathcal{O} \bigg(\frac{\log t}{t^{2}} \bigg) \bigg\Vert_{L^2_xL^2_p}
	\leq
	\mathcal{F}\frac{\log t}{t^{2}}.
\end{equation}
Applying the inequality \eqref{eq:introL2f} with $\check{f}$ and setting $K=0$ in the final condition \eqref{eq:introfcheckk3}, it follows that the energy $\check{\mathcal{E}} = \check{\mathcal{E}}_{[0]}$, defined by \eqref{eq:introEtdef} with $N=0$, satisfies, for all $T_0 \leq t \leq T_f$,
\[
	\check{\mathcal{E}}(t)
	\lesssim
	\int_t^{T_f} \Vert F_{[0]} (s,\cdot, \cdot) \Vert_{L^2_xL^2_p} ds
	\lesssim
	\mathcal{F}
	\int_t^{T_f} s^{-\frac{1}{2}} \Vert \nabla_x \check{\phi} (s,\cdot) \Vert_{L^2_x} ds
	+
	\mathcal{F}\frac{\log t}{t}
	\lesssim
	\mathcal{F}
	\int_t^{T_f} s^{\frac{1}{2}} \Vert \check{\varrho} (s,\cdot) \Vert_{L^2_x} ds
	+
	\mathcal{F}\frac{\log t}{t}
	,
\]
where the last step uses \eqref{eq:introL2phi}.  Using the Cauchy--Schwarz inequality and the fact \eqref{eq:introindicator}, it follows that the energy $\check{\mathcal{E}}(t)$ satisfies, for all $T_0 \leq t \leq T_f$,
\begin{equation} \label{eq:introEtinequ}
	\check{\mathcal{E}}(t)
	\leq
	\mathcal{F} \int_t^{T_f} \frac{\check{\mathcal{E}}(s)}{s} ds + \frac{\mathcal{F}}{t^{1-}},
\end{equation}
where $\mathcal{F}$ is a constant depending on an appropriate norm of the scattering data $f_{\infty}$, and the $-$ in $t^{1-}$ indicates a logarithmic loss.  The Gr\"{o}nwall inequality (see Proposition \ref{prop:Gronwall}) then implies that
\[
	\check{\mathcal{E}}(t)
	\leq
	\frac{\mathcal{F}}{t^{1-}}
	+
	\frac{\mathcal{F}^2}{t^{\mathcal{F}}} \int_t^{T_f} \frac{1}{s^{-\mathcal{F}+2 -}} ds.
\]
If $\mathcal{F}$ is sufficiently small --- which can be achieved by making the relevant norm of the scattering profile $f_{\infty}$ suitably small --- then this inequality gives
\[
	\check{\mathcal{E}}(t)
	\leq
	\frac{\mathcal{F} +\mathcal{F}^2}{t^{1-}},
\]
as desired.  For larger $\mathcal{F}$, however, this inequality fails to provide an estimate for $\check{\mathcal{E}}(t)$ which is uniform in $T_f$.

\subsubsection*{General $f_{\infty}$ and the case $K\geq 1$}

In order to obtain uniform estimates of the solution without assuming any smallness on the scattering data $f_{\infty}$, one considers the energy $\check{\mathcal{E}}_{[K]}(t)$, for general $K \geq 0$.  Consider the equation \eqref{eq:introfcheckk1}--\eqref{eq:introfcheckk2}.  

The approximate solution $f_{[K]}$ retains the property \eqref{eq:intropartialpf0}, i.\@e.\@ for all $t \in [T_0,T_f]$,
\[
	\sup_{x \in \mathbb{R}^3} \Big( \int_{\mathbb{R}^3} \vert \partial_p f_{[K]}(t,x,p) \vert^2 dp \Big)^{\frac{1}{2}}
	\leq
	t^{-\frac{1}{2}} \mathcal{F},
\]
and, further, the final term on the right hand side of \eqref{eq:introfcheckk2} satisfies \eqref{eq:introorderestimate} with $(\log t)^{1+K}/t^{2+K}$ in place of $\log t/t^2$.
Revisiting the steps above one therefore sees that the energy $\check{\mathcal{E}}_{[K]}(t)$ satisfies an inequality which, in the case $N=0$, takes the form
\[
	\check{\mathcal{E}}_{[K]}(t)
	\leq
	\mathcal{F} \int_t^{T_f} \frac{\check{\mathcal{E}}_{[K]}(s)}{s} ds + \frac{\mathcal{F}}{t^{K+1-}},
\]
(see the proof of Proposition \ref{prop:fcheckkho}, in particular equation \eqref{eq:eqnforEcheckk}, for the inequality satisfied by $\check{\mathcal{E}}_{[K]}(t)$ for general $N \geq 6$).  Note the improvement, compared to \eqref{eq:introEtinequ}.  The Gr\"{o}nwall inequality (see Proposition \ref{prop:Gronwall}) then gives
\[
	\check{\mathcal{E}}_{[K]}(t)
	\leq
	\frac{\mathcal{F}}{t^{K+1-}}
	+
	\frac{\mathcal{F}^2}{t^{\mathcal{F}}} \int_t^{T_f} \frac{1}{s^{-\mathcal{F}+K+2-}} ds.
\]
Now, if $K$ is chosen to be suitably large, with respect to $\mathcal{F}$ (and hence to the size of $f_{\infty}$), it follows that
\[
	\check{\mathcal{E}}_{[K]}(t)
	\leq
	\frac{\mathcal{F} +\mathcal{F}^2}{t^{K+1-}},
\]
as desired.

\subsection{Outline of the paper}

Section \ref{section:prelim} concerns certain preliminaries, including various inequalities which are used throughout, along with a collection of vector fields used in the proof of Theorem \ref{thm:intromain} and a discussion of their basic properties.  In Section \ref{section:theorem}, more precise versions of Theorem \ref{thm:intromain}, Theorem \ref{thm:intromain2} and Theorem \ref{thm:intromain3} are stated.  Section \ref{section:expansions} concerns the the functions $f_{k,l}$, $\varrho_{k,l}$, $\phi_{k,l}$, which define approximate solutions of the Vlasov--Poisson system \eqref{eq:VP1}--\eqref{eq:VP2}, as outlined in Section \ref{subsec:introexpansion} above.  In Section \ref{section:finiteprob} a sequence of ``finite problems'' is introduced, and the approximate solutions of Section \ref{section:expansions} are used to obtain estimates on the solutions of these finite problems, as outlined in Section \ref{subsec:introexistence} above.  Finally, in Section \ref{section:logicofproof}, the proof of the precise versions of Theorem \ref{thm:intromain}, Theorem \ref{thm:intromain2} and Theorem \ref{thm:intromain3}, stated in Section \ref{section:theorem}, are given using the solutions of these finite problems and the estimates obtained in Section \ref{section:finiteprob}.

\subsection*{Acknowledgements}

We would like to thank MATRIX, the mathematical research institute in Creswick, Australia, for their hospitality and  the collaborative environment during the workshop on ``Hyperbolic PDEs and Nonlinear Evolution Problems'', 18--29th September 2023, where part of this research was conducted.
     M.\@T.\@ acknowledges support through Royal Society Tata University Research Fellowship URF\textbackslash R1\textbackslash 191409.

\section{Preliminaries}
\label{section:prelim}

This section contains certain preliminaries which will be used throughout the remainder of the article.  In Section \ref{subsec:ineqs} some functional inequalities, such as Sobolev and Hardy inequalities, are collected.  Section \ref{subsec:elliptictransport} contains basic $L^2$ based estimates for the Vlasov equation and the Poisson equation.  In Section \ref{subsec:vectorfields} a collection of vector fields is introduced and their commutation properties with the Vlasov equation are discussed.  Further, a weighted Sobolev inequality, which features these vector fields, is shown.  Finally, in Section \ref{subsec:constants}, conventions for the constants appearing in this work are stated.

\subsection{Functional inequalities}
\label{subsec:ineqs}

Define the norms, for any smooth function $h\colon \mathbb{R}^3 \to \mathbb{R}$,
\[
	\Vert h \Vert_{L^{\infty}(\mathbb{R}^3)} = \sup_{x\in\mathbb{R}^3} \vert h(x) \vert,
	\qquad
	\Vert h \Vert_{L^2(\mathbb{R}^3)}^2 = \int_{\mathbb{R}^3} \vert h(x) \vert^2 dx,
\]
\[
	\Vert h \Vert_{H^k(\mathbb{R}^3)}^2
	=
	\sum_{i_1,i_2,i_3=0}^k \int_{\mathbb{R}^3} \vert \partial_{x^1}^{i_1} \partial_{x^2}^{i_2} \partial_{x^3}^{i_3} h(x) \vert^2 dx,
	\qquad
	\Vert h \Vert_{\mathring{H}^k(\mathbb{R}^3)}^2
	=
	\sum_{i_1,i_2,i_3=k} \int_{\mathbb{R}^3} \vert \partial_{x^1}^{i_1} \partial_{x^2}^{i_2} \partial_{x^3}^{i_3} h(x) \vert^2 dx.
\]
Functions $h\colon \mathbb{R}^6 \to \mathbb{R}$ of both $x$ and $p$ will also often be considered and so, for emphasis, we also write
\[
	\Vert h \Vert_{L^2_xL^2_p}^2 := \Vert h \Vert_{L^2(\mathbb{R}^6)}^2 = \int_{\mathbb{R}^3_x} \int_{\mathbb{R}^3_p} \vert h(x,p) \vert^2 dp dx.
\]

The first result is the following $L^{\infty}$---$L^2$ Sobolev inequality.

\begin{proposition}[$L^{\infty}$---$L^2$ Sobolev inequality on $\mathbb{R}^3$] \label{prop:SobolevL2}
	There exists a constant $C$ such that, for any function $h \in H^2(\mathbb{R}^3)$,
	\[
		\Vert h \Vert_{L^{\infty}(\mathbb{R}^3)}
		\leq
		C
		\Vert h \Vert_{L^{2}(\mathbb{R}^3)}^{\frac{1}{4}} \Vert h \Vert_{\mathring{H}^2(\mathbb{R}^3)}^{\frac{3}{4}}.
	\]
\end{proposition}

\begin{proof}
	First note that
	\begin{equation} \label{eq:prelimSobolev}
		\Vert h \Vert_{L^{\infty}(\mathbb{R}^3)}
		\leq
		C(
		\Vert h \Vert_{L^{2}(\mathbb{R}^3)} + \Vert h \Vert_{\mathring{H}^2(\mathbb{R}^3)}).
	\end{equation}
	Indeed, since $h \in H^2(\mathbb{R}^3)$ the Fourier inversion formula holds and so, for any $x\in \mathbb{R}^3$,
	\[
		\vert h (x) \vert
		\leq
		\int_{\mathbb{R}^3} \vert \hat{f} (\xi) \vert d \xi
		\leq
		\Big( \int_{\vert \xi \vert \leq 1} 1 \, d\xi \Big)^{\frac{1}{2}} \Big( \int \vert \hat{f} (\xi) \vert^2 d \xi \Big)^{\frac{1}{2}}
		+
		\Big( \int_{\vert \xi \vert > 1} \vert \xi \vert^{-4} d\xi \Big)^{\frac{1}{2}} \Big( \int \vert \xi \vert^4 \vert \hat{f} (\xi) \vert^2 d \xi \Big)^{\frac{1}{2}}.
	\]
	The proof of \eqref{eq:prelimSobolev} then follows from the Plancherel formula and the fact that
	\[
		\int_{\vert \xi \vert > 1} \vert \xi \vert^{-4} d\xi < \infty,
	\]
	in $3$ dimensions.
	
	Now, for any $\lambda>0$, define
	\[
		h_{\lambda}(x) = h(\lambda x).
	\]
	It follows that
	\[
		\Vert h_{\lambda} \Vert_{L^2(\mathbb{R}^3)} = \lambda^{-\frac{3}{2}} \Vert h \Vert_{L^2(\mathbb{R}^3)},
		\qquad
		\Vert h_{\lambda} \Vert_{\mathring{H}^1(\mathbb{R}^3)} = \lambda^{-\frac{1}{2}} \Vert h \Vert_{\mathring{H}^1(\mathbb{R}^3)},
		\qquad
		\Vert h_{\lambda} \Vert_{\mathring{H}^2(\mathbb{R}^3)} = \lambda^{\frac{1}{2}} \Vert h \Vert_{\mathring{H}^2(\mathbb{R}^3)},
	\]
	and thus, applying \eqref{eq:prelimSobolev} with $h_{\lambda}$ in place of $h$,
	\[
		\Vert h \Vert_{L^{\infty}(\mathbb{R}^3)}
		=
		\Vert h_{\lambda} \Vert_{L^{\infty}(\mathbb{R}^3)}
		\leq
		C(
		\lambda^{-\frac{3}{2}} \Vert h \Vert_{L^{2}(\mathbb{R}^3)} + \lambda^{\frac{1}{2}} \Vert h \Vert_{\mathring{H}^2(\mathbb{R}^3)}).
	\]
	The proof follows from setting $\lambda = \Vert h \Vert_{L^{2}(\mathbb{R}^3)}^{\frac{1}{2}} \Vert h \Vert_{\mathring{H}^2(\mathbb{R}^3)}^{-\frac{1}{2}}$ (and noting that the inequality is trivial if $\Vert h \Vert_{\mathring{H}^2(\mathbb{R}^3)} = 0$).
\end{proof}

\begin{remark}[Sobolev inequality with weighted derivatives] \label{rmk:Sobolev}
	In the proof of the main results of this article, weighted operators of the form $t\nabla_x$ will often be considered, and so the Sobolev inequality of Proposition \ref{prop:SobolevL2} will typically be used in the form, for any appropriate function $h\colon [T_0,\infty) \times \mathbb{R}^3 \to \mathbb{R}$,
	\begin{align*}
		\Vert h(t,\cdot) \Vert_{L^{\infty}(\mathbb{R}^3)}
		\leq
		C \Vert h(t,\cdot) \Vert_{L^{2}(\mathbb{R}^3)}^{\frac{1}{4}} \Vert h(t,\cdot) \Vert_{\mathring{H}^2(\mathbb{R}^3)}^{\frac{3}{4}}
		&
		\leq
		\frac{C}{t^{\frac{3}{2}}}
		\Vert h(t,\cdot) \Vert_{L^{2}(\mathbb{R}^3)}^{\frac{1}{4}} 
		\sum_{i,j=1}^3 \Vert (t\partial_{x^i}) (t\partial_{x^j}) h(t,\cdot) \Vert_{L^2(\mathbb{R}^3)}^{\frac{3}{4}}
		\\
		&
		\leq
		\frac{C}{t^{\frac{3}{2}}}
		\Big(
		\Vert h(t,\cdot) \Vert_{L^{2}(\mathbb{R}^3)}
		+
		\sum_{i,j=1}^3 \Vert (t\partial_{x^i}) (t\partial_{x^j}) h(t,\cdot) \Vert_{L^2(\mathbb{R}^3)}
		\Big),
	\end{align*}
	where the final step follows from Young's Inequality.
	See, in particular, the use of this fact in the proof of Proposition \ref{prop:systemfcheck}.
\end{remark}

Similarly, one has the following Sobolev inequality which can be applied to functions $a \colon \mathbb{R}^3_x \times \mathbb{R}^3_p \to \mathbb{R}$.

\begin{proposition}[$L^{\infty}$---$L^2$ Sobolev inequality on $\mathbb{R}^6$] \label{prop:SobolevL2R6}
	There exists a constant $C$ such that, for any function $a \in H^4(\mathbb{R}^6)$,
	\[
		\Vert a \Vert_{L^{\infty}(\mathbb{R}^6)}
		\leq
		C
		\Vert a \Vert_{L^{2}(\mathbb{R}^6)}^{\frac{1}{4}} \Vert a \Vert_{\mathring{H}^4(\mathbb{R}^6)}^{\frac{3}{4}}
		\leq
		C
		\Vert a \Vert_{H^4(\mathbb{R}^6)}.
	\]
\end{proposition}

Consider also the following Hardy inequality.

\begin{proposition}[Hardy inequality on $\mathbb{R}^3$] \label{prop:Hardy}
	For any function $h \in H^1(\mathbb{R}^3)$,
	\[
		\int_{\mathbb{R}^3} \vert x \vert^{-2} h^2 dx
		\leq
		4 \int_{\mathbb{R}^3} \vert \nabla h \vert^2 dx. 
	\]
\end{proposition}

\begin{proof}
	Consider polar coordinates $(r,\theta^1,\theta^2)$ on $\mathbb{R}^3$.  For fixed $(\theta^1,\theta^2)$, since $h \in H^1(\mathbb{R}^3)$, there is a sequence $r_n$ such that $\lim_{n \to \infty} (rh^2)(r_n, \theta^1,\theta^2) = 0$.  Integrating $\partial_r (r h^2)$ between $0$ and $r_n$, and taking $n \to \infty$, it follows that
	\[
		0
		=
		\int_0^{\infty} \partial_r (r h^2) dr
		=
		\int_0^{\infty} h^2 + 2 r h \partial_r h dr,
	\]
	and so, by the Cauchy--Schwarz inequality,
	\[
		\int_0^{\infty} h^2 dr
		\leq
		2 \Big( \int_0^{\infty} h^2 dr\Big)^{\frac{1}{2}} \Big(\int_0^{\infty} (\partial_r h)^2 r^2 dr\Big)^{\frac{1}{2}}.
	\]
	The result follows after dividing by $( \int_0^{\infty} h^2 dr)^{\frac{1}{2}}$, squaring, and integrating over $S^2$.
\end{proof}

The following form of Taylor's Theorem will be used.

\begin{proposition}[Taylor's Theorem] \label{prop:Taylor}
	For any smooth function $h \colon \mathbb{R}^n \to \mathbb{R}$, any $x_0 \in \mathbb{R}^n$, and any $K \geq 0$,
	\begin{align*}
		h(x)
		=
		\
		&
		\sum_{k=0}^K
		\frac{1}{k!} (x-x_0)^{i_1} \ldots (x-x_0)^{i_k} (\partial_{i_1} \ldots \partial_{i_k} h)(x_0)
		\\
		&
		+
		\frac{1}{K!} (x-x_0)^{i_1} \ldots (x-x_0)^{i_{K+1}} \int_0^1 (1-s)^K (\partial_{i_1} \ldots \partial_{i_{K+1}} h)(sx) ds,
	\end{align*}
	where $(x-x_0)^i \in \mathbb{R}$ denotes the $i$-th component of $x-x_0 = ((x-x_0)^1,\ldots,(x-x_0)^n) \in \mathbb{R}^n$.
\end{proposition}

Finally, the following Gr\"{o}nwall type inequality will also be used.

\begin{proposition}[Gr\"{o}nwall inequality] \label{prop:Gronwall}
	Suppose $0<T<T_f$ and $v \colon [T,T_f]\to \mathbb{R}$ satisfies
	\[
		v(t) \leq b(t) + \int_t^{T_f} a(s) v(s) ds,
	\] 
	for all $T \leq t \leq T_f$, for some $a,b \colon [T,T_f]\to [0,\infty)$.  Then
	\[
		v(t) \leq b(t) + \int_t^{T_f} a(s) b(s) e^{\int_t^s a(s')ds'}ds,
	\]
	for all $T_0 \leq t \leq T$.
\end{proposition}

\begin{proof}
	First note that
	\[
		\frac{d}{dt} \Big(  \int_t^{T_f} a(s) v(s) ds \, e^{-\int^{T_f}_t a(s')ds'} \Big)
		=
		- a(t) \, e^{\int_t^{T_f} a(s')ds'} \Big( v(s) - \int_t^{T_f} a(s) v(s) ds \Big)
		\geq
		- a(t) b(t) \, e^{-\int_t^{T_f} a(s')ds'}.
	\]
	Integrating between $T_f$ and $t$ gives
	\[
		\int_t^{T_f} a(s) v(s) ds \, e^{-\int_t^{T_f} a(s')ds'}
		\leq
		\int_t^{T_f} a(s) b(s) \, e^{-\int_s^{T_f} a(s')ds'} ds,
	\]
	and the result follows.
\end{proof}

\subsection{Elliptic and transport estimates}
\label{subsec:elliptictransport}

The following gradient estimate for the Poisson equation will be used.

\begin{proposition}[Gradient estimate] \label{prop:gradelliptic}
	For any function $h\in H^2(\mathbb{R}^3)$,
	\begin{equation} \label{eq:gradelliptic}
		\Vert \nabla h \Vert_{L^2} \lesssim \Vert r \Delta h \Vert_{L^2}.
	\end{equation}
\end{proposition}

\begin{proof}
	Integrating by parts,
	\[
		\int_{\mathbb{R}^3} \vert \nabla h \vert^2 dx
		=
		-
		\int_{\mathbb{R}^3} h \Delta h dx
		\leq
		\Vert r^{-1} h \Vert_{L^2} \Vert r \Delta h \Vert_{L^2},
	\]
	where $r(x) = \vert x \vert$.  The proof then follows from the Hardy inequality, Proposition \ref{prop:Hardy}.
\end{proof}

\begin{remark}[Full elliptic estimate] \label{rmk:fullelliptic}
	In addition to Proposition \ref{prop:gradelliptic}, one also has the full elliptic estimate
	\begin{equation} \label{eq:fullelliptic}
		\sum_{i,j=1}^3\Vert \partial_i \partial_j h \Vert_{L^2} \lesssim \Vert \Delta h \Vert_{L^2},
	\end{equation}
	which holds for all functions $h\in H^2(\mathbb{R}^3)$.  In order to illustrate an approach which can be applied to other equations, we have elected not to use the estimate \eqref{eq:fullelliptic} in this work (see Remark \ref{rmk:ellipticity}).  In particular, the estimate \eqref{eq:gradelliptic} has appropriate analogues --- in terms of number of derivatives appearing on each side of the inequality --- for hyperbolic operators, in contrast to the full elliptic estimate \eqref{eq:fullelliptic}.
\end{remark}

The next result is an $L^2$ estimate for the Vlasov equation.

\begin{proposition}[$L^2$ estimates for inhomogeneous Vlasov equation] \label{prop:generalVlasovestimates}
	Let $a\colon [T,T_f] \times \mathbb{R}^3_x \times \mathbb{R}^3_p \to \mathbb{R}$ be a suitably decaying smooth function satisfying
	\[
		\partial_t a + p^i \partial_{x^i} a + \partial_{x^i} \phi \, \partial_{p^i} a = H,
	\]
	for some smooth $H\colon [T,T_f] \times \mathbb{R}^3_x \times \mathbb{R}^3_p \to \mathbb{R}$ and some smooth $\phi \colon [T,T_f] \times \mathbb{R}^3_x \to \mathbb{R}$.  Then $a$ satisfies the $L^2$ estimate,
	\[
		\Vert a(t,\cdot,\cdot) \Vert_{L^2_x L^2_p}
		\lesssim
		\Vert a(T_f,\cdot,\cdot) \Vert_{L^2_x L^2_p}
		+
		\int^{T_f}_t \Vert H(s,\cdot,\cdot) \Vert_{L^2_x L^2_p} ds,
	\]
	for all $t \in [T,T_f]$, where
	\[
		\Vert a(t,\cdot,\cdot) \Vert_{L^2_x L^2_p}^2
		=
		\int_{\mathbb{R}^3} \int_{\mathbb{R}^3} \vert a(t,x,p) \vert^2 dp dx.
	\]
\end{proposition}

\begin{proof}
	First note that $a^2$ satisfies
	\[
		\partial_t a^2 + p^i \partial_{x^i} a^2 + \partial_{x^i} \phi \, \partial_{p^i} a^2 = 2 a H,
	\]
	and so
	\[
		\partial_t \Vert a(t,\cdot,\cdot) \Vert_{L^2_x L^2_p}^2
		=
		\int_{\mathbb{R}^3} \int_{\mathbb{R}^3}
		-
		p \cdot \nabla_x (\vert a(t,x,p) \vert^2)
		-
		\nabla_x \phi \cdot \nabla_p (\vert a(t,x,p) \vert^2)
		+
		2 a(t,x,p) H(t,x,p)
		dp dx.
	\]
	Since $a$ decays, it follows that the first two terms on the right hand side vanish and so, by Cauchy--Schwarz, for any $t \leq t_0 \leq T_f$,
	\[
		\partial_t \Vert a(t,\cdot,\cdot) \Vert_{L^2_x L^2_p}
		=
		\frac{\partial_t \Vert a(t,\cdot,\cdot) \Vert_{L^2_x L^2_p}^2}{2 \Vert a(t,\cdot,\cdot) \Vert_{L^2_x L^2_p}}
		\leq
		\Vert H(t,\cdot,\cdot) \Vert_{L^2_x L^2_p}.
	\]
	The result follows after integrating from $t$ to $T_f$.
\end{proof}

\subsection{Vector fields and multi-index notation}
\label{subsec:vectorfields}

Suppose a smooth function $\phi_{\infty} \colon \mathbb{R}^3 \to \mathbb{R}$ is given.  For $k=1,2,3$, define vector fields
\begin{equation} \label{eq:defofL}
	L_k
	=
	t \partial_{x^k} + \partial_{p^k}
	+
	\frac{\log t}{t} \partial_k \partial_i \phi_{\infty}(p) \partial_{p^i}.
\end{equation}

Given a multi-index $I=(I_1,I_2,I_3)$, with $I_1,I_2,I_3 \geq 0$, define the operators
\[
	L^I = (L_1)^{I_1} (L_2)^{I_2} (L_3)^{I_3},
	\quad
	(t^{-1}\partial_p)^I = (t^{-1}\partial_{p^1})^{I_1} (t^{-1}\partial_{p^2})^{I_2} (t^{-1}\partial_{p^3})^{I_3},
	\quad
	(t \partial_x)^I = (t \partial_{x^1})^{I_1} (t \partial_{x^2})^{I_2} (t \partial_{x^3})^{I_3}.
\]

The main reason for introducing the vector fields $L_k$ is due to the form of the commutator $[ \mathbb{X},L_k]$, given in the following proposition.  See the discussion in Remark \ref{rmk:goodvectorfields} below. The vector fields $L_k$ also behave well when applied to functions of $x-tp + \log t \nabla \phi_{\infty}( p )$ and $p$.  See Remark \ref{rmk:functionsofy} below.

The vector fields $L_i$ and $t^{-1} \partial_{p^i}$ have the following commutation properties with the Vlasov equation.

\begin{proposition}[Commutation of the Vlasov equation with $L_i$, and $t^{-1} \partial_{p^i}$] \label{prop:commVlasovLS}
	For $i=1,2,3$, the vector fields $L_i$ satisfy,
	\begin{align} \label{eq:gsLtildecommutator}
		\Big[
		\gs_{\phi}
		,
		L_i
		\Big]
		=
		\
		&
		\Big(
		t^{-2} \partial_i \partial_{k} \phi_{\infty}(p)
		-
		t \partial_{x^i} \partial_{x^k} \phi(t,x)
		\Big)
		\partial_{p^k}
		-
		\frac{\log t}{t^2} \partial_i \partial_{j} \phi_{\infty}(p) \, \big( t \partial_{x^j} + \partial_{p^j} \big)
		\\
		&
		+
		\frac{\log t}{t} \partial_{x^k} \phi(t,x) \, \partial_i \partial_k \partial_{l} \phi_{\infty}(p) \, \partial_{p^l}
		,
		\nonumber
	\end{align}
	and $t^{-1} \partial_{p^i}$ satisfy,
	\begin{equation} \label{eq:gsdpcomm}
		\Big[
		\gs_{\phi}
		,
		t^{-1} \partial_{p^i}
		\Big]
		=
		-
		\frac{1}{t^2} \big( t \partial_{x^i} + \partial_{p^i} \big)
		.
	\end{equation}
\end{proposition}

\begin{proof}
	The proof is a direct computation.
\end{proof}

\begin{remark}[Corrections to vector fields]
	Note that the vector fields obtained by setting $\phi_{\infty} \equiv 0$ in \eqref{eq:defofL} commute with the free transport operator (i.\@e.\@ the operator $\partial_t + p^i \partial_{x^i}$).  The logarithmic corrections in \eqref{eq:defofL} are crucial in ensuring the good commutation properties of Proposition \ref{prop:commVlasovLS} (see Remark \ref{rmk:goodvectorfields}), and the boundedness property of the right hand side of \eqref{eq:functionsofy1} below.  Corrections to vector fields with good commutation properties of the linearised part of the equation, of this form, played a central role in the proof \cite{Smu}, and the works \cite{FJS,LiTa} on the stability of Minkowski space for the Einstein--Vlasov system.  In fact, it is exactly because the solutions only scatter in a modified sense in $3+1$ dimensions that corrections to these vector fields are required in these works.  Note that these corrections are, in a sense, simpler in the present setting since they are defined directly with respect to the scattering profile $\phi_{\infty}$.
\end{remark}

\begin{remark}[Good commutation properties of $L_i$] \label{rmk:goodvectorfields}
	In order to understand the improvement that the corrected vector fields $L_i$ have over the non-corrected vector fields $t\partial_{x^i} + \partial_{p^i}$, it is a helpful exercise to assume the conclusions of the main results, in particular that solutions satisfy the pointwise estimates
	\begin{equation} \label{eq:goodvectorfields}
		\vert L_j f \vert + \vert t^{-1} \partial_p f\vert \leq \mathcal{F},
		\qquad
		\vert \nabla^2 \phi \vert
		\leq
		\frac{\mathcal{F}}{t^3},
		\qquad
		\vert \nabla^2 \check{\phi}_{[0]} \vert
		\leq
		\frac{\mathcal{F} \log t}{t^4},
	\end{equation}
	where $\check{\phi}_{[0]}(t,x) = \phi(t,x) - t^{-1} \phi_{\infty}(x/t)$, and check the behaviour of the respective commutators when applied to a solution.
	Indeed, assuming \eqref{eq:goodvectorfields}, the commutators \eqref{eq:gsLtildecommutator}, when applied to a solution $f$, have the property that
	\begin{multline*}
		\Big\vert \Big[
		\gs_{\phi}
		,
		L_i
		\Big] f \Big\vert
		\lesssim
		t^2 \vert \nabla^2 \check{\phi}_{[0]} \vert \, \vert t^{-1} \partial_p f \vert
		+
		\frac{\log t}{t^2} \vert \nabla^2 \phi_{\infty} \vert \, \vert L_j f \vert
		\\
		+
		\frac{(\log t)^2}{t^2} \vert \nabla^2 \phi_{\infty} \vert^2 \vert t^{-1} \partial_p f \vert
		+
		\frac{\log t}{t^2} \vert \nabla^3 \phi_{\infty} \vert \vert \nabla \phi \vert \vert t^{-1} \partial_p f \vert
		\lesssim
		\frac{\log t}{t^2} \mathcal{F}
		.
	\end{multline*}
	Contrast with the borderline behaviour exhibited by the commutator of $\gs_{\phi}$ with the non-corrected vector fields (obtained by setting $\phi_{\infty} \equiv 0$ in \eqref{eq:gsLtildecommutator}):
	\[
		\Big\vert \Big[
		\gs_{\phi}
		,
		(t\partial_{x^i} + \partial_{p^i})
		\Big] f \Big\vert
		\lesssim
		t^2 \vert \nabla^2 \phi \vert \, \vert t^{-1} \partial_p f \vert
		\lesssim
		\frac{\mathcal{F}}{t}.
	\]
	This borderline behaviour would lead to a logarithmic divergence upon integrating globally in time.
\end{remark}

\begin{remark}[Functions of $x - tp + \log t \nabla \phi_{\infty}(p)$] \label{rmk:functionsofy}
	Many of the objects considered in the following are functions of $x - tp + \log t \nabla \phi_{\infty}(p)$ and $p$, and the vector fields $L_i$ also have good properties when applied to such functions.  Indeed, for any smooth $h \colon \mathbb{R}^3 \times \mathbb{R}^3 \to \mathbb{R}$,
	\begin{equation} \label{eq:functionsofy1}
		L_i \big( h(x - tp + \log t \nabla \phi_{\infty}(p),p) \big)
		=
		\frac{(\log t)^2}{t} \partial_k \partial_i \phi_{\infty}(p) (\partial_{x^k}h)
		+
		\Big( \delta_{ik} + \frac{\log t}{t} \partial_k \partial_i \phi_{\infty}(p) \Big) (\partial_{p^k}h),
	\end{equation}
	and the coefficients of the derivatives of $h$ are not growing in $t$.  Contrast with the non-corrected vector fields $t\partial_{x^i} + \partial_{p^i}$.  Note also that
	\begin{equation} \label{eq:functionsofy2}
		(t^{-1} \partial_{p^i}) \big( h(x - tp + \log t \nabla \phi_{\infty}(p),p) \big)
		=
		\Big( - \delta_{ik} + \frac{\log t}{t} \partial_k \partial_i \phi_{\infty}(p) \Big) (\partial_{x^k}h)(y,p)
		+
		t^{-1} (\partial_{p^i}h)(y,p)
		.
	\end{equation}
	These facts are in particular used in the proofs of Proposition \ref{prop:phiderivexpansion} and Proposition \ref{prop:fKdef} below.
\end{remark}

One also has a weighted version of the Sobolev inequality in which the derivatives are taken with respect to the above vector fields.

\begin{proposition}[Sobolev inequality with vector field derivatives] \label{prop:Sobolevvf}
	There exists a constant $C$ such that, for any function $h \in H^2(\mathbb{R}^3_x \times \mathbb{R}^3_p)$ and any $t \geq 1$,
	\[
		\sup_{x \in \mathbb{R}^3} \Big( \int_{\mathbb{R}^3} \vert h(x,p) \vert^2 dp \Big)^{\frac{1}{2}}
		\leq
		\frac{C}{t^{\frac{3}{2}}}
		\sum_{\vert I \vert =1}^3 \sum_{n = 0}^2 \Big( \frac{\log t}{t} \Vert \partial^I \Delta \phi_{\infty} \Vert_{L^{\infty}}
		\Big)^n
		\sum_{\vert J \vert =0}^3 \Vert L^J h \Vert_{L^2_x L^2_p}
		.
	\]
\end{proposition}

\begin{proof}
	First note that, for any smooth compactly supported function $h \colon \mathbb{R}^3_x \times \mathbb{R}^3_p \to \mathbb{R}$,
	\begin{equation} \label{eq:L1Sobolev}
		\sup_{x \in \mathbb{R}^3} \Big\vert \int_{\mathbb{R}^3}  h(x,p) dp \Big\vert
		\leq
		\frac{C}{t^3}
		\sum_{\vert I \vert =1}^3 \sum_{n = 0}^3 \Big( \frac{\log t}{t} \Vert \partial^I \Delta \phi_{\infty} \Vert_{L^{\infty}}
		\Big)^n
		\sum_{\vert J \vert =0}^3
		\int_{\mathbb{R}^3} \int_{\mathbb{R}^3} \vert L^J h(x,p) \vert dp dx.
	\end{equation}
	Indeed, for any $x=(x^1,x^2,x^3) \in \mathbb{R}^3$,
	\begin{align}
		\int_{\mathbb{R}^3}  h(x,p) dp
		&
		=
		\frac{1}{t} \int_{-\infty}^{x^3} t \partial_{x^3} \Big( \int_{\mathbb{R}^3}  h(x^1,x^2,\tilde{x}^3,p) dp \Big) d\tilde{x}^3
		\nonumber
		\\
		&
		=
		\frac{1}{t^3}
		\int_{-\infty}^{x^1} \int_{-\infty}^{x^2} \int_{-\infty}^{x^3}
		(t \partial_{x^1}) (t \partial_{x^2}) (t \partial_{x^3}) \Big(\int_{\mathbb{R}^3}  h(\tilde{x},p) dp \Big)
		d\tilde{x}.
		\label{eq:Sobolevproof}
	\end{align}
	Now,
	\[
		t \partial_{x^i} \Big(\int_{\mathbb{R}^3}  h(x,p) dp \Big)
		=
		\int_{\mathbb{R}^3} L_i h dp
		+
		\frac{\log t}{t} \int_{\mathbb{R}^3} \partial_i \Delta \phi_{\infty}(p) h(x,p) dp,
	\]
	where
	\[
		\int_{\mathbb{R}^3} \partial_{p^i} h(x,p) dp = 0,
		\qquad
		\int_{\mathbb{R}^3} \partial_{p^k}\Big( \partial_i \partial_k \phi_{\infty}(p) h(x,p) \Big) dp = 0,
	\]
	is used.  The estimate \eqref{eq:L1Sobolev} follows from repeatedly applying to \eqref{eq:Sobolevproof}.
	
	Replacing $h$ with $h^2$ in \eqref{eq:L1Sobolev}, and using the fact that
	\[
		\sum_{\vert I \vert =0}^3 \big\vert L^I  \big( h(x,p)^2 \big) \big\vert
		\leq
		C
		\sum_{\vert I \vert =0}^3 \vert L^I h(x,p) \vert^2,
	\]
	the proof follows.
\end{proof}

\subsection{Constants} \label{subsec:constants}
In what follows a constant $B>0$ will be fixed and the assumption that
\begin{equation} \label{eq:suppfinfty}
	\supp(f_{\infty}) \subset \{ (x,p) \in \mathbb{R}^3 \times \mathbb{R}^3 \mid \vert x \vert + \vert p \vert \leq B\},
\end{equation}
will be made.  An integer $N\geq 6$, corresponding to the number of derivatives of the solution which are estimated, will also be assumed fixed.  Recall the approximate solution, $(f_{[K]}, \varrho_{[K]}, \phi_{[K]})$, introduced in \eqref{eq:introansatz1}--\eqref{eq:introansatz2}.  The notation $K \in \mathbb{N}$ will be used to denote the order of this approximate solution (also denoted $k \in \mathbb{N}$ in Section \ref{section:finiteprob}).

The notation
\begin{equation} \label{eq:lesssim}
	a \lesssim b,
\end{equation}
will be used when there is a constant $C>0$, which may depend on $B$ and $N$, such that
\[
	a \leq C b.
\]
This constant $C$ implicit in \eqref{eq:lesssim} is \underline{not} allowed to depend on $K$.  Any constants which depend on $K$ will be denoted $C_K$.

\section{Precise versions of the main results}
\label{section:theorem}

In this section, more precise versions of Theorem \ref{thm:intromain}, Theorem \ref{thm:intromain2} and Theorem \ref{thm:intromain3} are stated.

Define first, for any given sequence $\{\mathfrak{n}(n) \}_{n=0}^{\infty}$,
for any smooth compactly supported function $f_{\infty} \colon \mathbb{R}^3 \times \mathbb{R}^3 \to [0,\infty)$, and $n \geq 0$,
\begin{equation} \label{eq:defofcalF}
	\mathcal{F}_{\infty}^n
	=
	\mathcal{F}_{\infty}^n[f_{\infty},\{\mathfrak{n}(n) \}]
	=
	\sum_{\vert I \vert + \vert J \vert \leq n}
	\Big(
	\Vert \partial_x^I \partial_p^J f_{\infty} \Vert_{L^2_x L^2_p}
	+
	\Vert \partial_x^I \partial_p^J f_{\infty} \Vert_{L^2_x L^2_p}^{\mathfrak{n}(n)}
	\Big).
\end{equation}
The sequences considered will always be increasing and have the property that $\mathfrak{n}(n)\in \mathbb{N}$ and $\mathfrak{n}(n) > n$ for all $n$.
In order to ease notation, the dependence of $\mathcal{F}_{\infty}^n$ on $f_{\infty}$, and the sequence $\{\mathfrak{n}(n) \}$, is typically suppressed.  In expressions to follow, the sequence $\{\mathfrak{n}(n) \}$ may change from line to line.

Recall the function $\phi_{\infty} \colon \mathbb{R}^3 \to \mathbb{R}$, defined to be the unique solution of
\[
	\Delta \phi_{\infty} (p) = \varrho_{\infty}(p),
	\qquad
	\phi_{\infty} (p) \to 0, \quad \text{as } \vert p \vert \to \infty,
	\qquad
	\varrho_{\infty}(p) = - \int_{\mathbb{R}^3} f_{\infty} (y,p) dy.
\]
Note that, assuming that $f_{\infty}$ satisfies \eqref{eq:suppfinfty} for some $B>0$, for any $n \geq 2$, the Sobolev inequality (see Proposition \ref{prop:SobolevL2}) and the gradient estimate \eqref{eq:gradelliptic} imply that
\begin{equation} \label{eq:phiinftypointwise}
	\sum_{\vert I \vert =0}^{n-2} \Vert \partial^{I} \nabla \phi_\infty \Vert_{L^{\infty}}
	\lesssim
	\sum_{\vert I \vert =0}^{n} \Vert \partial^{I} \nabla \phi_\infty \Vert_{L^2}
	\lesssim
	\sum_{\vert I \vert =0}^{n} \Vert \partial^{I} \varrho_\infty \Vert_{L^2}
	\lesssim
	\mathcal{F}_{\infty}^{n}.
\end{equation}

Given functions $f_{k,l} \colon \mathbb{R}^3 \times \mathbb{R}^3 \to \mathbb{R}$ and $\varrho_{k,l} \colon \mathbb{R}^3 \to \mathbb{R}$ for $k = 0,1,2,\ldots$ and $l=0,\ldots k$, define $\phi_{k,l} \colon \mathbb{R}^3 \to \mathbb{R}$ to be the unique solutions of the Poisson equation sourced by $\varrho_{k,l}$,
\begin{equation} \label{eq:mainthmPoisson}
	\Delta_{\mathbb{R}^3} \phi_{k,l} = \varrho_{k,l},
	\qquad
	\phi_{k,l}(w) \to 0, \quad \text{as } \vert w \vert \to \infty,
\end{equation}
and define functions, for $K \geq 0$,
\begin{equation} \label{eq:fKdefthm}
	f_{[K]}(t,x,p)
	:=
	\sum_{k=0}^K \sum_{l=0}^{k} \frac{(\log t)^l}{t^k} f_{k,l}\big(x-tp + \log t \, \nabla \phi_{\infty} ( p), p \big)
	,
\end{equation}
\begin{equation} \label{eq:rhophiKdefthm}
	\varrho_{[K]}(t,x)
	:= 
	\frac{1}{t^3}
	\sum_{k=0}^K \sum_{l=0}^{k} \frac{(\log t)^l}{t^k} \varrho_{k,l} \left( \frac{x}{t} \right)
	,
	\quad
	\phi_{[K]}(t,x)
	:=
	\frac{1}{t}
	\sum_{k=0}^K \sum_{l=0}^{k} \frac{(\log t)^l}{t^k} \phi_{k,l} \left( \frac{x}{t} \right)
	.
\end{equation}

The following theorem is a precise statement of Theorem \ref{thm:intromain}.

\begin{theorem}[Inverse modified scattering map --- precise statement of Theorem \ref{thm:intromain}] \label{thm:backwardsproblem}
	Consider a smooth compactly supported function $f_{\infty} \colon \mathbb{R}^3 \times \mathbb{R}^3 \to [0,\infty)$, satisfying \eqref{eq:suppfinfty} for some $B>0$, and some $N \geq 6$.  There exists $k_* = k_*(\mathcal{F}_{\infty}^{N+1})$ large and $T_0 = T_0(\mathcal{F}_{\infty}^{N+2k_*+4}) \in \mathbb{R}$ large such that:
	\begin{itemize}
		\item \textbf{\emph{Existence of solution:}}
			There exists a solution $(f,\varrho,\phi)$ of the Vlasov--Poisson system \eqref{eq:VP1}--\eqref{eq:VP2} on $[T_0,\infty)\times \mathbb{R}^3_x \times \mathbb{R}^3_p$ which attains the scattering data $f_{\infty}$ in the sense that, for all $t \geq T_0$,
	\begin{equation} \label{eq:dataattained}
		\sum_{\vert I \vert + \vert J \vert \leq N}
		\Big(
		\int \int
		\Big\vert
		\partial_x^I \partial_p^J
		\Big(
		f\big(t,x+tp - \log t \, \nabla \phi_{\infty}(p), p \big)
		-
		f_{\infty}(x,p)
		\Big)
		\Big\vert^2
		dp dx
		\Big)^{\frac{1}{2}}
		\lesssim
		\frac{\mathcal{F}_{\infty}^{N+4} (\log t)^2}{t},
	\end{equation}
	and
	\begin{equation} \label{eq:dataattained2}
		\sup_{x,p\in \mathbb{R}^3}
		\sum_{\vert I \vert + \vert J \vert \leq N-4}
		\Big\vert
		\partial_x^I \partial_p^J
		\Big(
		f\big(t,x+tp - \log t \, \nabla \phi_{\infty}(p), p \big)
		-
		f_{\infty}(x,p)
		\Big)
		\Big\vert
		\lesssim
		\frac{\mathcal{F}_{\infty}^{N+4} (\log t)^2}{t}.
	\end{equation}
	For any $t \geq T_0$, $f$ satisfies the support property
	\begin{equation} \label{eq:suppfstatement}
		\supp(f(t,\cdot,\cdot))
		\subset
		\{
		(x,p) \subset \mathbb{R}^3 \times \mathbb{R}^3
		\mid
		\vert x - t p + \log t \nabla \phi_{\infty}(p) \vert \leq 2\mathcal{F}_{\infty}^3 + B, \vert p \vert \leq 2 B
		\},
	\end{equation}
	where $B>0$ is as in \eqref{eq:suppfinfty}.  
	\end{itemize}
\end{theorem}

The proof of Theorem \ref{thm:intromain} is based on a sequence of approximate solutions of the Vlasov--Poisson system, defined explicitly in terms of the scattering profile $f_{\infty}$.  The following theorem concerns the existence of these approximate solutions, along with their key properties.

\begin{theorem}[Explicit approximate solutions] \label{thm:backwardsproblem2}
	Under the assumptions of Theorem \ref{thm:backwardsproblem}:
	\begin{itemize}

		\item \textbf{\emph{Approximate solutions:}}
			There are sequences of smooth functions
			\[
				f_{k,l} \colon \mathbb{R}^3 \times \mathbb{R}^3 \to \mathbb{R},
				\qquad
				\varrho_{k,l} \colon \mathbb{R}^3 \to \mathbb{R},
			\]
			for $k = 0,1,2,\ldots$ and $l=0,\ldots k$ --- defined explicitly in terms of $f_{\infty}$ --- 
			such that $(f_{[K]}, \varrho_{[K]}, \phi_{[K]})$, defined by \eqref{eq:fKdefthm}--\eqref{eq:rhophiKdefthm}, is an \emph{approximate solution of order $K$}, in the sense that, for all $t \geq T_0$,
			\begin{align*}
				\sum_{\vert I \vert + \vert J \vert \leq N}
				\big\Vert L^I (t^{-1}\partial_p)^J \big( \gs_{\phi_{[K]}} f_{[K]} \big) (t,\cdot,\cdot) \big\Vert_{L^2_xL^2_p}
				&
				\leq
				C_{K} (\mathcal{F}_{\infty}^{N+2K+3})^2
				\frac{(\log t)^{1+K}}{t^{2+K}},
			\\
				\sum_{\vert I \vert \leq N}
				\big\Vert (t\partial_x)^I \Big( \int_{\mathbb{R}^3} f_{[K]} (t,\cdot,p) dp - \varrho_{[K]} (t,\cdot) \Big) \big\Vert_{L^2_x}
				&
				\leq
				C_{K} \mathcal{F}_{\infty}^{N+2K+3}
				\frac{(\log t)^{K+1}}{t^{\frac{5}{2}+K}},
			\end{align*}
			and the scattering data $f_{\infty}$ is attained in the sense \eqref{eq:dataattained2}.
			
			The functions $f_{k,l}$ and $\varrho_{k,l}$ satisfy the support property
			\[
				\supp(f_{k,l})
				\subset 
				\{ (x,p) \in \mathbb{R}^3 \times \mathbb{R}^3 \mid \vert x \vert \leq B, \vert p \vert \leq B \},
				\qquad
				\supp(\varrho_{k,l})
				\subset 
				\{ x \in \mathbb{R}^3 \mid \vert x \vert \leq B \},
			\]
			and the estimates
			\begin{equation} \label{eq:mainapproxthmfklrhokl}
				\sum_{\vert I \vert + \vert J \vert \leq N} \Vert \partial_x^I \partial_p^J f_{k,l} \Vert_{L^2_x L^2_p}
				\leq
				C_{k} \mathcal{F}_{\infty}^{N+k},
				\qquad
				\sum_{\vert I \vert + \vert J \vert \leq N}
				\Vert \partial_x^I \varrho_{k,l} \Vert_{L^2_x}
				+
				\Vert \partial_x^I \nabla \phi_{k,l} \Vert_{L^2_x}
				\leq
				C_{k} \mathcal{F}_{\infty}^{N+k}.
			\end{equation}
		
	\end{itemize}
\end{theorem}

The functions $f_{k,l}$ and $\varrho_{k,l}$ are defined in Section \ref{section:expansions}, where the proof of Theorem \ref{thm:backwardsproblem2} is given.  See in particular Theorem \ref{thm:approx}, which also contains further properties of the approximate solutions.

The next theorem is a precise statement of Theorem \ref{thm:intromain2}, on the polyhomogeneous expansions satisfied by the solutions of Theorem \ref{thm:backwardsproblem}.

\begin{theorem}[Polyhomogeneous expansions --- precise statement of Theorem \ref{thm:intromain2}] \label{thm:backwardsproblem3}
	Under the assumptions of Theorem \ref{thm:backwardsproblem}:
	\begin{itemize}
			
		\item \textbf{\emph{Polyhomogeneous expansion in $L^2$:}}
			The approximate solution $(f_{[K]}, \varrho_{[K]}, \phi_{[K]})$ of Theorem \ref{thm:backwardsproblem2} agrees with true solution, of Theorem \ref{thm:backwardsproblem}, to order $K$ in the sense that, for any $K \in \mathbb{N}$ and $t\geq T_0$, the difference satisfies the $L^2$ estimates
	\begin{align} \label{eq:mainest1}
		\sum_{\vert I \vert + \vert J \vert \leq N}
		\Vert L^I (t^{-1}\partial_p)^J (f - f_{[K]}) (t,\cdot,\cdot) \Vert_{L^2_x L^2_p}
		&
		\leq
		C_K
		(\mathcal{F}_{\infty}^{N+2K+4})^2
		\frac{ (\log t)^{1+K}
		}{t^{1+ K}},
	\\
	\label{eq:mainest2}
		\sum_{\vert I \vert \leq N}
		\Vert (t\partial_x)^I ( \varrho - \varrho_{[K]}) (t,\cdot) \Vert_{L^2_x}
		&
		\leq
		C_{K} \mathcal{F}_{\infty}^{N+2K+4}
		\frac{(\log t)^{1+K}}{t^{\frac{5}{2} + K}},
	\\
	\label{eq:mainest3}
		\sum_{\vert I \vert \leq N}
		\Vert (t\partial_x)^I \nabla (\phi - \phi_{[K]}) (t,\cdot) \Vert_{L^2}
		&
		\leq
		C_{K} \mathcal{F}_{\infty}^{N+2K+4}
		\frac{(\log t)^{1+K}}{t^{\frac{3}{2} + K}}.
	\end{align}

		\item \textbf{\emph{Polyhomogeneous expansion in $L^{\infty}$:}}
			For any $K \in \mathbb{N}$ and $t\geq T_0$, the difference moreover satisfies the $L^{\infty}$ estimates
			\begin{align}
				\sum_{\vert I \vert + \vert J \vert \leq N-4}
				\sup_{x,p\in\mathbb{R}^3} \vert L^I (t^{-1}\partial_p)^J (f - f_{[K]}) (t,x,p) \vert
				&
				\leq
				C_K
				(\mathcal{F}_{\infty}^{N+2K+4})^2
				\frac{ (\log t)^{1+K}
				}{t^{1+ K}},
				\label{eq:mainest4}
				\\
				\sum_{\vert I \vert \leq N-2}
				\sup_{x\in\mathbb{R}^3} \vert (t\partial_x)^I ( \varrho - \varrho_{[K]}) (t,x) \vert
				&
				\leq
				C_{K} \mathcal{F}_{\infty}^{N+2K+4}
				\frac{(\log t)^{1+K}}{t^{4 + K}},
				\label{eq:mainest5}
				\\
				\sum_{\vert I \vert \leq N-2}
				\sup_{x\in\mathbb{R}^3} \vert (t\partial_x)^I \nabla (\phi - \phi_{[K]}) (t,x) \vert
				&
				\leq
				C_{K} \mathcal{F}_{\infty}^{N+2K+4}
				\frac{(\log t)^{1+K}}{t^{3 + K}}.
				\label{eq:mainest6}
			\end{align}
	
	\end{itemize}
\end{theorem}

Finally, the next theorem is a more precise statement of Theorem \ref{thm:intromain3}, on the uniqueness of the solutions of Theorem \ref{thm:backwardsproblem}.

\begin{theorem}[Uniqueness of modified scattering solutions --- precise statement of Theorem \ref{thm:intromain3}] 
\label{thm:backwardsproblem4}
	Under the assumptions of Theorem \ref{thm:backwardsproblem}:
	\begin{itemize}
		\item \textbf{\emph{Uniqueness of solution:}}
			The solution $(f,\varrho,\phi)$ of Theorem \ref{thm:backwardsproblem} is unique in the class of solutions which agree with the approximate solution $(f_{[K]}, \varrho_{[K]}, \phi_{[K]})$ of Theorem \ref{thm:backwardsproblem2} for sufficiently large $K$.  More precisely, if $(f',\varrho',\phi')$ is another solution such that there is a constant $C$ satisfying
	\begin{equation} \label{eq:uniquenesscondition}
		\Vert (f' - f_{[K]}) (t,\cdot,\cdot) \Vert_{L^2_x L^2_p}
		\leq
		C
		\frac{ (\log t)^{1+K}
		}{t^{1+ K}},
	\end{equation}
	for some $K=K(\mathcal{F}_{\infty}^{N+1})$ and for all $t$ sufficiently large, and moreover $f'$ satisfies the support property \eqref{eq:suppfstatement}, then $(f',\varrho',\phi') = (f,\varrho,\phi)$.
	\end{itemize}
\end{theorem}

The proofs of Theorem \ref{thm:backwardsproblem} and Theorem \ref{thm:backwardsproblem3} are based on a sequence of ``finite problems'', which are discussed in Section \ref{section:finiteprob}.  The proofs of Theorem \ref{thm:backwardsproblem} and Theorem \ref{thm:backwardsproblem3}, using the results of Section \ref{section:finiteprob}, are then given in Section \ref{section:logicofproof}, where the proof of Theorem \ref{thm:backwardsproblem4} is also given.

The reader interested only in seeing how the approximate solutions are used to establish Theorem \ref{thm:backwardsproblem}, Theorem \ref{thm:backwardsproblem3} and Theorem \ref{thm:backwardsproblem4}, who is willing to accept the statement of Theorem \ref{thm:backwardsproblem2} (or rather the version with additional statements, in Theorem \ref{thm:approx} below) without seeing the details of the construction, may skip ahead to Section \ref{section:finiteprob}.

\section{Approximate solutions}
\label{section:expansions}

In this section explicit definitions of $f_{k,l} \colon \mathbb{R}^3 \times \mathbb{R}^3 \to \mathbb{R}$ and $\varrho_{k,l} \colon \mathbb{R}^3 \to \mathbb{R}$ are given, for $k = 0,1,2,\ldots$ and $l=0,\ldots k$, and the corresponding $(f_{[K]},\varrho_{[K]}, \phi_{[K]})$, given by \eqref{eq:mainthmPoisson}--\eqref{eq:rhophiKdefthm}, are shown to be approximate solutions to the Vlasov--Poisson system \eqref{eq:VP1}--\eqref{eq:VP2}.

Recall the definition \eqref{eq:defofcalF} of $\mathcal{F}_{\infty}^n[f_{\infty},\{\mathfrak{n}(n) \}]$ and the convention that the dependence on the sequence $\{\mathfrak{n}(n)\}$ is omitted, and moreover that this omitted sequence can change from line to line.  This convention is used throughout this section.

The main result of this section is the following.

\begin{theorem}[Approximate solutions] \label{thm:approx}
	Let $N\geq 6$ be fixed and consider a smooth compactly supported function $f_{\infty} \colon \mathbb{R}^3 \times \mathbb{R}^3 \to [0,\infty)$ satisfying \eqref{eq:suppfinfty} for some $B>0$.  There exists $T_0>0$ and smooth functions $f_{k,l} \colon \mathbb{R}^3 \times \mathbb{R}^3 \to \mathbb{R}$ and $\phi_{k,l}, \varrho_{k,l} \colon \mathbb{R}^3 \to \mathbb{R}$ for $k = 0,1,2,\ldots$ and $l=0,\ldots k$, defined explicitly in terms of $f_{\infty}$ with
	\[
		f_{0,0} = f_{\infty}, \qquad \varrho_{0,0} = \varrho_{\infty}, \qquad \phi_{0,0} = \phi_{\infty},
	\]
	such that:
	\begin{itemize}
		\item
			The functions $f_{[K]}$, $\varrho_{[K]}$ and $\phi_{[K]}$, defined by \eqref{eq:mainthmPoisson}--\eqref{eq:rhophiKdefthm}, are an approximate solution of the Vlasov--Poisson system \eqref{eq:VP1}--\eqref{eq:VP2} of order $K$, in the sense that, for any $K \geq 1$ and for all $t \geq T_0$,
	\begin{equation} \label{eq:thmapprox1}
		\sum_{\vert I \vert + \vert J \vert \leq N}
		\big\Vert L^I (t^{-1}\partial_p)^J \big( \gs_{\phi_{[K]}} f_{[K]} \big) (t,\cdot,\cdot) \big\Vert_{L^2_xL^2_p}
		\leq
		C_{K} (\mathcal{F}_{\infty}^{N+2K+3})^2
		\frac{(\log t)^{1+K}}{t^{2+K}},
	\end{equation}
	\begin{equation} \label{eq:thmapprox2}
		\sum_{\vert I \vert \leq N}
		\big\Vert (t\partial_x)^I \Big( \int_{\mathbb{R}^3} f_{[K]} (t,\cdot,p) dp - \varrho_{[K]} (t,\cdot) \Big) \big\Vert_{L^2_x}
		\leq
		C_{K} \mathcal{F}_{\infty}^{N+2K+3}
		\frac{(\log t)^{K+1}}{t^{\frac{5}{2}+K}}.
	\end{equation}

		\item
			The functions $f_{[K]}$ moreover satisfy the esimtates, for all $t\geq T_0$,
	\begin{align}
		\sum_{\vert I \vert + \vert J \vert \leq N}
		\big\Vert L^I (t^{-1}\partial_p)^J f_{[K]} (t,\cdot,\cdot) \big\Vert_{L^2_xL^2_p}
		&
		\leq
		\mathcal{F}_{\infty}^{N}
		+
		\frac{C_{K} \mathcal{F}_{\infty}^{N+K} \log t}{t},
		\label{eq:fKestimate}
	\\
		\sum_{\vert I \vert + \vert J \vert \leq N}
		\big\Vert L^I (t^{-1}\partial_p)^J (f_{[K]} - f_{[0]}) (t,\cdot,\cdot) \big\Vert_{L^2_xL^2_p}
		&
		\leq
		\frac{C_{K} \mathcal{F}_{\infty}^{N+K} \log t}{t}.
		\label{eq:fKestimate2}
	\end{align}

		\item
			Finally, each $f_{k,l}$ and each $\varrho_{k,l}$ satisfies
	\begin{equation} \label{eq:fklsupport}
		\supp(f_{k,l})
		\subset 
		\{ (x,p) \in \mathbb{R}^3 \times \mathbb{R}^3 \mid \vert x \vert \leq B, \vert p \vert \leq B \},
		\qquad
		\supp(\varrho_{k,l})
		\subset 
		\{ x \in \mathbb{R}^3 \mid \vert x \vert \leq B \},
	\end{equation}
	along with the estimate, for $k \geq 1$,
	\begin{equation} \label{eq:flkL2estimate}
		\sum_{\vert I \vert + \vert J \vert \leq N} \Vert \partial_x^I \partial_p^J f_{k,l} \Vert_{L^2_x L^2_p}
		\leq
		C_{k} (\mathcal{F}_{\infty}^{N+k})^2,
		\qquad
				\sum_{\vert I \vert + \vert J \vert \leq N}
				\Vert \partial_x^I \varrho_{k,l} \Vert_{L^2_x}
				+
				\Vert \partial_x^I \nabla \phi_{k,l} \Vert_{L^2_x}
				\leq
				C_{k} \mathcal{F}_{\infty}^{N+k}.
	\end{equation}

	\end{itemize}
\end{theorem}

Throughout this section, it is supposed that a smooth compactly supported function $f_{\infty} \colon \mathbb{R}^3 \times \mathbb{R}^3 \to [0,\infty)$ is given, and $B>0$ is such that \eqref{eq:suppfinfty} holds.  It is also assumed that $N\geq 6$ is fixed.

In Section \ref{subsec:notation} notation is introduced which is used throughout.

In Section \ref{subsec:fexpansion}, for the purpose of an inductive argument, the definitions of $f_{k,l}$ are given for $k \leq K$, under the assumption that smooth functions $\phi_{k,l}$ are given for all $k \leq K$.  The properties \eqref{eq:thmapprox1}, \eqref{eq:fKestimate}--\eqref{eq:fKestimate2} and the former of \eqref{eq:fklsupport} and \eqref{eq:flkL2estimate} are established, assuming that $\phi_{k,l}$ satisfy the latter of \eqref{eq:flkL2estimate}.  These properties hold independently of how $\phi_{k,l}$ are defined.

In Section \ref{subsec:rhoexpansion}, it is assumed, for the purpose of the inductive argument, that $K\geq 1$ is such that functions $f_{k,l}$ are given for $k \leq K-1$ and satisfy the former of \eqref{eq:fklsupport} and \eqref{eq:flkL2estimate}.  Functions $f_{K,l}$ are defined, along with functions $\varrho_{k,l}$ for all $k \leq K$.  These functions are shown to satisfy \eqref{eq:thmapprox2}, along with the latter of \eqref{eq:fklsupport} and \eqref{eq:flkL2estimate}.  Again, these properties hold independently of how $f_{k,l}$, for $k \leq K-1$, are defined.

Finally the proof of Theorem \ref{thm:approx} is given in Section \ref{subsec:thmapprox} using the results of Section \ref{subsec:fexpansion} and Section \ref{subsec:rhoexpansion} together with an induction argument.

The reader is referred to the discussion in Section \ref{subsec:introexpansion} for explicit treatments of the $K=0$ and $K=1$ cases of Theorem \ref{thm:approx}, and a guide to the $K \geq 2$ case.

\subsection{Notation}
\label{subsec:notation}

The following notation will be used throughout this section.
First, $y\colon [T_0,\infty) \times \mathbb{R}^3 \times \mathbb{R}^3 \to \mathbb{R}^3$ will denote the function
\begin{equation} \label{eq:defofy}
	y(t,x,p) = x - tp + \log t \nabla \phi_{\infty}(p).
\end{equation}
Next, for any $X\in \mathbb{R}^3$, define
\[
	\otimes_k X := \underbrace{X \otimes \ldots \otimes X}_{k \text{ times}},
\]
and for $X_1, \ldots, X_k \in \mathbb{R}^3$, with $X_1 = (X_1^1, X_1^2, X_1^3)$ etc.\@, and any function $\Phi \colon \mathbb{R}^3 \to \mathbb{R}^3$, define
\[
	X_1 \otimes \ldots \otimes X_k \cdot \nabla^k \Phi
	:=
	\sum_{i_1,\ldots, i_k=1}^3 X_1^{i_1} \ldots X_k^{i_k} \partial_{i_1} \ldots \partial_{i_k} \Phi
	\in \mathbb{R}^3.
\]
For a given subset $A\subset \mathbb{N}^d$, let $\iota_A$ be the indicator function of this set
\begin{equation} \label{eq:iotadef}
	\iota_A(\underline{k})
	=
	\begin{cases}
		1, & \text{if $\underline{k} \in A$},\\
		0, & \text{otherwise}.
	\end{cases}
\end{equation}

\subsection{The approximate solution $f_{[K]}$}
\label{subsec:fexpansion}

For the purpose of an inductive argument, given later in Section \ref{subsec:thmapprox}, suppose that $K\geq 1$ is such that some smooth functions $\phi_{k,l}$ are given, for $k=0,\ldots,K$, $l=0,\ldots,k$, with $\phi_{0,0}=\phi_{\infty}$, which satisfy, for some constants $C_{n,k}$, for all $n \geq 6$,
\begin{equation} \label{eq:phiklassum}
	\sum_{\vert I \vert \leq n} \Vert \partial^I \nabla \phi_{k,l} \Vert_{L^2}
	\leq
	C_{n,k} \mathcal{F}_{\infty}^{n+k}.
\end{equation}
Let $\phi_{[K]}$ be defined in terms of these functions by \eqref{eq:rhophiKdefthm}.
In this section the functions $f_{k,l}$ are defined in terms of $\phi_{k,l}$ and it is shown that, under the assumption that $\phi_{k,l}$ satisfy \eqref{eq:phiklassum}, the Vlasov equation is satisfied by $f_{[K]}$ and $\phi_{[K]}$ to order $K$, in the sense that the estimate \eqref{eq:thmapprox1} holds.

First it is shown that, for $p\in \mathbb{R}^3$, if the function $y(t,x,p)$ is bounded then, in the coefficients of $\nabla \phi_{[K]}(t,x)$ --- i.\@e.\@ in $\nabla \phi_{k,l}(x/t)$ --- the argument can be replaced using \eqref{eq:defofy} and Taylor expanded to give a series whose coefficients are functions of $y(t,x,p)$ and $p$.  See \eqref{eq:intronablaphiexpansion} for the example of the $K=1$ case.

Define smooth functions $\Psi_{k,l} \colon \mathbb{R}^3 \times \mathbb{R}^3 \to \mathbb{R}^3$, for $k = 0,\ldots, K$, $l= 0, \ldots, k$, by
\begin{align} \label{eq:Psikl1}
	\Psi_{0,0}(y,p)
	=
	\nabla \phi_{\infty} (p),
	\quad
	\Psi_{1,1}(y,p)
	=
	\nabla \phi_{1,1}(p) - \nabla \phi_{\infty}(p) \cdot \nabla^{2} \phi_{\infty} (p),
	\quad
	\Psi_{1,0}(y,p)
	=
	\nabla \phi_{1,0}(p) + y \cdot \nabla^{2} \phi_{\infty} (p),
\end{align}
and, for $k \geq 2$,
\begin{equation} \label{eq:Psikl2}
	\Psi_{k,l}(y,p)
	=
	\sum_{\substack{ k_1+k_2=k \\ l_1+l_2 = l}}
	\binom{k_2}{l_2} \frac{(-1)^{l_2}}{k_2!} \otimes_{k_2-l_2} y \otimes_{l_2} \nabla \phi_{\infty}(p) \cdot \nabla^{k_2+1} \phi_{k_1,l_1} (p).
\end{equation}
where, in all summations in this section, unless explicitly stated otherwise, $0\leq l_1 \leq k_1$ and $0\leq l_2 \leq k_2$.  It is helpful to also rewrite the expression \eqref{eq:Psikl2} with the $k_1=k$ term isolated,
\begin{equation} \label{eq:Psikl2iso}
	\Psi_{k,l}(y,p)
	=
	\nabla \phi_{k,l}(p)
	+
	\sum_{\substack{ k_1+k_2=k \\ l_1+l_2 = l \\ k_1 \leq k-1}}
	\binom{k_2}{l_2} \frac{(-1)^{l_2}}{k_2!} \otimes_{k_2-l_2} y \otimes_{l_2} \nabla \phi_{\infty}(p) \cdot \nabla^{k_2+1} \phi_{k_1,l_1} (p).
\end{equation}

\begin{proposition}[Expansion for $\nabla \phi(t,x)$] \label{prop:phiderivexpansion}
	Consider some $K\geq 0$ and smooth functions $\phi_{k,l} \colon \mathbb{R}^3 \to \mathbb{R}$ satisfying \eqref{eq:phiklassum} for all $k=0,\ldots,K$, $l=0,\ldots,k$.
	The explicit functions $\Psi_{k,l} \colon \mathbb{R}^3 \to \mathbb{R}^3$, defined by \eqref{eq:Psikl1}--\eqref{eq:Psikl2}, for $k = 0,1,2,\ldots$, $l= 0, \ldots, k$, satisfy, 
	for any $p \in \mathbb{R}^3$, any $B > 0$, and any $t \geq T_0$,
	\begin{multline} \label{eq:phiPsibound}
		\sup_{\vert y(t,x,p) \vert \leq B}
		\sum_{\vert I \vert + \vert J \vert \leq N}
		\Big\vert
		L^I (t^{-1}\partial_p)^J
		\Big(
		\nabla_{x} \phi_{[K]}(t,x)
		-
		\frac{1}{t^2}
		\sum_{k=0}^K \sum_{l = 0}^k \frac{(\log t)^l}{t^k} \Psi_{k,l}(y(t,x,p),p)
		\Big)
		\Big\vert
		\\
		\leq
		C_{K}
		\mathcal{F}_{\infty}^{N+2K+3}
		\frac{(\log t)^{K+1}}{t^{K+3}}
		.
	\end{multline}
	Moreover, for all $n \geq 2$,
	\begin{equation} \label{eq:Psiklbound}
		\sum_{\vert I \vert \leq n} \Vert \partial^I \Psi_{k,l} \Vert_{L^2_x L^2_p}
		\leq
		C_{n,k} \mathcal{F}_{\infty}^{n+k}.
	\end{equation}
\end{proposition}

\begin{proof}
	First note that
	\begin{equation} \label{eq:phiKxovert}
		\nabla \phi_{[K]}(t,x)
		=
		\frac{1}{t^2}
		\Big[
		\nabla \phi_{\infty} \left( \frac{x}{t} \right)
		+
		\sum_{k=1}^K \sum_{l=0}^{k} \frac{(\log t)^l}{t^k} \nabla \phi_{k,l} \left( \frac{x}{t} \right)
		\Big].
	\end{equation}
	Note also that the assumption \eqref{eq:phiklassum}, along with the Sobolev inequality of Proposition \ref{prop:SobolevL2}, implies that, for any $n\geq 2$,
	\[
		\sum_{\vert I \vert \leq n-2} \Vert \partial^I \nabla \phi_{k,l} \Vert_{L^{\infty}}
		\leq
		C_{n,k} \mathcal{F}_{\infty}^{n+k}.
	\]
	Now,
	\[
		\frac{x}{t} =  p + \frac{y}{t} - \frac{\log t}{t} \nabla \phi_{\infty}(p),
	\]
	and so, for any function $\Phi \colon \mathbb{R}^3 \to \mathbb{R}^3$, by Taylor's Theorem (see Proposition \ref{prop:Taylor}) and the Binomial Theorem,
	\begin{multline*}
		\sup_{\vert y(t,x,p) \vert \leq B}
		\Big\vert
		\Phi \left( \frac{x}{t} \right)
		-
		\Big[
		\Phi(p)
		+
		\sum_{k=1}^K \sum_{l=0}^k \frac{(\log t)^l}{t^k} \binom{k}{l} \frac{(-1)^l}{k!} \otimes_{k-l} y \otimes_l \nabla \phi_{\infty}(p) \cdot \nabla^k \Phi (p)
		\Big]
		\Big\vert
		\\
		\lesssim
		(B + \Vert \nabla \phi_{\infty} \Vert_{L^{\infty}})^{K+1}
		\Vert \nabla^{K+1} \Phi \Vert_{L^{\infty}}
		\left( \frac{\log t}{t} \right)^{K+1}.
	\end{multline*}
	Using this expression with $\Phi = \nabla \phi_{\infty}$ and $\Phi = \nabla \phi_{k,l}$, for $k =1,\ldots,K$, $l=0,\ldots,k$, after inserting into \eqref{eq:phiKxovert} it follows that there is a sequence $\{\mathfrak{n}(n)\}$ such that
	\begin{multline*}
		\bigg\vert
		\nabla \phi_{[K]}(t,x)
		-
		\frac{1}{t^2}
		\sum_{k=0}^K \sum_{l=0}^{k} \frac{(\log t)^l}{t^k} 
		\sum_{\substack{ k_1+k_1=k \\ l_2+l_2 = l}}
		\binom{k_2}{l_2} \frac{(-1)^{l_2}}{k_2!} \otimes_{k_2-l_2} y \otimes_{l_2} \nabla \phi_{\infty}(p) \cdot \nabla^{k+1} \phi_{k_1,l_1} (p)
		\bigg\vert
		\\
		\lesssim
		\mathcal{F}_{\infty}^{2K+3}
		\frac{(\log t)^{K+1}}{t^{K+3}},
	\end{multline*}
	and the zeroth order part of the proof of \eqref{eq:phiPsibound} follows.  The estimates for the higher order derivatives follow similarly, using Remark \ref{rmk:functionsofy} and the fact that $L_i(x^j/t) = \delta_i^j$ for all $i,j=1,2,3$.

	Finally, the proof of \eqref{eq:Psiklbound} follows from the assumption \eqref{eq:phiklassum} and the Sobolev inequality of Proposition \ref{prop:SobolevL2}, using the inductive definitions \eqref{eq:Psikl1}--\eqref{eq:Psikl2}.
\end{proof}

We will now give explicit definitions of $f_{k,l}$, in terms of $\Psi_{k,l}$, to arrange that $\gs_{\phi_{[K]}} f_{[K]} $ vanishes to sufficiently high order, in the sense that \eqref{eq:thmapprox1} holds (see also \eqref{eq:introiteratedef1}--\eqref{eq:introiteratedef2} and the surrounding discussion for the $k=1$ case).

\begin{proposition}[Properties of $f_{k,l}$] \label{prop:fKdef}
	Consider some $K\geq 0$ and smooth functions $\phi_{k,l} \colon \mathbb{R}^3 \to \mathbb{R}$ satisfying \eqref{eq:phiklassum} for all $k=0,\ldots,K$, $l=0,\ldots,k$.
	The functions $f_{k,l} \colon \mathbb{R}^3 \times \mathbb{R}^3 \to \mathbb{R}$ defined by \eqref{eq:fklpsi1}--\eqref{eq:fklpsi3} below (with the functions $\Psi_{k,l}$ defined explicitly in terms of $\phi_{k,l}$ by \eqref{eq:Psikl1}--\eqref{eq:Psikl2}), satisfy,
	\begin{equation} \label{eq:fklsupportagain}
		\supp(f_{k,l})
		\subset 
		\{ (x,p) \in \mathbb{R}^3 \times \mathbb{R}^3 \mid \vert x \vert \leq B, \vert p \vert \leq B \},
	\end{equation}
	and, for all $k \geq 1$ and any $n \geq 6$,
	\begin{equation} \label{eq:fklestimatesagain}
		\sum_{\vert I \vert + \vert J \vert \leq n} \Vert \partial_x^I \partial_p^J f_{k,l} \Vert_{L^2_x L^2_p}
		\leq
		C_{n,k} (\mathcal{F}_{\infty}^{n+k})^2.
	\end{equation}
	For any $t \geq T_0$, the function $f_{[K]} \colon [T_0,\infty) \times \mathbb{R}^3_x \times \mathbb{R}^3_p \to \mathbb{R}$, defined by \eqref{eq:fKdefthm}, satisfies
	\begin{align}
		\sum_{\vert I \vert + \vert J \vert \leq N}
		\big\Vert L^I (t^{-1}\partial_p)^J f_{[K]} (t,\cdot,\cdot) \big\Vert_{L^2_xL^2_p}
		&
		\leq
		\mathcal{F}_{\infty}^{N}
		+
		\frac{C_{K} \mathcal{F}_{\infty}^{N+K} \log t}{t},
		\label{eq:fKestimateagain}
	\\
		\sum_{\vert I \vert + \vert J \vert \leq N}
		\big\Vert L^I (t^{-1}\partial_p)^J (f_{[K]} - f_{[0]}) (t,\cdot,\cdot) \big\Vert_{L^2_xL^2_p}
		&
		\leq
		\frac{C_{K} \mathcal{F}_{\infty}^{N+K} \log t}{t}.
		\label{eq:fKestimateagain2}
	\end{align}
	and
	\begin{equation} \label{eq:fKdefestimate}
		\sum_{\vert I \vert + \vert J \vert \leq N}
		\big\Vert L^I (t^{-1}\partial_p)^J \big( \gs_{\phi_{[K]}} f_{[K]} \big) (t,\cdot,\cdot) \Big\Vert_{L^2_xL^2_p}
		\leq
		C_{K} (\mathcal{F}_{\infty}^{N+2K+3})^2
		\frac{(\log t)^{1+K}}{t^{2+K}}.
	\end{equation}
\end{proposition}

The functions $f_{k,l} \colon \mathbb{R}^3 \times \mathbb{R}^3 \to \mathbb{R}$, for $k = 0,\ldots, K$, $l= 0, \ldots, k$, are defined by
\begin{equation} \label{eq:fklpsi1}
	f_{0,0}(y,p) = f_{\infty}(y,p),
\end{equation}
and, inductively, for $1 \leq k \leq K$,
\begin{align}
	f_{k,k}(y,p)
	=
	\
	&
	-
	\frac{1}{k}
	\sum_{\substack{
	k_1+k_2=k \\ l_1+l_2=k \\ k_1 \leq k-1
	}}
	(\partial_{x^i} f_{k_1,l_1})(y,p) \Psi^i_{k_2,l_2}(y,p)
	+
	\frac{1}{k}
	\sum_{\substack{
	k_1+k_2=k-1 \\ l_1+l_2=k -1
	}}
	(\partial_{x^i} f_{k_1,l_1})(y,p) \Psi^j_{k_2,l_2}(y,p) \partial_i \partial_j \phi_{\infty}(p)
	,
	\label{eq:fklpsi2}
\end{align}
and, for $l=k-1,\ldots,0$,
\begin{align} \label{eq:fklpsi3}
	&
	f_{k,l}(y,p)
	=
	\frac{l+1}{k} f_{k,l+1}(y,p)
	-
	\frac{1}{k}
	\sum_{\substack{
	k_1+k_2=k \\ l_1+l_2=l \\ k_1 \leq k-1
	}}
	(\partial_{x^i} f_{k_1,l_1})(y,p) \Psi^i_{k_2,l_2}(y,p)
	\\
	\
	&
	+
	\frac{\iota_{l\geq 1}}{k}
	\sum_{\substack{
	k_1+k_2=k-1 \\ l_1+l_2=l -1
	}}
	(\partial_{x^i} f_{k_1,l_1})(y,p) \Psi^j_{k_2,l_2}(y,p) \partial_i \partial_j \phi_{\infty}(p)
	+
	\frac{1}{k}
	\sum_{\substack{
	k_1+k_2=k-1 \\ l_1+l_2=l
	}}
	(\partial_{p^i} f_{k_1,l_1})(y,p) \Psi^j_{k_2,l_2}(y,p) \partial_i \partial_j \phi_{\infty}(p)
	,
	\nonumber
\end{align}
where $\iota$ is defined by \eqref{eq:iotadef}.
Note that $k_1=k$ is excluded from each of the summations in \eqref{eq:fklpsi2} and \eqref{eq:fklpsi3}, and so the definitions are indeed inductive.

The relevant aspect of the form of these expressions --- apart from the fact that they achieve exact cancellations in the following proof --- is explained in Remark \ref{rmk:totalyderiv} below.

\begin{proof}[Proof of Proposition \ref{prop:fKdef}]
	The proof of \eqref{eq:fklsupportagain} and \eqref{eq:fklestimatesagain} follow from the definitions \eqref{eq:fklpsi1}--\eqref{eq:fklpsi3} by a straightforward induction argument, using \eqref{eq:Psiklbound} and the Sobolev inequality of Proposition \ref{prop:SobolevL2R6} for \eqref{eq:fklestimatesagain}.

	The estimates \eqref{eq:fKestimateagain} and \eqref{eq:fKestimateagain2} follow from the properties \eqref{eq:functionsofy1}--\eqref{eq:functionsofy2} of the vector fields $L_i$ and $t^{-1} \partial_{p^i}$ acting on functions of $y$ and $p$, and the Sobolev inequality of Proposition \ref{prop:SobolevL2R6}, using also \eqref{eq:Psiklbound} and \eqref{eq:fklestimatesagain}.

	For \eqref{eq:fKdefestimate}, note that, with $y(t,x,p)$ defined by \eqref{eq:defofy},
	\[
		\gs_{\phi_{[K]}} (y(t,x,p)^j)
		=
		t \Big( \frac{1}{t^2} \partial_i \phi_{\infty}(p) - \partial_{x^i} \phi_{[K]}(t,x) \Big) + \log t \ \partial_{x^i} \phi_{[K]}(t,x) \partial_i \partial_j \phi_{\infty} (p),
	\]
	and thus,
	\begin{align}
		&
		\gs_{\phi_{[K]}} f_{[K]}
		=
		\sum_{k=0}^K \sum_{l=0}^{k}
		\bigg[
		\frac{l(\log t)^{l-1} - k (\log t)^l}{t^{k+1}}
		f_{k,l}(y,p)
		+
		\frac{(\log t)^l}{t^k}
		\partial_{x^i} \phi_{[K]}(t,x) (\partial_{p^i} f_{k,l})(y,p)
		\nonumber
		\\
		&
		\qquad \qquad \qquad
		+
		\frac{(\log t)^l}{t^k}
		\Big[
		t \Big( \frac{1}{t^2} \partial_i \phi_{\infty}(p) - \partial_{x^i} \phi_{[K]}(t,x) \Big) + \log t \ \partial_{x^i} \phi_{[K]}(t,x) \partial_i \partial_j \phi_{\infty} (p)
		\Big]
		(\partial_{x^i} f_{k,l})(y,p)
		\bigg]
		\nonumber
		\\
		&
		=
		\sum_{k=0}^{2K+1} \sum_{l=0}^{k}
		\frac{(\log t)^l}{t^{k+1}}
		\bigg[
		\Big( - k f_{k,l}(y,p) + (l+1) f_{k,l+1}(y,p) \Big) \iota_{k \leq K}
		\label{eq:gsKfK}
		\\
		&
		\
		-
		\sum_{\substack{
		k_1+k_2=k \\ l_1+l_2=l \\ k_2 \geq 1
		}}
		\iota_{k_1,k_2 \leq K}
		(\partial_{x^i} f_{k_1,l_1})(y,p) \Psi^i_{k_2,l_2}(y,p)
		+
		\iota_{l \geq 1} \!\!
		\sum_{\substack{
		k_1+k_2=k-1 \\ l_1+l_2=l -1
		}}
		\iota_{k_1,k_2 \leq K}
		(\partial_{x^i} f_{k_1,l_1})(y,p) \Psi^j_{k_2,l_2}(y,p) \partial_i \partial_j \phi_{\infty}(p)
		\nonumber
		\\
		&
		\qquad
		+
		\iota_{l \neq k}
		\sum_{\substack{
		k_1+k_2=k-1 \\ l_1+l_2=l
		}}
		\iota_{k_1,k_2 \leq K}
		(\partial_{p^i} f_{k_1,l_1})(y,p) \Psi^j_{k_2,l_2}(y,p) \partial_i \partial_j \phi_{\infty}(p)
		\bigg]
		\nonumber
		\\
		&
		\qquad
		+
		\Big(\partial_{x^i} \phi_{[K]}(t,x)
		-
		\frac{1}{t^2}
		\sum_{k=0}^K \sum_{l = 0}^k \frac{(\log t)^l}{t^k} \Psi_{k,l}^i(y,p)
		\Big)
		\partial_{p^i} f_{[K]}(t,x,p).
		\nonumber
	\end{align}
	After inserting the definitions \eqref{eq:fklpsi1}--\eqref{eq:fklpsi3}, the terms in the first summation corresponding to $k=0,\ldots,K$ all vanish.
	Recall now \eqref{eq:Psiklbound} and \eqref{eq:fklestimatesagain}.  It follows from the properties \eqref{eq:functionsofy1}--\eqref{eq:functionsofy2} of the vector fields $L_i$ and $t^{-1} \partial_{p^i}$ acting on functions of $y$ and $p$, and the Sobolev inequality of Proposition \ref{prop:SobolevL2R6}, that, for any $k\geq K+1$ and $l \leq k$ and $k_1,k_2 \leq K$, $l_1,l_2$ such that $k_1+k_2 = k$, $l_1+l_2=l$, and any $N \geq 6$, 
	\[
		\sum_{\vert I \vert + \vert J \vert \leq N}
		\Vert
		L^I (t^{-1} \partial_p)^J \big( \frac{(\log t)^l}{t^{k+1}} \partial_{x^i} f_{k_1,l_1}(y,p) \Psi^j_{k_2,l_2}(y,p) \big)
		\Vert_{L^2_x L^2_p}
		\leq
		C_{K}
		(\mathcal{F}_{\infty}^{N+k+1})^2
		\frac{(\log t)^{K+1}}{t^{K+2}}.
	\]
	Similarly for all but the final term on the right hand side of \eqref{eq:gsKfK}.  The final term is estimated using Proposition \ref{prop:phiderivexpansion} and \eqref{eq:fKestimateagain},
	\begin{multline*}
		\sum_{\vert I \vert + \vert J \vert \leq N}
		\Big\Vert
		L^I (t^{-1} \partial_p)^J
		\Big(
		\Big(\partial_{x^i} \phi_{[K]}(t,x)
		-
		\frac{1}{t^2}
		\sum_{k=0}^K \sum_{l = 0}^k \frac{(\log t)^l}{t^k} \Psi_{k,l}^i(y,p)
		\Big)
		\partial_{p^i} f_{[K]}(t,x,p)
		\Big)
		\Big\Vert_{L^2_x L^2_p}
		\\
		\leq
		C_{K} (\mathcal{F}_{\infty}^{N+2K+3})^2
		\frac{(\log t)^{1+K}}{t^{2+K}}.
	\end{multline*}
	The proof of \eqref{eq:fKdefestimate} then follows from Proposition \ref{prop:phiderivexpansion}.
\end{proof}

\begin{remark}[Leading order term in $f_{k,l}$ vanishes upon integration in $y$] \label{rmk:totalyderiv}
	For fixed $1 \leq k\leq K$ and $0\leq l \leq k$, the expressions \eqref{eq:fklpsi1}--\eqref{eq:fklpsi3} for $f_{k,l}$ involve $\phi_{k',l'}$ for all $0 \leq k'\leq k$ and $0\leq l' \leq k'$.  Note however that, upon integration in $y$, the terms involving $\phi_{k,l'}$ vanish, for all $l' = 0,\ldots,k$ --- regardless of how $\phi_{k,l}$ are defined --- and only the terms involving $\phi_{k',l'}$ with $k'<k$ and $0\leq l'\leq k'$ remain.  See \eqref{eq:introiteratedefintegrated}, and the surrounding discussion, for this observation in the special case $k=1$.
	
	Indeed, considering the expression \eqref{eq:Psikl2iso}, it follows that the integral of \eqref{eq:fklpsi2} takes the form
	\begin{align} \label{eq:fklpsi2int}
		&
		\int_{\mathbb{R}^3} f_{k,k}(y,p) dy
		=
		\frac{1}{k}
		\sum_{\substack{
		k_1+k_2=k-1 \\ l_1+l_2=k -1
		}}
		\int_{\mathbb{R}^3} (\partial_{x^i} f_{k_1,l_1})(y,p) \Psi^j_{k_2,l_2}(y,p) \partial_i \partial_j \phi_{\infty}(p) dy
		\\
		&
		\qquad \qquad
		-
		\frac{1}{k}
		\sum_{\substack{
		k_1+k_2=k \\ l_1+l_2=k \\ k_2 \geq 1
		}}
		\sum_{\substack{ k_3+k_4=k_2 \\ l_3+l_4 = l_2 \\ k_4 \geq 1}}
		\binom{k_4}{l_4} \frac{(-1)^{l_4}}{k_4!} 
		\int_{\mathbb{R}^3}
		\otimes_{k_4-l_4} y \otimes_{l_4} \nabla \phi_{\infty}(p) \cdot \nabla^{k_4+1} \phi_{k_3,l_3} (p)
		(\partial_{x^i} f_{k_1,l_1})(y,p) 
		dy
		,
		\nonumber
	\end{align}
	and the integral of \eqref{eq:fklpsi3} takes the form, for $l=k-1,\ldots,0$,
	\begin{align} \label{eq:fklpsi3int}
		&
		\int_{\mathbb{R}^3} f_{k,l}(y,p) dy
		=
		\frac{l+1}{k} \int_{\mathbb{R}^3} f_{k,l+1}(y,p) dy
		\\
		&
		\qquad \qquad
		+
		\frac{\iota_{l\geq 1}}{k}
		\sum_{\substack{
		k_1+k_2=k-1 \\ l_1+l_2=l -1
		}}
		\int_{\mathbb{R}^3}
		(\partial_{x^i} f_{k_1,l_1})(y,p) \Psi^j_{k_2,l_2}(y,p) \partial_i \partial_j \phi_{\infty}(p)
		dy
		\nonumber
		\\
		&
		\qquad \qquad
		+
		\frac{1}{k}
		\sum_{\substack{
		k_1+k_2=k-1 \\ l_1+l_2=l
		}}
		\int_{\mathbb{R}^3}
		(\partial_{p^i} f_{k_1,l_1})(y,p) \Psi^j_{k_2,l_2}(y,p) \partial_i \partial_j \phi_{\infty}(p)
		dy
		\nonumber
		\\
		&
		\qquad \qquad
		-
		\frac{1}{k}
		\sum_{\substack{
		k_1+k_2=k \\ l_1+l_2=l \\ k_2 \geq 1
		}}
		\sum_{\substack{ k_3+k_4=k_2 \\ l_3+l_4 = l_2 \\ k_4 \geq 1}}
		\binom{k_4}{l_4} \frac{(-1)^{l_4}}{k_4!} 
		\int_{\mathbb{R}^3}
		\otimes_{k_4-l_4} y \otimes_{l_4} \nabla \phi_{\infty}(p) \cdot \nabla^{k_4+1} \phi_{k_3,l_3} (p)
		(\partial_{x^i} f_{k_1,l_1})(y,p)
		dy
		.
		\nonumber
	\end{align}
	Note again that terms of the form $\phi_{k,l'}$ are precluded from the final summations of \eqref{eq:fklpsi2int} and \eqref{eq:fklpsi3int} by the condition that $k_4 \geq 1$.
	
	This fact that the top order $\phi_{k,l}$ terms vanish upon integration in $y$ is used in the proof of Theorem \ref{thm:approx} in Section \ref{subsec:thmapprox} below.  See also Section \ref{subsubsec:fKlrhoKl} below.
\end{remark}

\subsection{The approximate solution $\varrho_{[K]}$}
\label{subsec:rhoexpansion}

Throughout this section suppose, for the purpose of the inductive argument given in Section \ref{subsec:thmapprox}, that $K\geq 1$ is such that some smooth functions $f_{k,l}$ are given, for $k=0,\ldots,K-1$, $l=0,\ldots,k$, with $f_{0,0}=f_{\infty}$, which satisfy the support property \eqref{eq:fklsupport} and, for some constants $C_{n,k}$, for all $n \geq 6$,
\begin{equation} \label{eq:calP0}
	\sum_{\vert I \vert + \vert J \vert \leq n} \Vert \partial_x^I \partial_p^J f_{k,l} \Vert_{L^2_x L^2_p}
	\leq
	C_{n,k} \mathcal{F}_{\infty}^{n+k}.
\end{equation}
In this section there are two main aims:
\begin{itemize}
	\item
		The first aim is to give explicit definitions of $\varrho_{k,l}$ in terms of $f_{k,l}$, for $0\leq k \leq K-1$ and $0\leq l \leq k$, so that $\varrho_{[K-1]}$ defined by \eqref{eq:rhophiKdefthm} is approximately equal to the spatial density of the function $f_{[K-1]}$ defined by \eqref{eq:fKdefthm}, i.\@e.\@ so that
		\begin{equation} \label{eq:rhosectionkeyproperty}
			\sum_{k=0}^{K-1} \sum_{l=0}^k \frac{(\log t)^l}{t^k}
			\Big( \int f_{k,l}(y(t,x,p),p) dp - \frac{1}{t^{3}} \varrho_{k,l} \Big( \frac{x}{t} \Big) \Big)
			=
			\mathcal{O} \bigg( \frac{(\log t)^{K-1}}{t^{4+(K-1)}} \bigg),
		\end{equation}
		or, more precisely, so that the estimate \eqref{eq:thmapprox2} holds for $K-1$.  This property is equivalent to the requirement that the relations \eqref{eq:rhoklf2} below hold, and it is indeed the relations \eqref{eq:rhoklf2} which are used to define the functions $\varrho_{k,l}$.
	\item
		The second aim is to give explicit definitions of functions $f_{K,l}$ and $\varrho_{K,l}$, for $l=0,\ldots,K$, in terms of the functions $f_{k,l}$, for $0\leq k \leq K-1$, $l=0,\ldots,k$.	
		Using the observation of Remark \ref{rmk:totalyderiv}, these functions can be defined in such a way that the relation \eqref{eq:rhoklf2} also holds for $k=K$, and thus that the property \eqref{eq:rhosectionkeyproperty} moreover holds for $K$ in place of $K-1$.  See Remark \ref{rmk:stillholds} below.
\end{itemize}
Both of these aims are achieved together in Proposition \ref{prop:rhoKdef} below.

\subsubsection{Expansions for $p(t,x,y)$ and $\det (\partial y/\partial p)$}

Before giving the definitions of $\varrho_{k,l}$, certain properties of the expression \eqref{eq:defofy} for $y(t,x,p)$ are established.  A computation gives
	\begin{equation} \label{eq:J1J20}
		\det \frac{\partial y}{\partial p}
		=
		-
		t^3
		+
		t^2 \log t \mathring{J}_1 (p)
		+
		t (\log t)^2 \mathring{J}_2 (p)
		+
		(\log t)^3 \mathring{J}_3 (p)
		,
	\end{equation}
	where
	\begin{equation} \label{eq:J1J2}
		\mathring{J}_1
		=
		\Delta \phi_{\infty},
	\qquad
		\mathring{J}_2
		=
		(\partial_1 \partial_2 \phi_{\infty})^2
		+
		(\partial_1 \partial_3 \phi_{\infty})^2
		+
		(\partial_2 \partial_3 \phi_{\infty})^2
		-
		\partial_1^2 \phi_{\infty} \partial_2^2 \phi_{\infty}
		-
		\partial_1^2 \phi_{\infty} \partial_3^2 \phi_{\infty}
		-
		\partial_2^2 \phi_{\infty} \partial_3^2 \phi_{\infty},
	\end{equation}
	and
	\begin{equation} \label{eq:J1J22}
		\mathring{J}_3
		=
		\partial_1^2 \phi_{\infty} \partial_2^2 \phi_{\infty} \partial_3^2 \phi_{\infty}
		+
		2 \partial_1 \partial_2 \phi_{\infty} \partial_1 \partial_3 \phi_{\infty} \partial_2 \partial_3 \phi_{\infty}
		-
		\partial_1^2 \phi_{\infty} (\partial_2 \partial_3 \phi_{\infty})^2
		-
		\partial_2^2 \phi_{\infty} (\partial_1 \partial_3 \phi_{\infty})^2
		-
		\partial_3^2 \phi_{\infty} (\partial_1 \partial_2 \phi_{\infty})^2
		.
	\end{equation}
	It therefore follows from the Implicit Function Theorem that, for bounded $p$, if $T_0$ is sufficiently large, the expression \eqref{eq:defofy} can be inverted to give $p(t,x,y)$ as a function of $t$, $x$ and $y$.

In order to define the expansion for $\varrho$ it is helpful to first provide an expansion for this function.

\begin{proposition}[Expansion for $p(t,x,y)$] \label{prop:pexpansion}
	There are explicit functions $p_{k,l} \colon \mathbb{R}^3 \times \mathbb{R}^3 \to \mathbb{R}^3$, for $k = 1,2,\ldots$, $l= 0, \ldots, k$, such that, for any $K \geq 1$, the function $p(t,x,y)$ satisfies, for any $B \geq 1$ and for all $t \geq T_0$,
	\[
		\sup_{\vert x/t \vert + \vert y \vert \leq B}
		\Big\vert
		p(t,x,y)
		-
		\Big[
		\frac{x}{t}
		+
		\sum_{k=1}^K \sum_{l=0}^k
		\frac{(\log t)^l}{t^k} p_{k,l} \Big( \frac{x}{t},y \Big)
		\Big]
		\Big\vert
		\lesssim
		\left( \frac{\log t}{t} \right)^{K+1}
		.
	\]
\end{proposition}

\begin{proof}
	The expression \eqref{eq:defofy} can be rewritten
	\begin{equation} \label{eq:pofy}
		p = \frac{x}{t} + \frac{y}{t} + \frac{\log t}{t} \nabla \phi_{\infty}(p).
	\end{equation}
	Rather than give an explicit expression for the $(k,l)$-th element of the expression, the general procedure for obtaining such an expression is described.  Expression \eqref{eq:pofy} moreover gives
	\[
		p
		=
		\frac{x}{t} + \frac{y}{t}
		+
		\frac{\log t}{t} \nabla \phi_{\infty} \Big( \frac{x}{t} + \frac{y}{t} + \frac{\log t}{t} \nabla \phi_{\infty}(p) \Big),
	\]
	and, for any $K \geq 1$, by Taylor's Theorem,
	\[
		\sup_{\vert x/t \vert + \vert y \vert \leq B}
		\Big\vert
		p(t,x,y)
		-
		\Big[
		\frac{x}{t} + \frac{y}{t}
		+
		\frac{\log t}{t}
		\Big[
		\phi_{\infty} \Big( \frac{x}{t} \Big)
		+
		\sum_{k=1}^K
		\otimes_k \Big( \frac{y}{t} + \frac{\log t}{t} \nabla \phi_{\infty}(p) \Big)
		\cdot \nabla^{k+1} \phi_{\infty} \Big( \frac{x}{t} \Big)
		\Big]
		\Big]
		\Big\vert
		\lesssim
		\left( \frac{\log t}{t} \right)^{K+1}.
	\]
	The explicit expressions are obtained by iteratively inserting the expression \eqref{eq:pofy} for $p$.
	
	The first elements of the sequence $\{p_{k,l}\}$ take the form
	\[
		p_{1,1}(w,y) = \nabla \phi_{\infty}(w),
		\qquad
		p_{1,0}(w,y) = y,
	\]
	\[
		p_{2,2}(w,y) = \nabla \phi_{\infty}(w) \cdot \nabla^2 \phi_{\infty}(w),
		\qquad
		p_{2,1}(w,y) = y \cdot \nabla^2 \phi_{\infty}(w),
		\qquad
		p_{2,0}(w,y) = 0.
	\]
\end{proof}

One similarly has an expansion for the Jacobian of the change $p \mapsto y(t,x,p)$.  See \eqref{eq:introdetexpansion} for a discussion of the first terms in this expansion.

\begin{proposition}[Expansion for $\mathrm{det} \frac{\partial p}{\partial y}$] \label{prop:Jexpansion}
	There are explicit functions $J_{k,l} \colon \mathbb{R}^3 \times \mathbb{R}^3 \to \mathbb{R}$, for $k = 0,1,2,\ldots$, $l= 0, \ldots, k$, such that, for any $K \geq 1$ and any $B \geq 1$, the determinant of $\frac{\partial p}{\partial y}$ satisfies, for all $t \geq T_0$,
	\[
		\sup_{\vert x/t \vert + \vert y \vert \leq B}
		\Big\vert
		\det \frac{\partial p}{\partial y}(t,x,y)
		-
		\frac{1}{t^3}
		\sum_{k=1}^K \sum_{l=0}^k
		\frac{(\log t)^l}{t^k} J_{k,l} \Big( \frac{x}{t},y \Big)
		\Big\vert
		\lesssim
		\frac{(\log t)^{K+1}}{t^{K+4}}
		.
	\]
\end{proposition}

\begin{proof}
	Recall \eqref{eq:J1J20}--\eqref{eq:J1J22}.  It follows that, for any $K \geq 1$,
	\[
		\sup_{\vert x/t \vert + \vert y \vert \leq B}
		\Big\vert
		\det \frac{\partial p}{\partial y}(t,x,y)
		-
		\Big[
		-
		\frac{1}{t^3}
		-
		\frac{1}{t^3}
		\sum_{k=1}^K
		\Big(
		\frac{\log t}{t} \mathring{J}_1 (p)
		+
		\frac{(\log t)^2}{t^2} \mathring{J}_2 (p)
		+
		\frac{(\log t)^3}{t^3} \mathring{J}_3 (p)
		\Big)^k
		\Big]
		\Big\vert
		\lesssim
		\frac{(\log t)^{K+1}}{t^{K+4}}
		.
	\]
	The proof then follows from using the expansions for $p(t,x,y)$ from Proposition \ref{prop:pexpansion}.  The first elements of the sequence $\{J_{k,l}\}$ take the explicit form
	\[
		J_{0,0}(w,y) = - 1,
		\qquad
		J_{1,1}(w,y) = - \mathring{J}_1(w),
		\qquad
		J_{1,0}(w,y) = 0,
	\]
	\[
		J_{2,2}(w,y) = - \mathring{J}_2 (w) - (\mathring{J}_1 (w))^2 - \nabla \phi_{\infty}(w) \cdot \nabla \mathring{J}_1 (w),
		\qquad
		J_{2,1}(w,y) = - y \cdot \nabla \mathring{J}_1 (w),
		\qquad
		J_{2,0}(w,y) = 0.
	\]
\end{proof}

\subsubsection{The definitions of $\varrho_{k,l}$ for $k \leq K-1$}

The definitions of $\varrho_{k,l}$, for $k \leq K-1$, are now given.
For smooth functions $h \colon \mathbb{R}^3 \times \mathbb{R}^3 \to \mathbb{R}$, define,
\[
	H_{0,0}[h](w,y) := h(y,w)
\]
and, for $k \geq 1$ and $0 \leq l \leq k$,
\begin{align} \label{eq:rhoklf3}
	H_{k,l}[h](w,y)
	=
	\sum_{m=1}^k
	\frac{1}{m!}
	\sum_{\substack{k_1+\ldots+k_m=k \\ l_1+\ldots+l_m=l \\ 1 \leq k_i \leq k, \ 0 \leq l_i \leq k_i}}
	p_{k_1,l_1} (w,y) \otimes \ldots \otimes p_{k_m,l_m} (w,y)
	\cdot
	(\nabla_p^{m} h) (y,w),
\end{align}
where $\nabla_p h$ denotes the gradient of $h$ with respect to the second argument, and the functions $p_{k,l}$ are as in Proposition \ref{prop:pexpansion}.

Define now,
\begin{align} 
	\varrho_{0,0}(w)
	&
	=
	- \int f_{\infty}(y,w)dy,
	\label{eq:rhoklf0}
\end{align}
and, for $1 \leq k \leq K-1$, $l=0,\ldots,k$,
\begin{align} \label{eq:rhoklf2}
	\varrho_{k,l}(w)
	=
	\sum_{\substack{k_1+k_2+k_3=k \\ l_1+l_2+l_3=l \\ 0 \leq k_i \leq k, \ 0 \leq l_i \leq k_i}}
	\int_{\mathbb{R}^3}
	H_{k_2,l_2}[f_{k_1,l_1}] (w,y) J_{k_3,l_3}(w,y)
	dy.
\end{align}
Property \eqref{eq:rhoklf2} is the relation required to achieve cancellations of the form \eqref{eq:rhosectionkeyproperty}.  See the proof of Proposition \ref{prop:rhoKdef} below, where these cancellations are exploited to show estimate \eqref{eq:proprho2}.

\subsubsection{The definitions of $f_{K,l}$ and $\varrho_{K,l}$}
\label{subsubsec:fKlrhoKl}

Functions $f_{K,l}$ and $\varrho_{K,l}$, for $0\leq l \leq K$, are now defined --- explicitly in terms of the functions $f_{k,l}$, for $0\leq k \leq K-1$, $0\leq l \leq k$ --- so that the property \eqref{eq:rhoklf2} moreover holds for $k=K$.  This fact ensures that cancellations of the form \eqref{eq:rhosectionkeyproperty} are in fact achieved with $K$ in place of $K-1$.

First define smooth functions $\phi_{k,l} \colon \mathbb{R}^3 \to \mathbb{R}$, for $k=0,\ldots, K-1$, $l=0,\ldots,k$ as the unique solution of the Poisson equation \eqref{eq:mainthmPoisson} sourced by $\varrho_{k,l}$ as defined above.  The functions $\phi_{k,l}$ are expressed explicitly in terms of the functions $\varrho_{k,l}$ via the well known representation formula for the Poisson equation
\begin{equation} \label{eq:Poissonrep}
	\phi_{k,l}(w) = \frac{1}{4\pi} \int_{\mathbb{R}^3} \frac{\varrho_{k,l}(y)}{\vert y - w \vert} dy.
\end{equation}

Define a smooth function $\mathcal{P}_{K,K} \colon \mathbb{R}^3 \to \mathbb{R}$ by setting $\mathcal{P}_{K,K}(p)$ to be equal to the right hand side of \eqref{eq:fklpsi2int} with $k=K$.  Next define $\mathcal{P}_{K,l} \colon \mathbb{R}^3 \to \mathbb{R}$ inductively, for $l=K-1,K-2,\ldots,0$, by setting $\mathcal{P}_{K,l}(p)$ to be equal to the right hand side of \eqref{eq:fklpsi3int} with $k=K$.  (Recall the discussion in Remark \ref{rmk:totalyderiv} --- these definitions of $\mathcal{P}_{K,l}$ indeed only involve $f_{K,l}$ and $\phi_{K,l}$ up to order $K-1$, and so $\mathcal{P}_{K,l}$ are well defined.)

Define now, for $l=0,\ldots,K$,
\begin{align} \label{eq:rhoklf2toporder}
	\varrho_{K,l}(w)
	=
	-
	\mathcal{P}_{K,l}(w)
	+
	\sum_{\substack{k_1+k_2+k_3=K \\ l_1+l_2+l_3=l \\ 0 \leq k_i \leq K, \ 0 \leq l_i \leq k_i, \\ k_1\leq K-1, \ l_1 \leq l-1}}
	\int_{\mathbb{R}^3}
	H_{k_2,l_2}[f_{k_1,l_1}] (w,y) J_{k_3,l_3}(w,y)
	dy.
\end{align}

Finally, define a smooth function $f_{K,K} \colon \mathbb{R}^3 \times \mathbb{R}^3 \to \mathbb{R}$ by equation \eqref{eq:fklpsi2}, and define $f_{K,l} \colon \mathbb{R}^3 \times \mathbb{R}^3 \to \mathbb{R}$ inductively, for $l=K-1,K-2,\ldots,0$, by equation \eqref{eq:fklpsi3}, with $\phi_{K,l}$ defined by \eqref{eq:mainthmPoisson} sourced by $\varrho_{K,l}$.

\begin{remark}[The formula \eqref{eq:rhoklf2} also holds for $\varrho_{K,l}$] \label{rmk:stillholds}
	Note that, with these definitions, $f_{K,l}$ satisfy
\begin{equation} \label{eq:intfeqP}
	\int_{\mathbb{R}^3} f_{K,l}(y,w) dy = \mathcal{P}_{K,l}(w),
\end{equation}
and thus \eqref{eq:rhoklf2} also holds for $k=K$.  This important property is used in the proof of Proposition \ref{prop:rhoKdef} below.
\end{remark}

\subsubsection{The main properties of $\varrho_{k,l}$}

The main result of this section is the following proposition, which concerns properties of the functions $\varrho_{k,l}$ as defined above.

\begin{proposition}[Properties of $\varrho_{k,l}$] \label{prop:rhoKdef}
	Consider some $K\geq 0$ and smooth functions $f_{k,l} \colon \mathbb{R}^3 \times \mathbb{R}^3 \to \mathbb{R}$ satisfying \eqref{eq:phiklassum} for all $k=0,\ldots,K-1$, $l=0,\ldots,k$.
	
	The functions $\varrho_{k,l} \colon \mathbb{R}^3 \to \mathbb{R}$ defined by \eqref{eq:rhoklf0}--\eqref{eq:rhoklf2} and \eqref{eq:rhoklf2toporder} satisfy, for all $k=0,\ldots,K$, $l=0,\ldots,k$ and any $n \geq 6$,
	\begin{equation} \label{eq:proprho1}
		\sum_{\vert I \vert \leq n}
		\Vert \partial_x^I \varrho_{k,l} \Vert_{L^2_x}
		\leq
		C_{n,k} \mathcal{F}_{\infty}^{n+k},
		\qquad
		\supp(\varrho_{k,l})
		\subset
		\{ x\in \mathbb{R}^3 \mid \vert x \vert \leq B \}.
	\end{equation}
	The functions $\phi_{k,l} \colon \mathbb{R}^3 \to \mathbb{R}$, as defined above, satisfy, for all $k=0,\ldots,K$, $l=0,\ldots,k$ and any $n \geq 6$,
	\begin{equation} \label{eq:proprho1A}
		\sum_{\vert I \vert \leq n} \Vert \partial^I \nabla \phi_{k,l} \Vert_{L^2}
		\leq
		C_{n,k} \mathcal{F}_{\infty}^{n+k}.
	\end{equation}
	For any $t \geq T_0$, the function $\varrho_{[K]} \colon [T_0,\infty) \times \mathbb{R}^3 \to \mathbb{R}$, defined by \eqref{eq:rhophiKdefthm}, satisfies
	\begin{equation} \label{eq:proprho2}
		\sum_{\vert I \vert \leq N}
		\big\Vert (t\partial_x)^I \Big( \int_{\mathbb{R}^3} f_{[K]} (t,\cdot,p) dp - \varrho_{[K]} (t,\cdot) \Big) \big\Vert_{L^2_x}
		\leq
		C_{K} \mathcal{F}_{\infty}^{N+2K+3}
		\frac{(\log t)^{K+1}}{t^{\frac{5}{2}+K}},
	\end{equation}
	where $f_{[K]}$ is defined by \eqref{eq:fKdefthm}, with $f_{K,l}$ as defined above.
\end{proposition}

\begin{proof}
	Consider first the properties \eqref{eq:proprho1}.  For $k =0,\ldots,K-1$, the properties \eqref{eq:proprho1} follow from the fact that each $f_{k,l}$, for $k=0,\ldots,K-1$, $l=0,\ldots,k$, satisfy the support property \eqref{eq:fklsupport} the estimates \eqref{eq:calP0}.
	
	The property \eqref{eq:proprho1A}, for $k =0,\ldots,K-1$, then follows from \eqref{eq:proprho1} and the gradient estimate of Proposition \ref{prop:gradelliptic} applied to the Poisson equation \eqref{eq:mainthmPoisson}.
	
	For $k = K$, the fact that \eqref{eq:proprho1A}, the support property \eqref{eq:fklsupport} and the estimates \eqref{eq:calP0} hold for $k =0,\ldots,K-1$ imply that $\mathcal{P}_{K,l}$ satisfy, for all $n \geq 6$,
	\begin{equation} \label{eq:calP1}
		\supp (\mathcal{P}_{K,l}) \subset \{ x \in \mathbb{R}^3 \mid \vert x \vert \leq B \},
		\qquad
		\sum_{\vert I \vert \leq n} \Vert \partial_x^I \mathcal{P}_{K,l} \Vert_{L^2_x}
		\leq
		C_{n,K} (\mathcal{F}_{\infty}^{n+K})^2.
	\end{equation}
	It also follows, from Proposition \ref{prop:fKdef}, that $f_{K,l}$ satisfies the support property \eqref{eq:fklsupport} and the estimate \eqref{eq:calP0}.  These facts imply that the properties \eqref{eq:proprho1} also hold for $k=K$.  The property \eqref{eq:proprho1A} for $k=K$ then follows again from the gradient estimate of Proposition \ref{prop:gradelliptic}.
	
	Consider now the estimate \eqref{eq:proprho2}.  First note that, for any smooth $h \colon \mathbb{R}^3 \times \mathbb{R}^3 \to \mathbb{R}$, satisfying
	\begin{equation} \label{eq:supportofhproprho}
		\supp(h) \subset \{ (x,p) \in \mathbb{R}^3 \times \mathbb{R}^3 \mid \vert x \vert \leq B, \vert p \vert \leq B \},
	\end{equation}
	by Taylor's Theorem (see Proposition \ref{prop:Taylor}), using the notation of Section \ref{subsec:notation},
	\begin{align*}
		h(y,p)
		=
		\
		&
		h\Big(y,\frac{x}{t}\Big)
		+
		\sum_{n=1}^K
		\frac{1}{n!} \otimes_n \Big(p - \frac{x}{t}\Big) \cdot (\nabla^n_p h)\Big(y,\frac{x}{t}\Big)
		+
		\frac{1}{K!} \otimes_{K+1} \Big(p - \frac{x}{t}\Big) \cdot \int_0^1 (1-s)^K (\nabla^{K+1}_p h)(y,sp) ds.
	\end{align*}	
	Thus Proposition \ref{prop:pexpansion} implies that
	\begin{align*}
		&
		h(y,p)
		=
		h \Big( y, \frac{x}{t} \Big)
		\\
		&
		\quad \qquad
		+
		\sum_{n=1}^K \frac{1}{n!}
		\sum_{k_1=1}^K \sum_{l_1=0}^{k_1} \ldots \sum_{k_n=1}^K \sum_{l_n=0}^{k_n}
		\frac{(\log t)^{l_1+\ldots + l_n}}{t^{k_1+ \ldots + k_n}}
		p_{k_1,l_1}\Big( \frac{x}{t},y \Big) \otimes \ldots \otimes p_{k_n,l_n}\Big( \frac{x}{t},y \Big)
		\cdot
		\nabla_p^{n} h \Big(y,\frac{x}{t} \Big)
		+
		E_h^{K+1}
		\\
		&
		\quad
		=
		h\Big(y, \frac{x}{t}\Big)
		+
		\sum_{k=1}^K \sum_{l=0}^k
		\frac{(\log t)^l}{t^k}
		\sum_{m=1}^k
		\frac{1}{m!}
		\sum_{\substack{k_1+\ldots+k_m=k \\ l_1+\ldots+l_m=l}}
		p_{k_1,l_1}\Big(\frac{x}{t},y\Big) \otimes \ldots \otimes p_{k_m,l_m}\Big(\frac{x}{t},y\Big)
		\cdot
		\nabla_p^{m} h \Big(y,\frac{x}{t}\Big)
		+
		E_h^{K+1}
		\\
		&
		\quad
		=
		h\Big(y, \frac{x}{t}\Big)
		+
		\sum_{k=1}^K \sum_{l=0}^k
		\frac{(\log t)^l}{t^k}
		H_{k,l}[h]\Big(\frac{x}{t},y\Big)
		+
		E_h^{K+1},
	\end{align*}
	where $H_{k,l}[h]$ is defined by \eqref{eq:rhoklf3}, and $E_h^{K+1}(y(t,x,p),p)$ satisfies
	\begin{multline} \label{eq:EhK1}
		\int_{\mathbb{R}^3} \vert E_h^{K+1}(y(t,x,p),p) \vert dp
		\lesssim
		\frac{(\log t)^{K+1}}{t^{K+1}}
		\sum_{\vert I \vert \leq K+1}
		\int  \vert \partial_p^I h (y,p(t,x,y)) \vert \, \Big\vert \det \frac{\partial p}{\partial y} \Big\vert dy
		\\
		\lesssim
		\frac{(\log t)^{K+1}}{t^{K+4}}
		\sum_{\vert I \vert \leq K+1}
		\sup_{p \in \mathbb{R}^3}
		\Big( \int \vert \partial_p^I h (y,p) \vert^2 dy \Big)^{\frac{1}{2}}
		\lesssim
		\frac{(\log t)^{K+1}}{t^{K+4}}
		\sum_{\vert I \vert \leq K+3} \Vert  \partial_p^I h \Vert_{L^2_x L^2_p},
	\end{multline}
	where the support property \eqref{eq:supportofhproprho} is used, along with the Sobolev inequality of Proposition \ref{prop:SobolevL2}.  Moreover,
	\begin{equation} \label{eq:EhK2}
		\supp(E_h^{K+1}(y(t,\cdot,p),p)) \subset \{ x\in \mathbb{R}^3 \mid \vert x \vert \leq (2B+\mathcal{F}_{\infty}^3) t \},
	\end{equation}
	which follows from the fact that $\vert x \vert \leq B + t B + \log t \mathcal{F}_{\infty}^3$ if $\vert y \vert \leq B$ and $\vert p \vert \leq B$.
	Thus
	\begin{align*}
		\int h(y(t,x,p),p) dp
		&
		=
		\int h(y,p(t,x,y)) \det \frac{\partial p}{\partial y}(t,x,y) dy
		\\
		&
		=
		\frac{1}{t^3}
		\sum_{k_1=0}^K \sum_{l_1=0}^{k_1} \sum_{k_2=0}^K \sum_{l_2=0}^{k_2}
		\frac{(\log t)^{l_1+l_2}}{t^{k_1 + k_2}}
		\int_{\mathbb{R}^3}
		H_{k_1,l_1}[h]\Big(\frac{x}{t},y\Big)
		J_{k_2,l_2}\Big(\frac{x}{t},y\Big)
		dy
		+
		\int E_h^{K+1} dp
		\\
		&
		=
		\frac{1}{t^3}
		\sum_{k=0}^K \sum_{l=0}^{k}
		\frac{(\log t)^{l}}{t^{k}}
		\sum_{\substack{m_1 + m_2=k \\ n_1+n_2=l}}
		\int_{\mathbb{R}^3}
		H_{m_1,n_1}[h]\Big(\frac{x}{t},y\Big)
		J_{m_2,n_2}\Big(\frac{x}{t},y\Big)
		dy
		+
		\int E_h^{K+1} dp,
	\end{align*}
	where $E_h^{K+1}$ varies from line to line, but always satisfies \eqref{eq:EhK1}--\eqref{eq:EhK2}.
	Applying with $h=f_{k,l}$, the definitions \eqref{eq:rhoklf0}--\eqref{eq:rhoklf2} of $\varrho_{k,l}$ for $0 \leq k \leq K-1$ and the fact that $\varrho_{K,l}$ also satisfies \eqref{eq:rhoklf2} (see the definition \eqref{eq:rhoklf2toporder} and the important property \eqref{eq:intfeqP}, discussed in Remark \ref{rmk:stillholds}) imply that
	\begin{align*}
		\sum_{k=0}^K \sum_{l=0}^{k}
		\frac{(\log t)^{l}}{t^{k}}
		\int f_{k,l}(y(t,x,p),p) dp
		&
		=
		\frac{1}{t^3}
		\sum_{k=0}^K \sum_{l=0}^{k}
		\frac{(\log t)^{l}}{t^{k}}
		\varrho_{k,l} \Big( \frac{x}{t} \Big)
		+
		\sum_{k=0}^K \sum_{l=0}^{k}
		E_{f_{k,l}}^{K+1}(y(t,x,p),p).
	\end{align*}
	The properties \eqref{eq:EhK1} and \eqref{eq:EhK2} imply that, for each $k,l$,
	\[
		\Big\Vert
		\int_{\mathbb{R}^3}
		E_{f_{k,l}}^{K+1}(y(t,\cdot,p),p)
		dp
		\Big\Vert_{L^2}
		\lesssim
		\frac{(\log t)^{K+1}}{t^{K+\frac{5}{2}}}
		\sum_{\vert I \vert \leq K+3} \Vert  \partial_p^I f_{k,l} \Vert_{L^2_x L^2_p}
		\lesssim
		\frac{(\log t)^{K+1}}{t^{K+\frac{5}{2}}} \mathcal{F}_{\infty}^{k+K+3},
	\]
	where the final inequality follows from \eqref{eq:calP0}.  The proof of the zeroth order part of \eqref{eq:proprho2} then follows.
	
	For the part of \eqref{eq:proprho2} involving higher order derivatives, it suffices to show that, for each $\vert I \vert \leq N$, there exists $\varrho_{[K]}^I$ such that
	\begin{equation} \label{eq:functionssatisfying}
		\big\Vert (t\partial_x)^I \Big( \int_{\mathbb{R}^3} f_{[K]} (t,\cdot,p) dp\Big)
		-
		\sum_{k=0}^K \sum_{l=0}^k \frac{(\log t)^l}{t^k} \varrho_{k,l}^I (\frac{\cdot}{t}) \big\Vert_{L^2_x}
		\leq
		C_{K} \mathcal{F}_{\infty}^{N+2K+3}
		\frac{(\log t)^{K+1}}{t^{\frac{5}{2}+K}}.
	\end{equation}
	It then follows that $\varrho_{k,l}^I = \partial^I\varrho_{k,l}$, and thus that \eqref{eq:proprho2} holds.
	
	To check that there are functions satisfying \eqref{eq:functionssatisfying}, consider first the case $\vert I \vert =1$, so that there is $i\in\{1,2,3\}$ such that $\partial_x^I = \partial_{x^i}$. Note first that, for any smooth $h \colon \mathbb{R}^3 \times \mathbb{R}^3 \to \mathbb{R}$,
	\begin{equation} \label{eq:horhoexp1}
		t \partial_{x^i} \Big( \int_{\mathbb{R}^3} h(y(t,x,p),p) dp \Big)
		=
		\int_{\mathbb{R}^3} (t \partial_{x^i}h)(y(t,x,p),p) dp.
	\end{equation}
	Moreover,
	\begin{align}
		0
		&
		=
		\int
		\partial_{p^i} \Big(
		h(y(t,x,p),p)
		\Big)
		dp
		\nonumber
		\\
		&
		=
		- t
		\int
		(\partial_{x^i}h)(y,p)
		dp
		+
		\log t 
		\int
		\partial_i \partial_k \phi_{\infty}(p)
		(\partial_{x^k}h)(y,p)
		dp
		+
		\int
		(\partial_{p^i}h)(y,p)
		dp,
		\label{eq:horhoexp2}
	\end{align}
	and so, adding to \eqref{eq:horhoexp1} gives,
	\[
		t \partial_{x^i} \Big( \int_{\mathbb{R}^3} h(y(t,x,p),p) dp \Big)
		=
		\log t 
		\int
		\partial_i \partial_k \phi_{\infty}(p)
		(\partial_{x^k}h)(y,p)
		dp
		+
		\int
		(\partial_{p^i}h)(y(t,x,p),p)
		dp.
	\]
	One can continue:
	\begin{align}
		0
		&
		=
		\frac{\log t}{t}
		\int
		\partial_{p^k} \Big(
		\partial_i \partial_k \phi_{\infty}(p)
		h(y(t,x,p),p)
		\Big)
		dp
		\nonumber
		\\
		&
		=
		- \log t
		\int
		\partial_i \partial_k \phi_{\infty}(p)
		(\partial_{x^k}h)(y,p)
		dp
		+
		\frac{(\log t)^2}{t}
		\int
		\partial_i \partial_k \phi_{\infty}(p) \partial_k \partial_{k_2} \phi_{\infty}(p) (\partial_{x^{k_2}}h)(y,p)
		dp
		\nonumber
		\\
		&
		\quad
		+
		\frac{\log t}{t}
		\int
		\partial_i \Delta \phi_{\infty}(p)
		\cdot h(y,p)
		+
		\partial_i \partial_k \phi_{\infty}(p) (\partial_{p^k}h)(y,p)
		dp,
		\nonumber
	\end{align}
	and more generally, for any $l \geq 0$,
	\begin{align}
		0
		&
		=
		\frac{(\log t)^{l+1}}{t^{l+1}}
		\int
		\partial_{p^{k_{l+1}}} \Big(
		\partial_i \partial_{k_1} \phi_{\infty}(p) \prod_{j=1}^{l} \partial_{k_j} \partial_{k_{j+1}} \phi_{\infty}(p) \ h(y(t,x,p),p)
		\Big)
		dp
		\nonumber
		\\
		&
		=
		- \frac{(\log t)^{l+1}}{t^{l}}
		\int
		\partial_i \partial_{k_1} \phi_{\infty}(p) \prod_{j=1}^{l} \partial_{k_j} \partial_{k_{j+1}} \phi_{\infty}(p)
		(\partial_{x^{k_{l+1}}}h)(y,p)
		dp
		\label{eq:horhoexp3}
		\\
		&
		\quad
		+
		\frac{(\log t)^{l+2}}{t^{l+1}}
		\int
		\partial_i \partial_{k_1} \phi_{\infty}(p) \prod_{j=1}^{l+1} \partial_{k_j} \partial_{k_{j+1}} \phi_{\infty}(p) 
		(\partial_{x^{k_{l+2}}}h)(y,p)
		dp
		\nonumber
		\\
		&
		\quad
		+
		\frac{(\log t)^{l}}{t^{l}}
		\int
		\Big[
		\partial_{p^{k_{l+1}}} \Big( \partial_i \partial_{k_1} \phi_{\infty}(p) \prod_{j=1}^{l} \partial_{k_j} \partial_{k_{j+1}} \phi_{\infty}(p) \Big)
		\cdot h(y,p)
		\nonumber
		\\
		&
		\quad \qquad
		+
		\partial_i \partial_{k_1} \phi_{\infty}(p) \prod_{j=1}^{l} \partial_{k_j} \partial_{k_{j+1}} \phi_{\infty}(p) (\partial_{p^{k_{l+1}}} h)(y,p)
		\Big]
		dp,
		\nonumber
	\end{align}
	where $\Pi_{j=1}^0 \partial_{k_j} \partial_{k_{j+1}} \phi_{\infty}(p) : = 1$.
	Summing \eqref{eq:horhoexp1}, \eqref{eq:horhoexp2}, along with \eqref{eq:horhoexp3} for $l=0,\ldots, K+1$, gives
	\begin{align*}
		&
		t \partial_{x^i} \Big( \int_{\mathbb{R}^3} h(y(t,x,p),p) dp \Big)
		=
		\sum_{l=0}^{K+1}
		\frac{(\log t)^{l}}{t^{l}}
		\int
		\Big[
		\partial_{p^{k_{l+1}}} \Big( \partial_i \partial_{k_1} \phi_{\infty}(p) \prod_{j=1}^{l} \partial_{k_j} \partial_{k_{j+1}} \phi_{\infty}(p) \Big)
		\cdot h(y,p)
		\\
		&
		\quad \qquad
		+
		\partial_i \partial_{k_1} \phi_{\infty}(p) \prod_{j=1}^{l} \partial_{k_j} \partial_{k_{j+1}} \phi_{\infty}(p) (\partial_{p^{k_{l+1}}} h)(y,p)
		\Big]
		dp
		+
		\int
		(\partial_{p^i}h)(y,p)
		dp
		\\
		&
		\quad \qquad
		+
		\frac{(\log t)^{K+3}}{t^{K+2}}
		\int
		\partial_i \partial_{k_1} \phi_{\infty}(p) \prod_{j=1}^{K+2} \partial_{k_j} \partial_{k_{j+1}} \phi_{\infty}(p) 
		(\partial_{x^{k_{K+3}}}h)(y,p)
		dp
		.
	\end{align*}
	Thus, if one defines
	\[
		G_0[h](y,p)
		=
		\partial_{p^{k}} \Big( \partial_i \partial_{k} \phi_{\infty}(p) \Big) \cdot h(y,p)
		+
		\partial_i \partial_{k} \phi_{\infty}(p) (\partial_{p^{k}} h)(y,p)
		+
		(\partial_{p^i}h)(y,p),
	\]
	and, for $l \geq 1$,
	\begin{align*}
		&
		G_l[h](y,p)
		=
		\partial_{p^{k_{l+1}}} \Big( \partial_i \partial_{k_1} \phi_{\infty}(p) \prod_{j=1}^{l} \partial_{k_j} \partial_{k_{j+1}} \phi_{\infty}(p) \Big)
		\cdot h(y,p)
		+
		\partial_i \partial_{k_1} \phi_{\infty}(p) \prod_{j=1}^{l} \partial_{k_j} \partial_{k_{j+1}} \phi_{\infty}(p) (\partial_{p^{k_{l+1}}} h)(y,p),
	\end{align*}
	it follows that
	\begin{multline*}
		\Big\vert
		t \partial_{x^i} \Big( \int_{\mathbb{R}^3} h(y(t,x,p),p) dp \Big)
		-
		\sum_{l=0}^{K}
		\frac{(\log t)^{l}}{t^{l}}
		\int G_l[h](y(t,x,p),p) dp
		\Big\vert 
		\\
		\lesssim
		\frac{(\log t)^{K+1}}{t^{K+1}}
		\Big\vert \int G_l[h](y(t,x,p),p) dp \Big\vert
		+
		\frac{(\log t)^{K+3}}{t^{K+2}} \mathcal{F}^4_{\infty} t^{-\frac{3}{2}}\Vert \nabla_p h \Vert_{L^2_x L^{\infty}_p}
		\lesssim
		\frac{(\log t)^{K+1}}{t^{K+\frac{5}{2}}}
		\mathcal{F}^4_{\infty} \Vert \nabla_p h \Vert_{L^2_x L^{\infty}_p},
	\end{multline*}
	where the latter holds for $t \geq T_0$ if $T_0$ is sufficiently large.  Hence
	\[
		t \partial_{x^i} \Big( \int_{\mathbb{R}^3} f_{[K]}(y(t,x,p),p) dp \Big)
		=
		\sum_{k=0}^K \sum_{l=0}^k \frac{(\log t)^l}{t^k} t \partial_{x^i} \Big( \int_{\mathbb{R}^3} f_{k,l}(y(t,x,p),p) dp \Big),
	\]
	satisfies
	\[
		\Big\vert
		t \partial_{x^i} \Big( \int_{\mathbb{R}^3} f_{[K]}(y(t,x,p),p) dp \Big)
		-
		\sum_{k=0}^K \sum_{l=0}^k \frac{(\log t)^l}{t^k}
		\int \tilde{f}_{k,l}(y(t,x,p),p) dp
		\Big\vert 
		\lesssim
		\frac{(\log t)^{K+1}}{t^{K+\frac{5}{2}}}
		(\mathcal{F}^4_{\infty}+\mathcal{F}^{k+3}_{\infty}),
	\]
	where
	\[
		\tilde{f}_{k,l}(y,p)
		:=
		\sum_{\substack{m+k'=k \\ m\leq l}}
		\int G_m[f_{k',l-m}](y,p).
	\]
	Revisiting the zeroth order case, with $\tilde{f}_{k,l}$ in place of $f_{k,l}$, it follows that \eqref{eq:functionssatisfying} holds where $\varrho_{k,l}^I$ is defined by replacing $f_{k,l}$ with $\tilde{f}_{k,l}$ in \eqref{eq:rhoklf2}.
	
	The cases involving higher order derivatives follows similarly.

\end{proof}

\subsection{Approximate solutions: the proof of Theorem \ref{thm:approx}}
\label{subsec:thmapprox}

The proof of Theorem \ref{thm:approx} can now be given.

\begin{proof}[Proof of Theorem \ref{thm:approx}]
	The first step in the proof of Theorem \ref{thm:approx} is to give explicit definitions of the functions
	\begin{equation} \label{eq:fklrhoklapproxproof}
		f_{k,l} \colon \mathbb{R}^3 \times \mathbb{R}^3 \to \mathbb{R},
		\qquad
		\text{ and }
		\qquad
		\phi_{k,l}, \varrho_{k,l} \colon \mathbb{R}^3 \to \mathbb{R},
	\end{equation}
	for $k = 0,1,2,\ldots$ and $l=0,\ldots k$.  Define first
	\[
		f_{0,0} = f_{\infty}, \qquad \varrho_{0,0} = \varrho_{\infty}, \qquad \phi_{0,0} = \phi_{\infty}.
	\]
	Next suppose, inductively, that $K\geq 1$ is such that \eqref{eq:fklrhoklapproxproof} have been defined --- and satisfy \eqref{eq:fklsupport} and \eqref{eq:flkL2estimate} --- for all $0 \leq k \leq K-1$ and $0 \leq l \leq k$.  Define smooth functions
	\[
		\mathcal{P}_{K,l} \colon \mathbb{R}^3 \to \mathbb{R},
		\qquad
		0 \leq l \leq K,
	\]
	by first setting $\mathcal{P}_{K,K}(p)$ to be equal to the right hand side of \eqref{eq:fklpsi2int} with $k=K$.  Then define $\mathcal{P}_{K,l} \colon \mathbb{R}^3 \to \mathbb{R}$ inductively, for $l=K-1,K-2,\ldots,0$, by setting $\mathcal{P}_{K,l}(p)$ to be equal to the right hand side of \eqref{eq:fklpsi3int} with $k=K$.  (Recall the discussion in Remark \ref{rmk:totalyderiv} --- these definitions of $\mathcal{P}_{K,l}$ indeed only involve \eqref{eq:fklrhoklapproxproof} up to order $K-1$ and so $\mathcal{P}_{K,l}$ are well defined.)

Define now, for $l=0,\ldots,K$, functions $\varrho_{K,l}$, for $0\leq l \leq K$, by equation \eqref{eq:rhoklf2toporder}.  The functions $\phi_{K,l}$ are then defined as the unique solutions of the Poisson equation \eqref{eq:mainthmPoisson} sourced by $\varrho_{K,l}$.  (Recall that $\phi_{K,l}$ are expressed explicitly in terms of $\varrho_{K,l}$ by the well known representation formula \eqref{eq:Poissonrep} for the Poisson equation.)  Finally, define $f_{K,K} \colon \mathbb{R}^3 \times \mathbb{R}^3 \to \mathbb{R}$ by equation \eqref{eq:fklpsi2}, and define $f_{K,l} \colon \mathbb{R}^3 \times \mathbb{R}^3 \to \mathbb{R}$ inductively, for $l=K-1,K-2,\ldots,0$, by equation \eqref{eq:fklpsi3} (where $\Psi_{k,l}$ are defined in terms of $\phi_{k,l}$ by \eqref{eq:Psikl1}--\eqref{eq:Psikl2}).  Recall that these definitions imply that \eqref{eq:rhoklf2} holds with $k=K$ (see Remark \ref{rmk:stillholds}) which is an important component of the estimate \eqref{eq:proprho2} of Proposition \ref{prop:rhoKdef}.
	
	The proof of \eqref{eq:thmapprox2} and the latter of \eqref{eq:fklsupport}, and the estimate for $\varrho_{K,l}$ in \eqref{eq:flkL2estimate}, for $k =K$, follow from Proposition \ref{prop:rhoKdef}.  The gradient estimate of Proposition \ref{prop:gradelliptic} then implies that $\phi_{K,l}$ satisfy \eqref{eq:phiklassum} (or, equivalently, the estimate for $\phi_{K,l}$ in \eqref{eq:flkL2estimate}).
	The proof of properties \eqref{eq:thmapprox1}, \eqref{eq:fKestimate}--\eqref{eq:fKestimate2} and the former of \eqref{eq:fklsupport} and \eqref{eq:flkL2estimate}, for $k =K$, then follow from Proposition \ref{prop:fKdef}.
\end{proof}

\section{The finite problems}
\label{section:finiteprob}

In this section solutions of Vlasov--Poisson will be considered on certain finite intervals of the form $[T,T_f]$.  ``Remainder quantities'', defined by subtracting the approximate solution of Section \ref{section:expansions}, will be considered, and the solutions are defined so as to have trivial ``finite data'' for the remainder at time $t=T_f$.  These problems are referred to as ``finite problems'' in view of the finite length of the interval $[T,T_f]$.  The proof of the main results, in Section \ref{section:logicofproof}, will follow from considering these finite problems and letting $T_f\to \infty$.

Throughout this section, suppose a smooth compactly supported function $f_{\infty} \colon \mathbb{R}^3 \times \mathbb{R}^3 \to [0,\infty)$, satisfying \eqref{eq:suppfinfty} for some $B>0$, and large times $T_f \gg T_0 \gg 1$ are given, where $T_0$ is assumed to be suitably large so as to satisfy the hypothesis of Theorem \ref{thm:approx}.  Suppose also that $N\geq 6$ is fixed.  Recall the functions $f_{[k]} \colon [T_0,\infty) \times \mathbb{R}^3 \times \mathbb{R}^3 \to \mathbb{R}$, $\varrho_{[k]}, \phi_{[k]} \colon [T_0,\infty) \times \mathbb{R}^3 \to \mathbb{R}$ --- defined, for all $k \geq 1$, explicitly in terms of $f_{\infty}$ --- of Theorem \ref{thm:approx}.

Given a solution $(f,\varrho, \phi)$ of the Vlasov--Poisson system \eqref{eq:VP1}--\eqref{eq:VP2}, define, for each $k\geq 1$, functions $\check{f}_{[k]} \colon [T_0,\infty) \times \mathbb{R}^3 \times \mathbb{R}^3 \to \mathbb{R}$, and $\check{\varrho}_{[k]}, \check{\phi}_{[k]} \colon [T_0,\infty) \times \mathbb{R}^3 \to \mathbb{R}$, by
\begin{equation} \label{eq:defoffcheck}
	\check{f}_{[k]}(t,x,p) = f(t,x,p) - f_{[k]}(t,x,p),
	\qquad
	\check{\varrho}_{[k]}(t,x) = \varrho(t,x) - \varrho_{[k]}(t,x),
	\qquad
	\check{\phi}_{[k]}(t,x) = \phi(t,x) - \phi_{[k]}(t,x).
\end{equation}
These functions are referred to as ``remainders''.

The main result of this section is the following.

\begin{theorem}[Estimates for solutions of the finite problems] \label{thm:bootstrap}
	There exists $k_* = k_*(\mathcal{F}_{\infty}^{N+1})$ large such that the following holds. 
	
	For any fixed $k \geq k_*$, suppose there exists a solution $f\colon [T,T_f] \times \mathbb{R}^3\times \mathbb{R}^3 \to [0,\infty)$ of Vlasov--Poisson \eqref{eq:VP1}--\eqref{eq:VP2}, for some $T_0 \leq T < T_f$, which satisfies the final condition
	\begin{equation} \label{eq:fTf}
	f (T_f,x,p)
	=
	f_{[k]}(T_f,x,p),
\end{equation}
	i.\@e.\@ $\check{f}_{[k]}(T_f,\cdot,\cdot) \equiv 0$, and suppose moreover that $\check{f}_{[k]}$ satisfies the bootstrap assumption
	\begin{equation} \label{eq:ba}
		\sup_{t\in [T,T_f]}
		\sum_{\vert I \vert + \vert J\vert \leq N}
		\big\Vert L^{I} (t^{-1}\partial_p)^{J} \check{f}_{[k]} (t,\cdot,\cdot) \big\Vert_{L^2_x L^2_p}
		\leq
		(\mathcal{F}_{\infty}^{N+2k+4})^2,
	\end{equation}
	and that $\phi$ satisfies the bootstrap assumption, for all $T\leq t \leq T_f$,
	\begin{equation} \label{eq:ba2}
		\sup_{x\in \mathbb{R}^3}
		\left\vert
		\nabla \phi (t,x) - t^{-2} \nabla \phi_{\infty} \left(\frac{x}{t} \right)
		\right\vert
		\leq
		\frac{\mathcal{F}_{\infty}^{N+2k+4}}{t^{\frac{5}{2}}}.
	\end{equation}
	Then, if $T_0$ is sufficiently large (with respect to $k$ and $\mathcal{F}_{\infty}^{N+2k+4}$), the solution in fact satisfies 
	\begin{align} \label{eq:bathm1}
		\sum_{\vert I \vert + \vert J \vert \leq N}
		\Vert L^I (t^{-1}\partial_p)^J \check{f}_{[k]} (t,\cdot,\cdot) \Vert_{L^2_x L^2_p}
		&
		\leq
		C_k
		(\mathcal{F}_{\infty}^{N+2k+4})^2
		\frac{ (\log t)^{1+k}
		}{t^{1+ k}},
	\\
	\label{eq:bathm2}
		\sum_{\vert I \vert \leq N}
		\Vert (t\partial_x)^I \check{\varrho}_{[k]} (t,\cdot) \Vert_{L^2_x}
		&
		\leq
		C_k \mathcal{F}_{\infty}^{N+2k+4}
		\frac{(\log t)^{1+k}}{t^{\frac{5}{2} + k}},
	\end{align}
	and
	\begin{align} \label{eq:bathm3}
		\sum_{\vert I \vert \leq N}
		\Vert (t\partial_x)^I \nabla \check{\phi}_{[k]} (t,\cdot) \Vert_{L^2}
		\leq
		C_{k} \mathcal{F}_{\infty}^{N+2k+4}
		\frac{(\log t)^{1+k}}{t^{\frac{3}{2} + k}}.
	\end{align}
	In particular, if $T_0$ is sufficiently large, then \eqref{eq:ba} holds with the right hand side replaced by $\frac{1}{2} (\mathcal{F}_{\infty}^{N+2k+4})^2$ and \eqref{eq:ba2} holds with the right hand side replaced by $\frac{\mathcal{F}_{\infty}^{N+2k+4}}{2t^{5/2}}$.  Finally,
	\begin{equation} \label{eq:suppoffcheck}
		\supp(\check{f}_{[k]}(t,\cdot,\cdot))
		\subset
		\{
		(x,p) \subset \mathbb{R}^3 \times \mathbb{R}^3
		\mid
		\vert x - t p + \log t \nabla \phi_{\infty}(p) \vert \leq 2\mathcal{F}_{\infty}^3 + B, \vert p \vert \leq 2 B
		\},
	\end{equation}
	where $B>0$ is as in \eqref{eq:suppfinfty}.
\end{theorem}

An overview of the main aspects of the proof of of Theorem \ref{thm:bootstrap} is given in Section \ref{subsec:introexistence}.

Recall again the convention, discussed at the beginning of Section \ref{section:theorem}, for the suppressed sequence appearing in $\mathcal{F}_{\infty}^n$.  We remark, in particular, that if $\{\mathfrak{n}(n)\}$ is the suppressed sequence appearing in \eqref{eq:ba}--\eqref{eq:ba2}, then the suppressed sequence in \eqref{eq:bathm1}--\eqref{eq:bathm3} will be a different sequence $\{\mathfrak{n}'(n)\}$, where $\mathfrak{n}'(n) \geq \mathfrak{n}(n)$ for all $n$.

The results of this section will all be in the above setting.
In particular, $N \geq 6$ will be assumed to be fixed throughout, and the dependence of constants on $N$ is suppressed.
In Section \ref{subsec:remaindersystem} the system of equations satisfied by the remainder quantities is obtained.  In Section \ref{subsec:fcheckkestimates} these reminder quantities are estimated, and in Section \ref{subsec:proofofbootstrap} the proof of Theorem \ref{thm:bootstrap} is completed.

\subsection{System of equations for remainders}
\label{subsec:remaindersystem}

The proof of Theorem \ref{thm:bootstrap} will proceed by considering the remainders \eqref{eq:defoffcheck}.  Recall that the final condition \eqref{eq:fTf} is defined so that the remainder $\check{f}_{[k]}$ vanishes at time $t=T_f$:
\begin{equation} \label{eq:fcheckTf}
	\check{f}_{[k]}(T_f,x,p)
	=
	0,
	\qquad
	\text{for all } x,p \in \mathbb{R}^3.
\end{equation}
The remainders satisfy the system of equations,
\begin{equation} \label{eq:differenceeqns}
	\gs_{\phi} \check{f}_{[k]}
	=
	F_{[k]},
	\qquad
	\Delta \check{\phi}_{[k]}
	=
	\check{\varrho}_{[k]},
	\qquad
	\check{\varrho}_{[k]}(t,x) - \int_{\mathbb{R}^3} \check{f}_{[k]}(t,x,p) dp
	=
	P_{[k]},
\end{equation}
where the functions $F_{[k]} \colon [T_0,\infty) \times \mathbb{R}^3 \times \mathbb{R}^3 \to \mathbb{R}$ and $P_{[k]} \colon [T_0,\infty) \times \mathbb{R}^3 \to \mathbb{R}$ are given by
\begin{equation} \label{eq:defofFk}
	F_{[k]}
	=
	\gs_{\phi} \check{f}_{[k]} = - \gs_{\phi} f_{[k]} = - \partial_{x^i} \check{\phi}_{[k]} \partial_{p^i} f_{[k]} - \gs_{\phi_{[k]}} f_{[k]},
\end{equation}
and
\[
	P_{[k]}(t,x)
	=
	\check{\varrho}_{[k]}(t,x) - \int_{\mathbb{R}^3} \check{f}_{[k]}(t,x,p) dp
	=
	- \varrho_{[k]}(t,x) + \int_{\mathbb{R}^3} f_{[k]}(t,x,p) dp.
\]
The relevant properties of $F_{[k]}$ and $P_{[k]}$ are given by the following proposition, which is a consequence of Theorem \ref{thm:approx}.

\begin{proposition}[Estimates for $F_{[k]}$ and $P_{[k]}$] \label{prop:systemfcheck}
	For each $k \geq 0$, if $T_0 = T_0(\mathcal{F}_{\infty}^{N+k})$ is sufficiently large, the functions $F_{[k]}$ and $P_{[k]}$ satisfy, for all $t \geq T_0$,
	\begin{equation} \label{eq:Fk}
		\sum_{\vert I \vert + \vert J \vert \leq N}
		\big\Vert L^I (t^{-1}\partial_p)^J F_{[k]} (t,\cdot, \cdot) \big\Vert_{L^2_x L^2_p}
		\leq
		\frac{C \mathcal{F}_{\infty}^{N+1}}{t^{\frac{3}{2}}}
		\sum_{\vert I\vert =1}^{N+1}
		\Vert (t \partial_x)^I \check{\phi}_{[k]}(t,\cdot) \Vert_{L^2}
		+
		C_{k}
		(\mathcal{F}_{\infty}^{N+2k+3})^2
		\frac{(\log t)^{1+k}}{t^{2+k}}
		,
	\end{equation}
	where the constant $C$ does not depend on $k$, and
	\begin{equation} \label{eq:Pk}
		\sum_{\vert I \vert \leq N}
		\big\Vert (t\partial_x)^I P_{[k]} (t,\cdot) \big\Vert_{L^2_x}
		\leq
		C_{k} \mathcal{F}_{\infty}^{N+2k+3}
		\frac{(\log t)^{1+k}}{t^{\frac{5}{2} + k}}
		.
	\end{equation}
\end{proposition}

\begin{proof}[Proof of Proposition \ref{prop:systemfcheck}]
	The estimate \eqref{eq:Pk} follows immediately from Theorem \ref{thm:approx} (see \eqref{eq:thmapprox2}).
	
	For \eqref{eq:Fk}, consider first the first term on the right hand side of \eqref{eq:defofFk}.  Note that, since $N \geq 3$,
	\begin{align*}
		&
		\sum_{\vert I \vert + \vert J \vert \leq N}
		\big\Vert L^I (t^{-1}\partial_p)^J \big( \partial_{x^i} \check{\phi}_{[k]} \partial_{p^i} f_{[k]} \big) (t,\cdot, \cdot) \big\Vert_{L^2_x L^2_p}
		\lesssim
		\\
		&
		\qquad\qquad \qquad
		\sum_{\vert I \vert = 1}^{N-2}
		\sup_{x\in \mathbb{R}^3} \big\vert (t\partial_x)^I \check{\phi}_{[k]} (t,\cdot) \big\vert
		\sum_{\substack{\vert I \vert + \vert J \vert \leq N+1\\ \vert J \vert \geq 1}}
		\big\Vert L^I (t^{-1}\partial_p)^J f_{[k]} (t,\cdot, \cdot) \big\Vert_{L^{2}_x L^2_p}
		\\
		&
		\qquad\qquad \qquad
		+
		\sum_{\vert I \vert = 1}^{N+1}
		\big\Vert (t\partial_x)^I \check{\phi}_{[k]} (t,\cdot) \big\Vert_{L^{2}_x}
		\sum_{\substack{\vert I \vert + \vert J \vert \leq N-2\\ \vert J \vert \geq 1}}
		\sup_{x\in \mathbb{R}^3}  \big\Vert L^I (t^{-1}\partial_p)^J f_{[k]} (t,x, \cdot) \big\Vert_{L^2_p},
	\end{align*}
	It follows from the Sobolev inequality --- see Remark \ref{rmk:Sobolev} --- that
	\[
		\sup_{x\in \mathbb{R}^3} 
		\sum_{\vert I\vert =1}^{N-2} \vert (t \partial_x)^I \check{\phi}_{[k]} (t,x) \vert
		\leq
		C
		t^{-\frac{3}{2}} \sum_{\vert I\vert =1}^N \Vert (t \partial_x)^I \check{\phi}_{[k]} (t,\cdot) \Vert_{L^2_x},
	\]
	and similarly, by Theorem \ref{thm:approx} (see \eqref{eq:fKestimate}) and the Sobolev inequality of Proposition \ref{prop:Sobolevvf},
	\begin{multline*}
		\sum_{\vert I \vert + \vert J \vert \leq N-2}
		\sup_{x\in \mathbb{R}^3}
		t^{\frac{3}{2}}
		\big\Vert L^I (t^{-1}\partial_p)^J f_{[k]} (t,x, \cdot) \big\Vert_{L^2_p}
		+
		\sum_{\vert I \vert + \vert J \vert \leq N+1}
		\big\Vert L^I (t^{-1}\partial_p)^J f_{[k]} (t,\cdot, \cdot) \big\Vert_{L^2_xL^2_p}
		\\
		\leq
		C\mathcal{F}_{\infty}^{N+1}
		+
		\frac{C_{k} \mathcal{F}_{\infty}^{N+k} \log t}{t}
		\leq
		C\mathcal{F}_{\infty}^{N+1},
	\end{multline*}
	where the constant $C$ does not depend on $k$, and the latter inequality follows after taking $T_0$ suitably large.  The second term on the right hand side of \eqref{eq:defofFk} is estimated in Theorem \ref{thm:approx} (see \eqref{eq:thmapprox1}).  The proof of \eqref{eq:Fk} then follows.
\end{proof}

In the following proposition the size of the support of $\check{f}_{[k]}$ is estimated.

\begin{proposition}[The support of $\check{f}_{[k]}$] \label{prop:suppf}
	If $T_0 = T_0(\mathcal{F}_{\infty}^{N+2k+4})$ is sufficiently large, then
	\begin{align*}
		\supp(\check{f}_{[k]}(t,\cdot,\cdot))
		&
		\subset
		\{
		(x,p) \subset \mathbb{R}^3 \times \mathbb{R}^3
		\mid
		\vert x - t p + \log t \nabla \phi_{\infty}(p) \vert \leq 2\mathcal{F}_{\infty}^3 + B, \vert p \vert \leq 2 B
		\}
		\\
		&
		\subset
		\{ (x,p) \subset \mathbb{R}^3 \times \mathbb{R}^3 \mid \vert x \vert \leq 2\mathcal{F}_{\infty}^3 + 3 B t, \vert p \vert \leq 2 B \}.
	\end{align*}
\end{proposition}

\begin{proof}
	Theorem \ref{thm:approx}, in particular the property \eqref{eq:fklsupport}, implies that
	\[
		\supp(f_{[k]})
		\subset
		\{
		(x,p) \subset \mathbb{R}^3 \times \mathbb{R}^3
		\mid
		\vert x - t p + \log t \nabla \phi_{\infty}(p) \vert \leq B, \vert p \vert \leq B
		\}.
	\]
	Consider, therefore the support of $f$.
	Note that, for $f_{T_f}(x,p) := f(T_f,x,p)$,
	\[
		\vert x-T_f p + \log T_f \, \nabla \phi_{\infty} (p) \vert
		\leq
		B,
		\qquad
		\text{for all } (x,p) \in \supp(f_{T_f}),
	\]
	which follows from the definition \eqref{eq:fTf} of $f_{T_f}$ and the fact that each $f_{k,l}$ satisfies \eqref{eq:fklsupport}.  Define the characteristic curves of $\gs$, denoted $X, P \colon [T,T_f] \times [T,T_f] \times \mathbb{R}^3 \times \mathbb{R}^3 \to \mathbb{R}^3$, as solutions of
	\[
		\frac{dX(s,t,x,p)}{ds} = P(s,t,x,p),
		\qquad
		\frac{dP(s,t,x,p)}{ds}
		=
		(\nabla \phi)(s,X(s,t,x,p)),
	\]
	along with the initial conditions
	\[
		X(t,t,x,p) = x,
		\qquad
		P(t,t,x,p) = p.
	\]
	Clearly, for any $x,p \in \mathbb{R}^3$, $T \leq s \leq T_f$,
	\[
		f(s,X(s,T_f,x,p), P(s,T_f,x,p)) = f(T_f,x,p),
	\]
	and so
	\begin{equation} \label{eq:suppfcharacteristics}
		\supp(f)
		=
		\{ (s,X(s,T_f,x,p), P(s,T_f,x,p)) \mid (x,p) \in \supp(f_{T_f}), T \leq s \leq T_f \}.
	\end{equation}
	We proceed by subtracting from $X$ and $P$ suitable approximate solutions of the characteristic equations and estimating the difference (cf.\@ \cite{LiTa}, where similar approximate solutions of the characteristic equations are introduced in the context of the Einstein--Vlasov system).  Define
	\[
		\overline{X}(s,t,x,p) = X(s,t,x,p) - \Big( x - (t-s)p - \log \Big( \frac{s}{t} \Big) \nabla \phi_{\infty}(p) \Big),
		\qquad
		\overline{P}(s,t,x,p) = P(s,t,x,p) - \Big( p - \frac{1}{s} \nabla \phi_{\infty}(p) \Big),
	\]
	so that
	\begin{equation} \label{eq:chareqns}
		\frac{d\overline{X}(s,t,x,p)}{ds} = \overline{P}(s,t,x,p),
		\qquad
		\frac{d\overline{P}(s,t,x,p)}{ds}
		=
		(\nabla \phi)(s,X(s,t,x,p)) - \frac{1}{s^2} \nabla \phi_{\infty}(p),
	\end{equation}
	along with the initial conditions
	\[
		X(t,t,x,p) = 0,
		\qquad
		P(t,t,x,p) = \frac{1}{t} \nabla \phi_{\infty}(p).
	\]
	
	Consider some $(x,p) \in \supp(f_{T_f})$, $T \leq t \leq s \leq T_f$.  Using the fact that
	\[
		\sup_{x\in \mathbb{R}^3}
		\left\vert
		\nabla \phi (t,x) - t^{-2} \nabla \phi_{\infty} \left(\frac{x}{t} \right)
		\right\vert
		\leq
		\frac{\mathcal{F}_{\infty}^{N+2k+4}}{t^{\frac{5}{2}}},
		\qquad
		\sup_{p\in \mathbb{R}^3} \Big( \vert \nabla \phi_{\infty}(p) \vert + \vert \nabla^2 \phi_{\infty}(p) \vert \Big) \lesssim \mathcal{F}_{\infty}^{3},
	\]
	(see the bootstrap assumption \eqref{eq:ba2} and the fact \eqref{eq:phiinftypointwise}) and,
	\begin{align*}
		\left\vert \frac{X(s,T_f,x,p)}{s} - p \right\vert
		&
		=
		s^{-1}\left\vert \overline{X}(s,T_f,x,p) - \log s \, \nabla \phi_{\infty}(p) + x - T_f p + \log T_f \nabla \phi_{\infty}(p) \right\vert
		\\
		&
		\lesssim
		\frac{\vert \overline{X}(s,T_f,x,p) \vert}{s} + \frac{\log s}{s} \mathcal{F}_{\infty}^{3} + \frac{B}{s}
		,
	\end{align*}
	it follows that
	\begin{align*}
		\left\vert \frac{d\overline{P}(s,T_f,x,p)}{ds} \right\vert
		&
		\lesssim
		\left\vert (\nabla \phi)(s,X(s,T_f,x,p)) - \frac{1}{s^2} (\nabla \phi_{\infty}) \left(\frac{X(s,T_f,x,p)}{s} \right) \right\vert
		\\
		&
		\quad
		+
		\frac{1}{s^2} \left\vert (\nabla \phi_{\infty}) \left( \frac{X(s,T_f,x,p)}{s} \right) - (\nabla \phi_{\infty})(p) \right\vert
		\\
		&
		\lesssim
		\mathcal{F}_{\infty}^{3} \frac{\vert \overline{X}(s,T_f,x,p) \vert}{s^3}
		+
		\frac{\mathcal{F}_{\infty}^{N+2k+4}}{s^{\frac{5}{2}}}
		+
		\frac{\mathcal{F}_{\infty}^{3} (B + \log s \mathcal{F}_{\infty}^{3})}{s^3}.
	\end{align*}
	Integrating the equations \eqref{eq:chareqns} backwards from time $T_f$, the fact that
	\[
		\int_t^{T_f} \int_s^{T_f} \frac{\vert \overline{X}(s',T_f,x,p) \vert}{(s')^3} ds' ds
		=
		\int_t^{T_f} (s-t) \frac{\vert \overline{X}(s,T_f,x,p) \vert}{s^3} ds,
	\]
	then implies that
	\begin{align*}
		\vert \overline{X}(t,T_f,x,p) \vert
		&
		\lesssim
		\vert \nabla \phi_{\infty}(p) \vert
		+
		\frac{\mathcal{F}_{\infty}^{N+2k+4}}{t^{\frac{1}{2}}}
		+
		\frac{\mathcal{F}_{\infty}^{3} (B + \log t \mathcal{F}_{\infty}^{3})}{t}
		+
		\mathcal{F}_{\infty}^{3} \int_t^{T_f} \frac{\vert \overline{X}(s,T_f,x,p) \vert}{s^2} ds
		\\
		&
		\leq
		\frac{3}{2}\mathcal{F}_{\infty}^{3}
		+
		C\mathcal{F}_{\infty}^{3} \int_t^{T_f} \frac{\vert \overline{X}(s,T_f,x,p) \vert}{s^2} ds
		,
	\end{align*}
	if $T_0$ is suitably large.
	The Gr\"{o}nwall inequality, Proposition \ref{prop:Gronwall}, then gives
	\[
		\vert \overline{X}(t,T_f,x,p) \vert + t \vert \overline{P}(t,T_f,x,p) \vert
		\leq
		\frac{3}{2}\mathcal{F}_{\infty}^3
		+
		C(\mathcal{F}_{\infty}^3)^2 \int_t^{T_f} \frac{1}{s^2} e^{\int_t^s \frac{C\mathcal{F}_{\infty}^{3}}{(s')^2}ds'} ds
		\leq
		\frac{3}{2}\mathcal{F}_{\infty}^3
		+
		\frac{C(\mathcal{F}_{\infty}^3)^2 e^{\frac{C\mathcal{F}_{\infty}^{3}}{t}}}{t}.
	\]
	It then follows that
	\[
		\big\vert
		X(t,T_f,x,p) - t p + \log t \nabla \phi_{\infty}(p) \big)
		\big\vert
		+
		t
		\vert
		P(t,T_f,x,p) - p
		\vert
		\leq
		2\mathcal{F}_{\infty}^3 + B,
	\]
	which, by \eqref{eq:suppfcharacteristics}, completes the proof.
\end{proof}

\subsection{Estimates for remainders}
\label{subsec:fcheckkestimates}

In this section the remainder quantities $\check{f}_{[k]}$, $\check{\varrho}_{[k]}$, $\check{\phi}_{[k]}$ are estimated.  The main result of this section is the following estimate for $\check{f}_{[k]}$.

\begin{proposition}[Estimates for $L^I (t^{-1}\partial_p)^J\check{f}_{[k]}$] \label{prop:fcheckkho}
	Under the assumptions of Theorem \ref{thm:bootstrap}, if $k$ is sufficiently large ($k> (C+1) \mathcal{F}_{\infty}^{N+1}-1$ suffices, where $C$ is the constant from \eqref{eq:Fk}), and $T_0$ is sufficiently large (depending on $k$ and $\mathcal{F}_{\infty}^{N+2k+4}$) then, for all $T\leq t \leq T_f$,
	\begin{equation} \label{eq:fcheckkho2}
		\sum_{\vert I \vert + \vert J \vert \leq N}
		\Vert L^I (t^{-1}\partial_p)^J \check{f}_{[k]} (t,\cdot,\cdot) \Vert_{L^2_x L^2_p}
		\leq
		C_k
		(\mathcal{F}_{\infty}^{N+2k+4})^2
		\frac{ (\log t)^{1+k}
		}{t^{1+ k}}.
	\end{equation}
\end{proposition}

The reader is again referred to Section \ref{subsec:introexistence} for an overview of the proof of Proposition \ref{prop:fcheckkho}.
The proof requires the following preliminary results.

\begin{proposition}[Estimates for $(t\partial_x)^I \check{\phi}_{[k]}$] \label{prop:phicheckho}
	For any multi-index $I$ such that $\vert I \vert \leq N$, for all $k \geq 0$ and all $T\leq t \leq T_f$,
	\[
		\Vert \nabla (t\partial_x)^I \check{\phi}_{[k]} (t,\cdot) \Vert_{L^2}
		\lesssim
		t \Vert (t\partial_x)^I \check{\varrho}_{[k]} (t,\cdot) \Vert_{L^2}.
	\]
\end{proposition}

\begin{proof}
	Recall that
	\[
		\Delta_{\mathbb{R}^3} \check{\phi}_{[k]}
		=
		\check{\varrho}_{[k]}.
	\]
	The proof follows from Proposition \ref{prop:gradelliptic} and the fact that $\vert x \vert \leq 3Bt$ in $\mathrm{supp}(\check{\varrho}_{[k]})$ (see Proposition \ref{prop:suppf}, and the property \eqref{eq:fklsupport}), since $t \partial_{x^i}$ commutes with $\Delta_{\mathbb{R}^3}$.
\end{proof}

\begin{proposition}[Estimates for $(t\partial_x)^I \check{\varrho}_{[k]}$] \label{prop:rhocheckho}
	For any multi-index $I$ such that $\vert I \vert \leq N$, for all $k \geq 0$ and all $T\leq t \leq T_f$,
	\[
		\Vert (t\partial_x)^I \check{\varrho}_{[k]} (t,\cdot) \Vert_{L^2_x}
		\lesssim
		t^{-\frac{3}{2}} \Vert L^I \check{f}_{[k]} (t,\cdot,\cdot) \Vert_{L^2_x L^2_p}
		+
		C_{k} \mathcal{F}_{\infty}^{N+2k+3}
		\frac{(\log t)^{1+k}}{t^{\frac{5}{2} + k}}.
	\]
\end{proposition}

\begin{proof}
	Recall (see Proposition \ref{prop:systemfcheck}) that
	\[
		\check{\varrho}_{[k]}(t,x) 
		=
		\int_{\mathbb{R}^3} \check{f}_{[k]}(t,x,p) dp
		+
		P_{[k]},
	\]
	where $P_{[k]}$ satisfies \eqref{eq:Pk}.
	Since
	\[
		\int_{\mathbb{R}^3} \partial_{p^i} L^J \check{f}_{[k]} (t,x,p) dp = 0,
	\]
	for $i=1,2,3$ and each multi-index $J$, it follows that, for any multi-index $I$,
	\[
		(t\partial_x)^I \check{\varrho}_{[k]} (t,x) = \int_{\mathbb{R}^3} L^I \check{f}_{[k]} (t,x,p) dp
		+
		(t\partial_x)^I P_{[k]} (t,x).
	\]
	Hence
	\begin{multline*}
		\Vert (t\partial_x)^I \check{\varrho}_{[k]} (t,\cdot) \Vert_{L^2_x}
		\lesssim
		\Big( \int_{\mathbb{R}^3} \Big( \int_{\mathbb{R}^3} L^I \check{f}_{[k]} (t,x,p) dp \Big)^2 dx \Big)^{\frac{1}{2}}
		+
		\Vert (t\partial_x)^I P_{[k]} (t,\cdot) \Vert_{L^2_x}
		\\
		\lesssim
		\Big( 
		\int_{\mathbb{R}^3} \int_{\mathbb{R}^3} \mathds{1}_{\check{f}_{[k]}} (t,x,p) dp
		\int_{\mathbb{R}^3} \vert L^I \check{f}_{[k]} (t,x,p) \vert^2 dp dx
		\Big)^{\frac{1}{2}}
		+
		\Vert (t\partial_x)^I P_{[k]} (t,\cdot) \Vert_{L^2_x}.
	\end{multline*}
	Note now that
	\[
		\int_{\mathbb{R}^3} \mathds{1}_{\supp(\check{f}_{[k]})} (t,x,p) dp \lesssim \frac{1 + (\mathcal{F}_{\infty}^3)^3}{t^3},
	\]
	which follows from the fact that
	\[
		\Big\vert
		\frac{x}{t}
		+
		\frac{\log t}{t} \nabla\phi_{\infty} \Big( \frac{x}{t} \Big)
		- p
		\Big\vert
		\leq
		\frac{3\mathcal{F}_{\infty}^3 + 2B}{t}
		\quad
		\text{in}
		\quad
		\supp(\check{f}_{[k]}).
	\]
	See Proposition \ref{prop:suppf} (and cf.\@ Proposition \ref{prop:pexpansion}).
	The proof then follows from the estimate \eqref{eq:Pk}.
\end{proof}

In the following proposition, the terms arising from commuting the equation \eqref{eq:differenceeqns} for $\check{f}_{[k]}$ with the vector fields of Section \ref{subsec:vectorfields} are estimated, using the bootstrap assumption \eqref{eq:ba}.

\begin{proposition}[Commuted equation for $\check{f}_{[k]}$] \label{prop:fcheckkcommuted}
	If the bootstrap assumption \eqref{eq:ba} holds, then
	\begin{multline} \label{eq:commuted2}
		\sum_{\vert I \vert + \vert J \vert \leq N}
		\Vert \gs_{\phi} (L^I (t^{-1}\partial_p)^J \check{f}_{[k]}) \Vert_{L^2_x L^2_p}
		\lesssim
		\sum_{\vert I \vert + \vert J \vert \leq N}
		\Vert L^I (t^{-1}\partial_p)^J F_{[k]} \Vert_{L^2_x L^2_p}
		\\
		+
		\frac{C_{k} \mathcal{F}_{\infty}^{N+2k+3} (\log t)^2}{t^2}
		\sum_{\vert I \vert + \vert J \vert =1}^N
		\Vert L^{I} (t^{-1}\partial_p)^{J} \check{f}_{[k]} \Vert_{L^2_x L^2_p}
		+
		(\mathcal{F}_{\infty}^{N+2k+4})^2
		\frac{\log t}{t^{\frac{5}{2}}} 
		\sum_{\vert K \vert = 1}^{N+1} \Vert (t \partial_x)^K \check{\phi}_{[k]} \Vert_{L^2_x}
		.
	\end{multline}
\end{proposition}

\begin{proof}
	Recall Proposition \ref{prop:commVlasovLS}, which implies that, for $i=1,2,3$,
	\begin{align*}
		\Big[
		\gs_{\phi}
		,
		L_i
		\Big]
		=
		\
		&
		-
		\frac{\log t}{t^2} \partial_i \partial_{j} \phi_{\infty}(p) \,
		L_j
		+
		\Big[
		t^{-1} \Big( \partial_i \partial_{k} \phi_{\infty}(p) - \partial_i \partial_{k} \phi_{\infty}\Big( \frac{x}{t} \Big) \Big)
		+
		\frac{\log t}{t^2} \partial_{l} \phi_{\infty}(p) \, \partial_i \partial_k \partial_{l} \phi_{\infty}(p)
		\\
		&
		+
		\frac{(\log t)^2}{t^2} \partial_i \partial_{j} \phi_{\infty}(p) 
		\partial_j \partial_k \phi_{\infty}(p)
		-
		t \partial_{x^i} \partial_{x^k} \check{\phi}_{[0]}(t,x)
		+
		\log t \, \partial_i \partial_k \partial_{l} \phi_{\infty}(p) \, \partial_{x^l} \check{\phi}_{[0]}(t,x)
		\Big]
		t^{-1} \partial_{p^k}
		,
	\end{align*}
	and
	\begin{align*}
		\Big[
		\gs_{\phi}
		,
		t^{-1} \partial_{p^i}
		\Big]
		=
		-
		\frac{1}{t^2}
		L_i
		+
		\frac{\log t}{t^2} \partial_i \partial_k \phi_{\infty}(p) \, t^{-1} \partial_{p^k}
		.
	\end{align*}
	Note that, for any smooth $h \colon \mathbb{R}^3 \to \mathbb{R}$,
	\[
		L_i \Big( h(p) - h \Big( \frac{x}{t} \Big) \Big)
		=
		(\partial_i h)(p) - (\partial_i h) \Big( \frac{x}{t} \Big)
		+
		\frac{\log t}{t} \partial_i \partial_k \phi_{\infty}(p) (\partial_k h)(p)
		,
	\]
	and so
	\[
		\sup_{\vert y(t,x,p) \vert \leq B} \Big\vert L_i \Big( h(p) - h \Big( \frac{x}{t} \Big) \Big) \Big\vert
		\lesssim
		\Vert \nabla^2 h \Vert_{L^{\infty}}
		\frac{2B + \mathcal{F}_{\infty}^{3} \log t}{t}
		+
		\Vert \nabla h \Vert_{L^{\infty}} \frac{\mathcal{F}_{\infty}^3 \log t}{t}
		.
	\]
	It follows from a simple induction argument, along with \eqref{eq:phiinftypointwise} and the fact that, for any $n \geq 0$ and any multi-indices $I$, $J$ such that $\vert I \vert + \vert J \vert \leq n$,
	\begin{multline} \label{eq:commuted1}
		\Vert \gs_{\phi} (L^I (t^{-1}\partial_p)^J \check{f}_{[k]}) \Vert_{L^2_x L^2_p}
		\leq
		\Vert L^I (t^{-1}\partial_p)^J F_{[k]} \Vert_{L^2_x L^2_p}
		+
		C
		\frac{\mathcal{F}_{\infty}^{n+3} (\log t)^2}{t^2}
		\sum_{\vert I ' \vert + \vert J'\vert =1}^n
		\Vert L^{I'} (t^{-1}\partial_p)^{J'} \check{f}_{[k]} \Vert_{L^2_x L^2_p}
		\\
		+
		C
		(1 + \mathcal{F}_{\infty}^{n+3})
		\,
		\frac{\log t}{t}
		\Big(
		\sum_{\vert K \vert = 1}^{n+1} \Vert (t \partial_x)^K \check{\phi}_{[0]} \Vert_{L^2_x}
		\sum_{\vert I ' \vert + \vert J'\vert =1}^{\lfloor \frac{n+2}{2} \rfloor}
		\Vert L^{I'} (t^{-1}\partial_p)^{J'} \check{f}_{[k]} \Vert_{L^{\infty}_x L^2_p}
		\\
		+
		\sum_{\vert K \vert = 1}^{\lfloor \frac{n+4}{2} \rfloor} \Vert (t \partial_x)^K \check{\phi}_{[0]} \Vert_{L^{\infty}_x}
		\sum_{\vert I ' \vert + \vert J'\vert =1}^n
		\Vert L^{I'} (t^{-1}\partial_p)^{J'} \check{f}_{[k]} \Vert_{L^2_x L^2_p}
		\Big)
		,
	\end{multline}
	The Sobolev inequality of Proposition \ref{prop:Sobolevvf} and the fact \eqref{eq:phiinftypointwise} (since $N \geq 3$), implies that the right hand side of \eqref{eq:commuted1} is controlled by
	\begin{multline*} 
		\sum_{\vert I \vert + \vert J \vert \leq N}
		\Vert L^I (t^{-1}\partial_p)^J F_{[k]} \Vert_{L^2_x L^2_p}
		\\
		+
		\Big(
		\frac{\mathcal{F}_{\infty}^{N+3} (\log t)^2}{t^2}
		+
		(1 + \mathcal{F}_{\infty}^{N+3})
		\,
		\frac{\log t}{t^{\frac{5}{2}}} 
		\sum_{\vert K \vert = 1}^{N+1} \Vert (t \partial_x)^K \check{\phi}_{[0]} \Vert_{L^2_x}
		\Big)
		\sum_{\vert I \vert + \vert J \vert =1}^N
		\Vert L^{I} (t^{-1}\partial_p)^{J} \check{f}_{[k]} \Vert_{L^2_x L^2_p}
		.
	\end{multline*}
	Note now that $\check{\phi}_{[0]} = \check{\phi}_{[k]} + (\phi_{[k]} - \phi_{[0]})$, and that Proposition \ref{prop:phicheckho} and Proposition \ref{prop:rhocheckho} imply that
	\[
		\sum_{\vert I\vert =1}^{N+1}
		\Vert (t \partial_x)^I (\phi_{[k]} - \phi_{[0]})(t,\cdot) \Vert_{L^2}
		\lesssim
		t^{\frac{1}{2}} 
		\sum_{\vert I \vert \leq N} \Vert L^I (f_{[k]} - f_{[0]}) (t,\cdot,\cdot) \Vert_{L^2_x L^2_p}
		\leq
		\frac{C_{k} \mathcal{F}_{\infty}^{N+2k+3} \log t}{t^{\frac{1}{2}}},
	\]
	where \eqref{eq:fKestimate2} has been used.  It follows that
	\begin{multline*} 
		\sum_{\vert I \vert + \vert J \vert \leq N}
		\Vert \gs_{\phi} (L^I (t^{-1}\partial_p)^J \check{f}_{[k]}) \Vert_{L^2_x L^2_p}
		\lesssim
		\sum_{\vert I \vert + \vert J \vert \leq N}
		\Vert L^I (t^{-1}\partial_p)^J F_{[k]} \Vert_{L^2_x L^2_p}
		\\
		+
		\Big(
		\frac{C_{k} \mathcal{F}_{\infty}^{N+2k+3} (\log t)^2}{t^2}
		+
		(1 + \mathcal{F}_{\infty}^{N+3})
		\,
		\frac{\log t}{t^{\frac{5}{2}}} 
		\sum_{\vert K \vert = 1}^{N+1} \Vert (t \partial_x)^K \check{\phi}_{[k]} \Vert_{L^2_x}
		\Big)
		\sum_{\vert I \vert + \vert J \vert =1}^N
		\Vert L^{I} (t^{-1}\partial_p)^{J} \check{f}_{[k]} \Vert_{L^2_x L^2_p}
		.
	\end{multline*}
	The proof of \eqref{eq:commuted2} follows from inserting the bootstrap assumption \eqref{eq:ba}.
	
\end{proof}

The proof of Proposition \ref{prop:fcheckkho} can now be given.

\begin{proof}[Proof of Proposition \ref{prop:fcheckkho}]
	If $T_0=T_0(\mathcal{F}_{\infty}^{N+2k+4})$ is suitably large so that the coefficient of the final term in Proposition \ref{prop:fcheckkcommuted} satisfies
	\[
		(\mathcal{F}_{\infty}^{N+2k+4})^2
		\frac{\log t}{t}
		\leq
		\mathcal{F}_{\infty}^{N+1},
	\]
	then
	Proposition \ref{prop:systemfcheck} and Proposition \ref{prop:fcheckkcommuted} imply that,
	\begin{align*}
		\sum_{\vert I \vert + \vert J \vert \leq N}
		\Vert \gs_{\phi} (L^I (t^{-1}\partial_p)^J \check{f}_{[k]}) \Vert_{L^2_x L^2_p}
		\leq
		\
		&
		C_{k}
		(\mathcal{F}_{\infty}^{N+2k+3})^2
		\frac{(\log t)^{1+k}}{t^{2+k}}
		+
		\frac{C \mathcal{F}_{\infty}^{N+1}}{t^{\frac{3}{2}}}
		\sum_{\vert I\vert =1}^{N+1}
		\Vert (t \partial_x)^I \check{\phi}_{[k]}(t,\cdot) \Vert_{L^2}
		\\
		&
		+
		\frac{C_{k} \mathcal{F}_{\infty}^{N+2k+3} (\log t)^2}{t^2}
		\sum_{\vert I \vert + \vert J \vert =1}^N
		\Vert L^{I} (t^{-1}\partial_p)^{J} \check{f}_{[k]} \Vert_{L^2_x L^2_p}
		.
	\end{align*}
	Now Proposition \ref{prop:phicheckho} and Proposition \ref{prop:rhocheckho} imply that
	\[
		\sum_{\vert I\vert =1}^{N+1}
		\Vert (t \partial_x)^I \check{\phi}_{[k]}(t,\cdot) \Vert_{L^2}
		\lesssim
		t^{\frac{1}{2}} 
		\sum_{\vert I \vert \leq N} \Vert L^I \check{f}_{[k]} (t,\cdot,\cdot) \Vert_{L^2_x L^2_p}
		+
		C_{k} \mathcal{F}_{\infty}^{N+2k+3}
		\frac{(\log t)^{1+k}}{t^{\frac{1}{2} + k}}.
	\]
	Thus, if $T_0 = T_0(\mathcal{F}_{\infty}^{N+2k+3},k)$ is suitably large so that
	\[
		\frac{C_{k} \mathcal{F}_{\infty}^{N+2k+3} (\log t)^2}{t}
		\leq
		\mathcal{F}_{\infty}^{N+1},
	\]
	then
	\begin{align*}
		\sum_{\vert I \vert + \vert J \vert \leq N}
		\Vert \gs_{\phi} (L^I (t^{-1}\partial_p)^J \check{f}_{[k]}) \Vert_{L^2_x L^2_p}
		\leq
		\frac{C \mathcal{F}_{\infty}^{N+1}}{t}
		\sum_{\vert I \vert + \vert J \vert \leq N}
		\Vert L^I (t^{-1}\partial_p)^J \check{f}_{[k]} \Vert_{L^2_x L^2_p}
		+
		C_{k}
		(\mathcal{F}_{\infty}^{N+2k+4})^2
		\frac{(\log t)^{1+k}}{t^{2+k}}.
	\end{align*}
	Defining
	\[
		\check{\mathcal{E}}_{[k]}(t)
		: =
		\sum_{\vert I \vert +\vert J \vert = 0}^N \Vert L^I (t^{-1} \partial_p)^J \check{f}_{[k]}(t,\cdot,\cdot) \Vert_{L^2_x L^2_p},
	\]
	and noting that the vanishing of $\check{f}_{[k]}$ at time $T_f$ (see \eqref{eq:fcheckTf}) implies that $\check{\mathcal{E}}_{[k]}(T_f) = 0$,
	Proposition \ref{prop:generalVlasovestimates} then implies that, for all $T \leq t \leq T_f$,
	\begin{equation} \label{eq:eqnforEcheckk}
		\check{\mathcal{E}}_{[k]}(t)
		\leq
		C_{k}
		(\mathcal{F}_{\infty}^{N+2k+4})^2
		\frac{(\log t)^{1+k}}{t^{1+k}}
		+
		C \mathcal{F}_{\infty}^{N+1}
		\int_t^{T_f}
		\frac{\check{\mathcal{E}}_{[k]}(s)}{s}
		ds.
	\end{equation}
	Defining
	\[
		a(t)
		=
		\frac{C \mathcal{F}_{\infty}^{N+1}}{t},
		\qquad
		b(t) = C_{k}
		(\mathcal{F}_{\infty}^{N+2k+4})^2
		\frac{(\log t)^{1+k}}{t^{1+k}},
	\]
	it follows that
	\[
		e^{\int_t^s a(s')ds'} = \left( \frac{s}{t} \right)^{C \mathcal{F}_{\infty}^{N+1}}.
	\]
	Thus, if $k> C \mathcal{F}_{\infty}^{N+1} - 1$,
	\begin{align*}
		\int_t^{T_f} a(s) b(s) e^{\int_t^s a(s')ds'}ds
		&
		\leq
		\frac{C_k \mathcal{F}_{\infty}^{N+1}
		(\mathcal{F}_{\infty}^{N+2k+4})^2
		}{t^{C\mathcal{F}_{\infty}^{3}}}
		\int_t^{T_f} \frac{(\log s)^{1+k}}{s^{2 + k - C\mathcal{F}_{\infty}^{N+1}}}ds
		\leq
		\frac{C_k \mathcal{F}_{\infty}^{N+1}
		(\mathcal{F}_{\infty}^{N+2k+4})^2 (\log t)^{1+k}
		}{t^{1+ k}},
	\end{align*}
	and the Gr\"{o}nwall inequality, Proposition \ref{prop:Gronwall}, implies that
	\[
		\check{\mathcal{E}}_{[k]}(t)
		\leq
		C_k
		(\mathcal{F}_{\infty}^{N+2k+4})^2(1+\mathcal{F}_{\infty}^{N+1})
		\frac{ (\log t)^{1+k}
		}{t^{1+ k}}
		,
	\]
	which completes the proof of the estimate \eqref{eq:fcheckkho2} (after increasing the elements of the sequence $\{\mathfrak{n}(n)\}$ appropriately).
	
\end{proof}

This section is ended with the following proposition, which will be used to recover the bootstrap assumption \eqref{eq:ba2}.

\begin{proposition}[Estimate for $\nabla \check{\phi}_{[0]}$] \label{prop:checkphi0}
	Under the assumptions of Proposition \ref{prop:fcheckkho}, for all $T\leq t \leq T_f$,
	\begin{equation} \label{eq:phizero}
		\sum_{\vert K \vert = 1}^{N+1} \Vert (t \partial_x)^K \check{\phi}_{[0]} \Vert_{L^2_x}
		\lesssim
		\frac{C_{k} \mathcal{F}_{\infty}^{N+2k+4} \log t}{t^{\frac{1}{2}}}.
	\end{equation}
\end{proposition}

\begin{proof}
	Proposition \ref{prop:phicheckho} and Proposition \ref{prop:rhocheckho} imply that
	\begin{multline*}
		\sum_{\vert I\vert =1}^{N+1}
		\Vert (t \partial_x)^I (\check{\phi}_{[0]})(t,\cdot) \Vert_{L^2}
		\lesssim
		t^{\frac{1}{2}} 
		\sum_{\vert I \vert \leq N} \Vert L^I \check{f}_{[0]} (t,\cdot,\cdot) \Vert_{L^2_x L^2_p}
		\\
		\leq
		t^{\frac{1}{2}} 
		\sum_{\vert I \vert \leq N}
		\big(
		\Vert L^I \check{f}_{[k]} (t,\cdot,\cdot) \Vert_{L^2_x L^2_p}
		+
		\Vert L^I (f_{[k]} - f_{[0]}) (t,\cdot,\cdot) \Vert_{L^2_x L^2_p}
		\big)
		\leq
		\frac{C_{k} \mathcal{F}_{\infty}^{N+2k+4} \log t}{t^{\frac{1}{2}}},
	\end{multline*}
	where the final inequality follows from Proposition \ref{prop:fcheckkho} and Theorem \ref{thm:approx} (see \eqref{eq:fKestimate2}).
\end{proof}

\subsection{Estimates for the finite problems: the proof Theorem \ref{thm:bootstrap}}
\label{subsec:proofofbootstrap}
The proof of Theorem \ref{thm:bootstrap} can now be given.

\begin{proof}[Proof of Theorem \ref{thm:bootstrap}]
	The estimate \eqref{eq:bathm1} follows from Proposition \ref{prop:fcheckkho}.  The estimate \eqref{eq:bathm2} follows from Proposition \ref{prop:rhocheckho} and the estimate \eqref{eq:bathm1}, and the estimate \eqref{eq:bathm3} follows from Proposition \ref{prop:phicheckho} and the estimate \eqref{eq:bathm2}.
	
	The bootstrap assumption \eqref{eq:ba} clearly holds with the right hand side replaced by $\frac{1}{2} (\mathcal{F}_{\infty}^{N+2k+4})^2$ if $T_0$ is sufficiently large.  For the improvement of \eqref{eq:ba2}, note that Proposition \ref{prop:checkphi0} and the Sobolev inequality of Remark \ref{rmk:Sobolev} imply that, for all $T \leq t \leq T_f$,
	\begin{equation} \label{eq:phicheck0}
		\sup_{x\in \mathbb{R}^3}
		\vert \nabla \check{\phi}_{[0]}(t,x) \vert
		\leq
		\frac{C}{t^{\frac{5}{2}}}
		\sum_{\vert I \vert = 1}^{N+1} \Vert (t \partial_x)^I \check{\phi}_{[0]} \Vert_{L^2_x}
		\leq
		\frac{C_{k} \mathcal{F}_{\infty}^{N+2k+4} \log t}{t^3},
	\end{equation}
	and so, if $T_0$ is sufficiently large, \eqref{eq:ba2} holds with the right hand side replaced by $\frac{\mathcal{F}_{\infty}^{N+2k+4}}{2t^{5/2}}$.
	
	Finally, the property \eqref{eq:suppoffcheck} follows from Proposition \ref{prop:suppf}.
\end{proof}

\section{The proof of the main results}
\label{section:logicofproof}

In this section the proofs of Theorem \ref{thm:backwardsproblem}, Theorem \ref{thm:backwardsproblem3}, and Theorem \ref{thm:backwardsproblem4} are given.  The following (non-optimal) local well posedness theorem will be used.

\begin{theorem}[Local existence for Vlasov--Poisson] \label{thm:localexistence}
	There is $k$ such that, for any $t_0\in \mathbb{R}$ and any smooth compactly supported function $f_0 \colon \mathbb{R}^3_x \times \mathbb{R}^3_p \to [0,\infty)$, there exists $T=T(\Vert f_0 \Vert_{H^k_xH^k_p}) > 0$ and a unique smooth solution $(f, \varrho, \phi)$ of the Vlasov--Poisson system \eqref{eq:VP1}--\eqref{eq:VP2} on $(t_0-T,t_0+T) \times \mathbb{R}^3_x \times \mathbb{R}^3_p$ such that
	\[
		f(t_0,x,p) = f_0(x,p),
	\]
	for all $(x,p) \in \mathbb{R}^3 \times \mathbb{R}^3$.
\end{theorem}

Let $T_0 \gg 1$ be sufficiently large so as to satisfy the assumptions of Theorem \ref{thm:approx} and Theorem \ref{thm:bootstrap}.   Before the proofs of Theorem \ref{thm:backwardsproblem}, Theorem \ref{thm:backwardsproblem3}, and Theorem \ref{thm:backwardsproblem4} are given, Theorem \ref{thm:bootstrap} is used to establish, for each $T_f>T_0$, the existence of a solution of Vlasov--Poisson on the interval $[T_0,T_f]$ with trivial remainder $\check{f}_{[k]}$ at $t=T_f$.

\begin{theorem}[Existence of solutions of the finite problems on {$[T_0,T_f]$}] \label{thm:finiteprob}
	Let $k>k_*$ be as in Theorem \ref{thm:bootstrap}.  For any $T_f > T_0$, there exists a solution $f\colon [T_0,T_f] \times \mathbb{R}^3\times \mathbb{R}^3 \to [0,\infty)$ of Vlasov--Poisson \eqref{eq:VP1}--\eqref{eq:VP2} such that
	\begin{equation} \label{eq:section6finalcond}
		f\vert_{t=T_f} = f_{[k]}(T_f,\cdot,\cdot),
	\end{equation}
	which satisfies the estimates \eqref{eq:bathm1}--\eqref{eq:bathm3} and the support property \eqref{eq:suppoffcheck} for all $t\in [T_0,T_f]$.
\end{theorem}

\begin{proof}
	Let $\mathcal{B} \subset [T_0,T_f]$ be the set of times $T\in [T_0,T_f]$ such that a solution $f\colon [T,T_f] \times \mathbb{R}^3\times \mathbb{R}^3 \to [0,\infty)$ of Vlasov--Poisson \eqref{eq:VP1}--\eqref{eq:VP2} exists on $[T,T_f]$ satisfying the final condition \eqref{eq:section6finalcond} and the estimates \eqref{eq:ba}--\eqref{eq:ba2}.
	
	The set $\mathcal{B}$ is a non-empty, closed subset of $[T_0,T_f]$, by Theorem \ref{thm:localexistence} and the fact that $\check{f}_{[k]}(T_f,\cdot,\cdot) \equiv 0$ and
	\[
		\sup_{x\in \mathbb{R}^3}
		\vert \nabla \check{\phi}_{[0]}(T_f,x) \vert
		\leq
		\frac{C_{k} \mathcal{F}_{\infty}^{N+2k+4} \log T_f}{(T_f)^3}
		\leq
		\frac{\mathcal{F}_{\infty}^{N+2k+4}}{2(T_f)^{\frac{5}{2}}},
	\]
	for $T_0$ as in Theorem \ref{thm:bootstrap} (see, for example, the estimate \eqref{eq:phicheck0}).  If $T\in \mathcal{B}$ satisfies $T>T_0$, Theorem \ref{thm:localexistence} implies that there is $\delta>0$ such that the solution extends to the interval $[T-\delta,T_f]$.  Theorem \ref{thm:bootstrap} implies that the inequalities \eqref{eq:ba}--\eqref{eq:ba2} in fact hold with better constants and thus, by continuity, continue to hold for all $t\in [T - \delta,T_f]$ if $\delta$ is sufficiently small.  The set $\mathcal{B}$ is therefore open, and hence equal to $[T_0,T_f]$, and the proof follows from Theorem \ref{thm:bootstrap}.
\end{proof}

The proofs of Theorem \ref{thm:backwardsproblem} and Theorem \ref{thm:backwardsproblem3} can now be given.  It is convenient to give both together.

\begin{proof}[Proof of Theorem \ref{thm:backwardsproblem} and Theorem \ref{thm:backwardsproblem3}]
	The proof of Theorem \ref{thm:backwardsproblem} and Theorem \ref{thm:backwardsproblem3} is divided into several steps.  First, the main part of the proof --- namely the existence of the solution $(f,\varrho,\phi)$ --- is given.  Then properties of the solution, including the expansions of Theorem \ref{thm:backwardsproblem3} and the support property \eqref{eq:suppfstatement}, are established.  Finally, the part of Theorem \ref{thm:backwardsproblem} concerning the fact that the solution attains the scattering data $f_{\infty}$ --- in the sense that the estimate \eqref{eq:dataattained} holds --- is shown.
	
	\noindent \textbf{The existence of the solution:}
	Consider some $T_f'>T_f >T_0$.  By Theorem \ref{thm:finiteprob}, there exist corresponding solutions
	\[
		f^{(T_f')} \colon [T_0,T_f'] \times \mathbb{R}^3 \times \mathbb{R}^3 \to [0,\infty),
		\qquad
		f^{(T_f)} \colon [T_0,T_f] \times \mathbb{R}^3 \times \mathbb{R}^3 \to [0,\infty),
	\]
	of the Vlasov--Poisson system, which each satisfy the estimates \eqref{eq:bathm1}--\eqref{eq:bathm3} and the support property \eqref{eq:suppoffcheck}.  The differences
	\[
		\widetilde{f} := f^{(T_f')} - f^{(T_f)},
		\qquad
		\widetilde{\phi} := \phi^{(T_f')} - \phi^{(T_f)},
		\qquad
		\widetilde{\varrho} := \varrho^{(T_f')} - \varrho^{(T_f)},
	\]
	satisfy the system
	\begin{equation} \label{eq:differences}
		\gs_{\phi'} \widetilde{f} = \partial_{x^i}\widetilde{\phi} \, \partial_{p^i} f^{(T_f)},
	\qquad
		\Delta_{\mathbb{R}^3} \widetilde{\phi} = \widetilde{\varrho},
	\end{equation}
	on $[T_0,T_f]$, where $\phi' = \phi^{(T_f')}$, along with the final condition
	\begin{equation} \label{eq:ftildefinal}
		\widetilde{f}(T_f,x,p) = \check{f}_{[k]}^{(T_f')}(T_f,x,p),
	\end{equation}
	for all $x,p\in \mathbb{R}^3$.
	
	The proof proceeds by estimating this difference on $[T_0,T_f]$, in a manner similar to the proof of Theorem \ref{thm:bootstrap}.  Indeed, using the latter of \eqref{eq:differences} and the fact that $\widetilde{f}$ satisfies the support property \eqref{eq:suppoffcheck}, minor adaptations of Proposition \ref{prop:phicheckho} and Proposition \ref{prop:rhocheckho} give, for any multi-index $I$, and any $T_0 \leq t \leq T_f$,
	\begin{equation} \label{eq:phitilderhotilde}
		\Vert \nabla (t\partial_x)^I \widetilde{\phi} (t,\cdot) \Vert_{L^2}
		\lesssim
		t \Vert (t\partial_x)^I \widetilde{\varrho} (t,\cdot) \Vert_{L^2}
		\lesssim
		t^{-\frac{1}{2}} \Vert L^I \widetilde{f} (t,\cdot,\cdot) \Vert_{L^2_x L^2_p}.
	\end{equation}
	Thus, minor adaptations of the proofs of Proposition \ref{prop:systemfcheck} and Proposition \ref{prop:fcheckkcommuted}, using the fact that $f^{(T_f')}$ and $f^{(T_f)}$ satisfy the estimates \eqref{eq:bathm1}--\eqref{eq:bathm3}, give
	\begin{align*}
		\sum_{\vert I \vert + \vert J \vert \leq N} \!\!
		\Vert \gs_{\phi} (L^I (t^{-1}\partial_p)^J \widetilde{f}) \Vert_{L^2_x L^2_p}
		\lesssim
		\
		&
		\frac{\mathcal{F}_{\infty}^{N+1}}{t^{\frac{3}{2}}}
		\sum_{\vert I\vert =1}^{N+1}
		\Vert (t \partial_x)^I \widetilde{\phi}(t,\cdot) \Vert_{L^2}
		+
		\frac{\mathcal{F}_{\infty}^{N+3} (\log t)^2}{t^2}  \!\!
		\sum_{\vert I \vert + \vert J \vert \leq N}  \!\!
		\Vert L^{I} (t^{-1}\partial_p)^{J} \widetilde{f} \Vert_{L^2_x L^2_p}
		\\
		\lesssim
		\
		&
		\frac{\mathcal{F}_{\infty}^{N+1}}{t}
		\sum_{\vert I \vert + \vert J \vert \leq N}
		\Vert L^{I} (t^{-1}\partial_p)^{J} \widetilde{f} \Vert_{L^2_x L^2_p}
		.
	\end{align*}
	Thus, using the estimate \eqref{eq:bathm1} for the final condition \eqref{eq:ftildefinal}, Proposition \ref{prop:generalVlasovestimates} gives that
	\[
		\widetilde{\mathcal{E}}(t) : = \sum_{\vert I \vert + \vert J \vert \leq N}
		\Vert L^{I} (t^{-1}\partial_p)^{J} \widetilde{f} (t,\cdot,\cdot) \Vert_{L^2_x L^2_p}
	\]
	satisfies, for all $T_0 \leq t \leq T_f$,
	\[
		\widetilde{\mathcal{E}}(t)
		\leq
		C_k
		(\mathcal{F}_{\infty}^{N+2k+4})^2
		\frac{(\log T_f)^{1+k}}{(T_f)^{1 + k}}
		+
		\int_t^{T_f}
		\frac{C\mathcal{F}_{\infty}^{N+1}}{s}
		\widetilde{\mathcal{E}}(s)
		ds.
	\]
	The Gr\"{o}nwall inequality, Proposition \ref{prop:Gronwall}, with 
	\[
		a(t)
		=
		\frac{C\mathcal{F}_{\infty}^{N+1}}{s}
		,
		\qquad
		b(t) \equiv C_k
		(\mathcal{F}_{\infty}^{N+2k+4})^2
		\frac{(\log T_f)^{1+k}}{(T_f)^{1 + k}}
	\]
	then gives
	\begin{align*}
		\widetilde{\mathcal{E}}(t)
		\leq
		C_k
		(\mathcal{F}_{\infty}^{N+2k+4})^2
		\frac{(\log T_f)^{1+k}}{(T_f)^{1 + k}}
		\bigg(
		1
		+
		\frac{C\mathcal{F}_{\infty}^{N+1}}{t^{C\mathcal{F}_{\infty}^{N+1}}}
		\int_t^{T_f}
		s^{C\mathcal{F}_{\infty}^{N+1} - 1}
		ds
		\bigg).
	\end{align*}
	Thus, provided $k \geq C\mathcal{F}_{\infty}^{N+1} + 1$,
	\begin{align*}
		\sup_{T_0 \leq t \leq T_f}
		\sum_{\vert I \vert + \vert J \vert \leq N}
		\Vert L^{I} (t^{-1}\partial_p)^{J} (f^{(T_f')} - f^{(T_f)}) (t,\cdot,\cdot) \Vert_{L^2_x L^2_p}
		\leq
		\frac{C_k
		(\mathcal{F}_{\infty}^{N+2k+4})^2
		}{T_f}.
	\end{align*}
	Returning to \eqref{eq:phitilderhotilde}, one has similar estimates for $\phi^{(T_f')} - \phi^{(T_f)}$ and $\varrho^{(T_f')} - \varrho^{(T_f)}$.  It follows that, for any increasing sequence $T_f^n \to \infty$, the sequence $\{ (f^{(T_f^n)}, \phi^{(T_f^n)}, \varrho^{(T_f^n)})\}$ is Cauchy on their common domains of the elements, in the above norms, and thus converges to a unique limit $(f,\phi, \varrho)$ (independent of the sequence $\{T_f^n\}$ chosen) which moreover satisfies the estimates \eqref{eq:bathm1}--\eqref{eq:bathm3} and the support property \eqref{eq:suppoffcheck} on $[T_0,\infty)$ and solves the Vlasov--Poisson system \eqref{eq:VP1}--\eqref{eq:VP2} on $[T_0,\infty)$.
	
	\noindent \textbf{Properties of the solution:}
	The fact that $f$ satisfies the support property \eqref{eq:suppfstatement} follows from \eqref{eq:fklsupport}, along with the fact that each $f^{(T_f)}$ satisfies \eqref{eq:suppoffcheck}.
	
	Consider now the estimates \eqref{eq:mainest1}--\eqref{eq:mainest3} and recall $k_*$ from Theorem \ref{thm:bootstrap}.  For $K \geq k_*$ the estimates \eqref{eq:mainest1}--\eqref{eq:mainest3} follow immediately from the fact that $(f,\varrho,\phi)$ satisfy \eqref{eq:bathm1}--\eqref{eq:bathm3}.  For $0 \leq K < k_*$, note that
	\[
		\check{f}_{[K]}(t,x,p)
		=
		\check{f}_{[k_*]}(t,x,p)
		-
		\sum_{k=K+1}^{k_*} \sum_{l=0}^{k} \frac{(\log t)^l}{t^k} f_{k,l}\big(x-tp + \log t \, \nabla \phi_{\infty} ( p), p \big),
	\]
	and thus \eqref{eq:flkL2estimate} and \eqref{eq:bathm1} imply that
	\[
		\sum_{\vert I \vert + \vert J \vert \leq N}
		\Vert L^I (t^{-1}\partial_p)^J \check{f}_{[K]} (t,\cdot,\cdot) \Vert_{L^2_x L^2_p}
		\leq
		C_{k_*}
		\Big(
		(\mathcal{F}_{\infty}^{N+2k_*+4})^2
		\frac{ (\log t)^{1+k_*}
		}{t^{1+ k_*}}
		+
		\sum_{k=K+1}^{k_*} \frac{(\log t)^k}{t^k} (\mathcal{F}_{\infty}^{N+k})^2
		\Big),
	\] 
	and \eqref{eq:mainest1} follows if $T_0$ is suitably large.  Similarly for \eqref{eq:mainest2}--\eqref{eq:mainest3}.
	
	The $L^{\infty}$ estimates \eqref{eq:mainest4} follow from the $L^2$ estimates \eqref{eq:mainest1} and the Sobolev inequality of Proposition \ref{prop:SobolevL2R6}, and the $L^{\infty}$ estimates \eqref{eq:mainest5}--\eqref{eq:mainest6} follow from the $L^2$ estimates \eqref{eq:mainest2}--\eqref{eq:mainest3} and the Sobolev inequality of Remark \ref{rmk:Sobolev}.
	
	\noindent \textbf{The scattering data is attained:}
	In order to see that the data $f_{\infty}$ is attained --- in the sense that the estimate \eqref{eq:dataattained} holds --- first note that, for
	\[
		y(t,x,p) = x - tp + \log t \nabla \phi_{\infty}(p),
		\qquad
		\tilde{y}(t,x,p) = x + tp - \log t \nabla \phi_{\infty}(p),
	\]
	one has
	\[
		\int\int \vert f(t,\tilde{y}(t,x,p),p) - f_{\infty}(x,p) \vert^2 dp dx
		=
		\int\int \vert f(t,x,p) - f_{\infty}(y(t,x,p),p) \vert^2 dp dx,
	\]
	and so the zeroth order part of the estimate \eqref{eq:dataattained} follows from setting $K=0$ in \eqref{eq:mainest1}.  For the terms involving first order derivatives, recall first Remark \ref{rmk:functionsofy} (considering, in particular, \eqref{eq:functionsofy1} and \eqref{eq:functionsofy2} with $h =f_{\infty}$) and note also that 
	\begin{align*}
		\partial_{x^i} \Big( f(t,\tilde{y}(t,x,p),p) \Big)
		=
		\
		&
		(\partial_{x^i} f) (t,\tilde{y}(t,x,p),p),
	\\
		\partial_{p^i} \Big( f(t,\tilde{y}(t,x,p),p) \Big)
		=
		\
		&
		\Big( \delta_{ij} - \frac{\log t}{t} \partial^2_{i,j} \phi_{\infty}(p) \Big)(L_i f) (t,\tilde{y}(t,x,p),p)
		\\
		&
		+
		\frac{(\log t)^2}{t} \partial_{i,j}^2 \phi_{\infty}(p) \partial_{j,k}^2 \phi_{\infty}(p) (t^{-1}\partial_{p^k} f) (t,\tilde{y}(t,x,p),p),
	\end{align*}
	and
	\begin{align*}
		(\partial_{x^i}f_{\infty})(y(t,x,p),p)
		=
		\
		&
		-
		(t^{-1} \partial_{p^i}) \big( f_{\infty}(y(t,x,p),p) \big)
		\\
		&
		+
		\frac{\log t}{t} \partial_k \partial_i \phi_{\infty}(p) (\partial_{x^k}f_{\infty})(y(t,x,p),p)
		+
		t^{-1} (\partial_{p^i}f_{\infty})(y(t,x,p),p)
		\\
		(\partial_{p^i}f_{\infty})(y(t,x,p),p)
		=
		\
		&
		L_i \big( f_{\infty}(y(t,x,p),p) \big)
		\\
		&
		-
		\frac{(\log t)^2}{t} \partial_k \partial_i \phi_{\infty}(p) (\partial_{x^k}f_{\infty})(y(t,x,p),p)
		-
		\frac{\log t}{t} \partial_k \partial_i \phi_{\infty}(p) (\partial_{p^k}f_{\infty})(y(t,x,p),p)
		.
	\end{align*}
	It follows that
	\begin{align*}
		&
		(\partial_{x^i}f)(t,x,p) - (\partial_{x^i}f_{\infty})(y(t,x,p),p)
		=
		-
		(t^{-1} \partial_{p^i}) \big( f(t,x,p) - f_{\infty}(y(t,x,p),p) \big)
		\\
		&
		+
		\frac{1}{t} \big( \partial_{x^i} + t \partial_{p^i} \big) f(t,x,p)
		-
		\frac{\log t}{t} \partial_k \partial_i \phi_{\infty}(p) (\partial_{x^k}f_{\infty})(y(t,x,p),p)
		-
		t^{-1} (\partial_{p^i}f_{\infty})(y(t,x,p),p).
	\end{align*}
	Hence
	\begin{align*}
		&
		\int\int \big\vert \partial_{x^i} \big( f(t,\tilde{y}(t,x,p),p) - f_{\infty}(x,p) \big) \big\vert^2 dp dx
		=
		\int\int \big\vert (\partial_{x^i} f)(t,x,p) - (\partial_{x^i} f_{\infty}) (y(t,x,p),p) \big\vert^2  dp dx
		\\
		&
		\qquad \qquad
		\lesssim
		\int\int \big\vert t^{-1} \partial_{p^i} \big( f(t,x,p) -  f_{\infty} (y(t,x,p),p) \big) \big\vert^2 dp dx
		+
		\frac{\mathcal{F}_{\infty}^3 \log t}{t}
		\lesssim
		\frac{\mathcal{F}_{\infty}^{N+4}\log t}{t}
		,
	\end{align*}
	and similarly,
	\begin{align*}
		&
		\int\int \big\vert \partial_{p^i} \big( f(t,\tilde{y}(t,x,p),p) - f_{\infty}(x,p) \big) \big\vert^2 dp dx
		\\
		&
		\qquad \qquad
		\lesssim
		\int\int \big\vert L_i \big( f(t,x,p) -  f_{\infty} (y(t,x,p),p) \big) \big\vert^2 dp dx
		+
		\frac{\mathcal{F}_{\infty}^3 (\log t)^2}{t}
		\lesssim
		\frac{\mathcal{F}_{\infty}^{N+4} (\log t)^2}{t}
		,
	\end{align*}
	by \eqref{eq:mainest1} with $K=0$.  The part of \eqref{eq:dataattained} concerning higher order derivatives follows similarly.
	The $L^{\infty}$ estimate \eqref{eq:dataattained2} follows from \eqref{eq:dataattained} and the Sobolev inequality of Proposition \ref{prop:SobolevL2R6}.
	\end{proof}
	
	Finally, the proof of Theorem \ref{thm:backwardsproblem4} is given.
	
	\begin{proof}[Proof of Theorem \ref{thm:backwardsproblem4}]
	Consider the solution $(f,\varrho,\phi)$ of Theorem \ref{thm:backwardsproblem}, and
	suppose that $(f',\varrho',\phi')$ is another solution which satisfies \eqref{eq:uniquenesscondition} for some $K=K(\mathcal{F}_{\infty}^{N+1})$ and for all $t$ sufficiently large, and $f'$ satisfies the support property \eqref{eq:suppfstatement}.  It follows that the difference of the two solutions satisfies the system
	\[
		\gs_{\phi'} (f-f') = \partial_{x^i}(\phi - \phi ') \, \partial_{p^i} f,
	\qquad
		\Delta_{\mathbb{R}^3} (\phi - \phi ') = \varrho - \varrho'.
	\]
	Since
	\[
		\sup_{x\in\mathbb{R}^3} \Vert \partial_{p^i} f (t,x,\cdot) \Vert_{L^2_p}
		\leq
		\frac{C\mathcal{F}_{\infty}^{N}}{t^{\frac{1}{2}}},
		\qquad
		\Vert \nabla (\phi - \phi ') (t,\cdot) \Vert_{L^2}
		\lesssim
		t \Vert (\varrho - \varrho ') (t,\cdot) \Vert_{L^2}
		\lesssim
		t^{-\frac{1}{2}} \Vert (f-f') (t,\cdot,\cdot) \Vert_{L^2_x L^2_p}.
	\]
	where the latter follows from the support property \eqref{eq:suppfstatement} (see Proposition \ref{prop:phicheckho} and Proposition \ref{prop:rhocheckho}), it follows that
	\[
		\Vert \gs_{\phi'} (f-f')(t,\cdot,\cdot)\Vert_{L^2_x L^2_p}
		\leq
		\frac{C\mathcal{F}_{\infty}^{N}}{t} \Vert (f-f')(t,\cdot,\cdot)\Vert_{L^2_x L^2_p}.
	\]
	Proposition \ref{prop:generalVlasovestimates} and the Gr\"{o}nwall inequality, Proposition \ref{prop:Gronwall}, with $a(t) = t^{-1}C\mathcal{F}_{\infty}^{N}$, $b(t) \equiv \Vert (f-f')(T,\cdot,\cdot)\Vert_{L^2_x L^2_p}$ then imply that, for any $T$ large and $t\leq T$,
	\begin{align*}
		\Vert (f-f')(t,\cdot,\cdot)\Vert_{L^2_x L^2_p}
		&
		\leq
		\Vert (f-f')(T,\cdot,\cdot)\Vert_{L^2_x L^2_p}
		\Big(
		1
		+
		\frac{C\mathcal{F}_{\infty}^{N}}{t^{C\mathcal{F}_{\infty}^{N}}} \int_t^{T} s^{C \mathcal{F}_{\infty}^{N} - 1} ds
		\Big)
		\\
		&
		\leq
		C_K
		\Big(
		1
		+
		(\mathcal{F}_{\infty}^{N+2K+4})^2
		\Big)
		\frac{ (\log T)^{1+K}
		}{T^{1+ K}} 
		T^{C\mathcal{F}_{\infty}^{N}}.		
	\end{align*}
	where the latter follows from \eqref{eq:mainest1} and the assumption \eqref{eq:uniquenesscondition}.  If $K \geq C \mathcal{F}_{\infty}^{N}$, letting $T\to \infty$, it follows that the left hand side vanishes for all $t$, and thus $(f',\varrho ',\phi ') \equiv (f,\varrho,\phi)$.
\end{proof}

\bibliography{scatteringVP}{}
\bibliographystyle{plain}

\end{document}